\documentclass[12pt]{article}

\usepackage[dvipsnames]{xcolor}
\usepackage{xspace}
\usepackage{colortbl}
\usepackage[utf8]{inputenc} 
\usepackage[T1]{fontenc}    
\usepackage{url}            
\usepackage{booktabs}       
\usepackage{amsfonts}       
\usepackage{nicefrac}       
\usepackage{microtype}      
\usepackage{xcolor}         
\usepackage{tabularx}
\usepackage{tikz}
\usepackage{arxiv_macros}
\usetikzlibrary{arrows.meta}
\definecolor{lightgreen}{RGB}{220, 245, 220}
\definecolor{PineGreen}{rgb}{0.0,0.47,0.44} 
\newcommand{\algname}[1]{{\color{PineGreen}\sf#1}\xspace}
\newcommand{\lalgname}[2]{%
  \begingroup
  \hypersetup{linkcolor=PineGreen}%
  \hyperref[#2]{\color{PineGreen}\sf#1}%
  \endgroup
  \xspace
}
\newcommand\inprd[2]{\left\langle #1,#2\right\rangle}
\newcommand{\p}[1]{\left( #1 \right)}
\newcommand{\set}[1]{\left\{#1\right\}}

\newcommand{\ugamma}[0]{\underline{\gamma}}
\newcommand{\overF}[0]{\overline{F}}
\newcommand{\xhatk}[0]{\hat{x}_k}
\newcommand{\supg}[0]{\sup_{y\in B}}
\newcommand{\xhat}[0]{\hat{x}}
\newcommand{\ifaone}[0]{\text{ if $\alpha = 1$ }}
\newcommand{\ifasmall}[0]{\text{ if $\alpha \in (0,1)$ }}
\newcommand{\ceil}[1]{\left\lceil #1 \right\rceil}
\newcommand{\floor}[1]{\left\lfloor #1 \right\rfloor}

\newcommand{\leqfrom}[1]{\overset{(#1)}{\leq}}
\newcommand{\geqfrom}[1]{\overset{(#1)}{\geq}}
\newcommand{\eqfrom}[1]{\overset{(#1)}{=}}
\newcommand{\gap}[1]{\texttt{Gap}_{F,B}\left( #1\right)}

\newcommand{\xbark}[0]{\bar{x}_K}
\newcommand{\half}[0]{\frac{1}{2}}
\newcommand{\mink}[0]{\min_{0\leq k\leq K}}
\newcommand{\eps}[0]{\varepsilon}

\newcommand{\eg}[0]{\lalgname{EG}{EG_update}}
\newcommand{\popov}[0]{\lalgname{PEG}{PEG_update}}
\newcommand{\adaeg}[0]{\algname{EG-LS}}
\newcommand{\adapeg}[0]{\algname{PEG-LS}}
\newcommand{\pegls}[0]{\algname{(P)EG-LS}}
\newcommand{\peg}[0]{\lalgname{PEG}{PEG_update}}
\newcommand{\bigo}[1]{\mathcal{O}\left(#1\right)}

\newcommand{\B}[0]{\mathbf{B}}

\newcommand{\J}[0]{\mathbf{J}}
\newcommand{\gapp}[0]{\texttt{Gap}_P}
\newcommand{\shrink}[2]{\resizebox{#1\linewidth}{!}{$#2$}}
\newcommand{\lamz}[0]{\bar{\lambda}_N^0}
\newcommand{\lamo}[0]{\bar{\lambda}_N^1}

\title{\textbf{Extragradient and Past Extragradient Methods \\for $(L_0,L_1)$-Lipschitz Variational Inequalities}}

\begin{document}

\author{ Ezra Greenberg and  Nicolas Loizou }

\author{
    Ezra Greenberg\\
    \texttt{egreen70@jhu.edu} 
    \and 
    Nicolas Loizou\\
    \texttt{nloizou@jhu.edu}
}

\date{
    \vspace{1em}
    Mathematical Institute for Data Science (MINDS),\\
    Department of Applied Mathematics and Statistics,\\ Johns Hopkins University\\
    \vspace{0.5em}
}

\maketitle

\begin{abstract}
      Variational inequalities (VIs) capture a broad class of optimization problems central to modern machine learning, including min-max optimization and equilibrium computation in multiplayer games. Traditionally, algorithms for solving VIs, such as Extragradient (EG) and Past Extragradient (PEG), have been analyzed under the standard $L$-Lipschitz assumption. However, this condition is often overly restrictive, and recent works have instead turned to the more relaxed $(L_0,L_1)$-Lipschitz setting. Existing guarantees for $(L_0,L_1)$-Lipschitz VIs require a quasi-sharpness condition [Vankov et al., 2024], while monotone guarantees for EG cover only unconstrained root finding [Choudhury and Loizou, 2025].
   In this work, we provide the first convergence guarantees for EG and PEG for monotone $(L_0,L_1)$-Lipschitz VIs. 
    We also improve upon prior convergence guarantees for strongly monotone VIs by supporting larger stepsizes in the constrained setting and eliminating exponential dependence on the initial distance from the solution for convergence rates in the unconstrained setting. Lastly, we propose line-search variants for when Lipschitz parameters are unknown.
   Numerical experiments complement our theoretical results and demonstrate the practical efficacy of our methods. 
 \end{abstract}
 \newpage
 \tableofcontents

\newpage
\section{Introduction}\label{sec: introduction}
Variational inequalities (VIs) have increasingly attracted interest for their ability to model diverse machine learning applications, such as reinforcement learning \citep{sokota2023a,sidahmed2025addressing}, static mean-field games \citep{yardim2025}, and generative adversarial networks \citep{goodfellow14,gidel2018a}. The VI problem is defined as
\vspace{-2mm}
\begin{equation}
    \text{find $x^*\in U$ s.t. $\inprd{F(x^*)}{x-x^*}\geq 0$ $\forall x\in U$},\label{VIP Solution}
\end{equation}
where $U\subset \R^d$ is closed and convex and $F:U\rightarrow \R^d$ is a continuous operator. In the unconstrained setting, the VI problem simplifies to the root-finding problem of finding $x^*$ such that $F(x^*)=0$~\citep{sgda_consensus21, pmlr-v151-gorbunov22b}.  In this work, we assume that the solution set of the VI problem \eqref{VIP Solution} is nonempty ($U^* \neq \emptyset$).

Problem~\eqref{VIP Solution} covers a wide array of optimization tasks based on the choice of operator $F$~\cite{facchinei_finite-dimensional_2003}. For example, taking $F(x)=\nabla f(x)$ for a continuously
differentiable function $f$ yields the first-order necessary
optimality conditions for the constrained \textbf{minimization}
problem $\min_{x\in U}f(x)$. When $f$ is convex, the VI solutions
are exactly the global minimizers. 
For \textbf{min-max optimization}, let
$\mathcal{L}:U_1\times U_2\to\R$ be continuously differentiable,
convex in $w_1$, and concave in $w_2$, with $U_1$ and $U_2$
nonempty, closed, and convex. On $U=U_1\times U_2$, solving the VI
with the saddle-point operator $F=\begin{pmatrix}
    \nabla_{w_1}\mathcal{L}^\top, -\nabla_{w_2}\mathcal {L}^\top
\end{pmatrix}^\top$ characterizes the saddle points of $\mathcal{L}$. 
Finally, consider an $n$-player game in which player $i$
chooses a strategy $x_i\in U_i$, where each $U_i$ is
nonempty, closed, and convex. In this setting, each player minimizes a continuously differentiable
loss $\mathcal{L}_i(x_i,x_{-i})$ that is convex in $x_i$
for every fixed $x_{-i}$.
The game's \textbf{Nash equilibria} are exactly the solutions
of the VI with the joint-gradient operator $F=\begin{pmatrix}
    \nabla_{x_1} \mathcal{L}_1(x_1,x_{-1})^\top, \cdots, \nabla_{x_n} \mathcal{L}_n(x_n,x_{-n})^\top
\end{pmatrix}^\top$.

\textbf{Problem Classes and Algorithms.} In this work we focus on two main classes of VIs: \textit{monotone} and \textit{strongly monotone}.
\begin{assumption}
    $F:U\rightarrow\R^d$ is $\mu$-strongly monotone if there exists some $\mu >0$ such that
    \begin{equation}
        \inprd{F(x)-F(y)}{x-y}\geq \mu\|x-y\|^2\hspace{0.5cm}\forall x,y\in U
        \label{muSC}
    \end{equation}
    When $\mu=0$, we call $F$ monotone.
\end{assumption}
A $\mu$-strongly monotone operator has a unique solution, while a monotone operator has a convex solution set $U^*$ \citep{facchinei_finite-dimensional_2003}.
Notably, even for simple monotone problems, the analogue to the gradient descent method for VIs, the Projection Method, $x_{k+1} = P_U(x_k-\gamma_k F(x_k))$, provably diverges~\citep{gidel2018a}. Instead, one must turn to the Extragradient Method~\citep{Korpelevich76}:
\vspace{-2mm}
\begin{equation}
\label{EG_update}
\begin{aligned}
    \xhatk &= P_U(x_k-\gamma_kF(x_k)) \notag\\
    x_{k+1} &= P_U(x_k-\gamma_k F(\xhatk)) 
\end{aligned}\tag{\algname{EG}}
\end{equation}
where $P_U$ is a projection to the constraint set. 
More recently, convergence guarantees have been established
for extragradient-type methods applied to structured
non-monotone VIs~\citep{diakonikolas21a_weak_mvi_eg,
choudhury2023singlecall}, as well as for
stochastic~\citep{pmlr-v108-mishchenko20a,
pmlr-v151-gorbunov22b, yoon2026same} and
distributed~\citep{beznosikov_decentralized_seg22,
beznosikov_distributed_2025} variants.

Closely related to \eg is Past Extragradient \citep{popov1980,gidel2018a}, which reduces the number of operator evaluations at each iteration by reusing computations from the previous iterate:
\vspace{-4mm}
\begin{equation}
\label{PEG_update}
    \begin{aligned}
    \xhatk&=P_U(x_k-\gamma_kF(\xhat_{k-1}))\\
    x_{k+1} &= P_U(x_k-\gamma_k F(\xhatk)) 
\end{aligned}\tag{\algname{PEG}}
\end{equation}
where $\xhat_{-1}=x_0\in U$. 

Traditionally, both \eg and \popov have been analyzed under the $L$-Lipschitz assumption, 
\vspace{-2mm}
\begin{equation}
    \|F(x)-F(y)\|\leq L\|x-y\|. \label{eqn:l lipschitz}
\end{equation}
More recently however, research in relaxing this assumption has greatly accelerated, catalyzed by the proposal of the $(L_0,L_1)$-smooth assumption in the minimization setting \citep{Zhang2020Why}. 
There, \cite{Zhang2020Why} show that along the trajectories of gradient descent algorithms for training neural networks, the spectral norm of the Hessian $\|\nabla^2f\|$  tends to scale with the norm of the gradient $\|\nabla f\|$, leading to the $(L_0,L_1)$-smoothness assumption, $\|\nabla^2 f\|\leq L_0+L_1\|\nabla f\|$. The use of this assumption provided the first theoretical justification for gradient norms/clipping, where stepsizes are inversely proportional to $\|\nabla f\|$~\cite{Zhang2020Why,gorbunov2025methods}.

For VIs, similar relationships between the Jacobian norm
$\|\J(x)\|$ and the norm $\|F(x)\|$ have been observed along
\eg trajectories in unconstrained root-finding
problems~\citep{choudhury2025extragradient}.
In Appendix~\ref{app: l0l1 trajectory} we present analogous
trajectory plots for \eg and \peg in the constrained setting~\eqref{VIP Solution}.
These observations motivate our study of VIs under the
$\alpha$-symmetric $(L_0,L_1)$-Lipschitz
assumption~\citep{pmlr-v235-vankov24a}.
\begin{assumption}\label{assmpt: l0 l1}
    An operator $F$ is $\alpha$-symmetric $(L_0, L_1)$-Lipschitz if for some $L_0, L_1 \geq 0$ and $\alpha \in (0, 1]$, and for all $x,y\in U$,
    \vspace{-2mm}
    \begin{equation}\label{l0l1 defn}
    \begin{aligned}
        \| F(x) - F(y)\| \leq  
        ( L_0 + L_1 \max_{\theta \in [0, 1]} \left\|F(\theta x + ( 1- \theta) y) \right\|^{\alpha}) \| x- y\|.
    \end{aligned}
    \end{equation}
\end{assumption}
This assumption allows the effective Lipschitz constant to grow with the operator norm. In the case of a differentiable operator $F$, \eqref{l0l1 defn}  is equivalent to $\|\J(x)\|\leq L_0+L_1\|F(x)\|^\alpha$ ~\citep{choudhury2025extragradient}. 
\cite{pmlr-v235-vankov24a} provided the first convergence guarantees for \eg and \peg in the $(L_0,L_1)$-Lipschitz setting for $p$-quasi sharp operators, a class of structured non-monotone operators which include strongly monotone operators as a special case. Notably, $p$-quasi sharpness is a strong enough assumption that the simple projection method can solve all $p$-quasi sharp VIs, despite the method diverging for monotone problems. For root-finding problems, \cite{choudhury2025extragradient} improved on the stepsizes and convergence guarantees of~\cite{pmlr-v235-vankov24a} for \eg in the strongly monotone setting and provided the first monotone convergence guarantees. However, convergence guarantees for \eg and \popov for monotone $(L_0,L_1)$-Lipschitz VIs remain an open question. Additionally, while \cite{choudhury2025extragradient} improved stepsizes and convergence guarantees for \eg compared to \cite{pmlr-v235-vankov24a}, similar results for \peg have not been attempted previously.

\subsection{Main Contributions}

Our contributions are summarized as follows.

   \begin{table*}[t]
\centering
\small
\begin{threeparttable}
\caption{Summary of Convergence guarantees for $\alpha$-symmetric
$(L_0,L_1)$-Lipschitz VIs.}
\label{tab:comparison_of_rates}

\renewcommand{\arraystretch}{1.1}
\setlength{\tabcolsep}{4pt}
\renewcommand{\tabularxcolumn}[1]{m{#1}}

\begin{tabularx}{\textwidth}{
    @{}
    >{\raggedright\arraybackslash}m{2.8cm}
    >{\centering\arraybackslash}m{1.5cm}
    *{3}{>{\centering\arraybackslash}X}
    >{\centering\arraybackslash}m{3.4cm}
    @{}
}
\toprule
\textbf{Work}
& \textbf{Methods}
& \textbf{\shortstack{Strongly\\monotone}}\tnote{1}
& \textbf{Monotone}\tnote{2}
& \textbf{Constrained}
& \textbf{\shortstack{No exponential\\dependence on $D_0$}}
\\
\midrule
\cite{pmlr-v235-vankov24a}
& \eg, \peg
& \cmark
& \ysquare\tnote{3}
& \cmark
& \xmark
\\
\cite{choudhury2025extragradient}
& \eg
& \cmark
& \cmark
& \xmark
& \ysquare\tnote{4}
\\
\rowcolor{lightgreen}
This work
& \eg, \peg
& \cmark
& \cmark
& \cmark
& \cmark
\\
\bottomrule
\end{tabularx}

\begin{tablenotes}[flushleft]
\footnotesize
\setlength{\itemsep}{0pt}
\setlength{\parsep}{0pt}

\item[] $D_0=\|x_0-x^*\|$; the last column concerns
unconstrained VIs with $\alpha=1$.

\item[1] Metric: $\|x_K-x^*\|^2$.

\item[2] Metrics: $\min_k\|F(x_k)\|^2$ (unconstrained);
restricted gap $\texttt{Gap}(\bar{x}_K)$ (constrained).

\item[3] Monotone cases require $p$-quasi sharpness:
$\inprd{F(x)}{x-x^*}\ge
\mu\,\operatorname{dist}^p(x,U^*)$.
This does not cover all monotone operators.
Metric: $\operatorname{dist}^2(x,U^*)$.

\item[4] No exponential dependence for the monotone
(sublinear) guarantee only.
\end{tablenotes}

\end{threeparttable}
\end{table*}
       \textbf{Convergence guarantees for monotone $(L_0,L_1)$-Lipschitz VIs:} We provide the first convergence guarantees for \eg and \popov
on monotone $\alpha$-symmetric
$(L_0,L_1)$-Lipschitz VIs
(Theorems~\ref{thm: EG monotone}
and~\ref{thm: popov monotone}).
       For both methods we show $\bigo{1/K}$ convergence of the restricted gap function. For \popov, our results also yield the first convergence
guarantees for monotone $(L_0,L_1)$-Lipschitz root-finding problems (unconstrained setting).

      \textbf{Convergence guarantees for strongly monotone $(L_0,L_1)$-Lipschitz VIs:}
We establish linear convergence of \eg and \popov
for strongly monotone $\alpha$-symmetric
$(L_0,L_1)$-Lipschitz VIs
(Theorems~\ref{thm: EG sm} and~\ref{thm: popov sm}),
using stepsizes that do not require knowledge of
the strong monotonicity parameter $\mu$.
For unconstrained root-finding problems with $\alpha=1$,
we further derive iteration complexity bounds for both
methods without exponential dependence on the initial
distance to the solution
(Corollaries~\ref{cor: eg usm} and~\ref{cor: popov sm}).

\textbf{Line-search variants with unknown Lipschitz parameters:}
For the scenarios that $L_0$ or $L_1$ parameters are unkown or hard to pre-compute we propose \adaeg and \adapeg, which are line-search variants of \eg and \peg for solving $(L_0,L_1)$-Lipschitz VIs without requiring knowledge
of $L_0$ or $L_1$ (Algorithm~\ref{alg:AdaptiveAlgo}).
For both methods, we establish $\bigo{1/K}$ convergence
for monotone VIs and linear convergence for strongly
monotone VIs, without requiring knowledge of $\mu$.
We prove that the line-search always terminates successfully and that the total number of while-loop iterations over the entire run is finite.
     
     \textbf{Numerical Experiments}: We provide numerical experiments validating our theoretical guarantees. We compare our stepsizes with existing alternatives, demonstrating the strong performance of our approach.
\vspace{-4mm}
\section{On $(L_0,L_1)$-Lipschitz: Setting and Stepsizes}\label{sec:stepsizes}

Let us provide additional details on the $(L_0,L_1)$-Lipschitz assumption~\eqref{l0l1 defn} and its use for the design of novel stepsizes for \ref{EG_update} and \ref{PEG_update}.

Essential for our convergence analysis is the following result from \cite{pmlr-v202-chen23ar,choudhury2025extragradient}.
\begin{proposition}[\cite{choudhury2025extragradient}]\label{lem: l0l1 smooth}
    Suppose $F$ is $\alpha$-symmetric $(L_0,L_1)$-Lipschitz. Then if $\alpha=1$
    \begin{equation}
    \begin{aligned}
         \|F(x)&-F(y)\| \leq (L_0+L_1\|F(y)\|)e^{L_1\|x-y\|}\|x-y\|
        \label{a1 l0l1}
    \end{aligned}
    \end{equation}
    and if $\alpha\in (0,1)$,
    \begin{equation}
    \begin{aligned}
        \|&F(x)-F(y)\| 
        \leq (K_0+K_1\|F(y)\|^\alpha + K_2\|x-y\|^{p_\alpha})\|x-y\|
    \end{aligned}
        \label{a small l0l1}
    \end{equation}
    where $p_\alpha=\nicefrac{\alpha}{1-\alpha}$, $K_0=L_0(2^{\alpha p_\alpha}+1)$, $K_1=L_1 2^{\alpha p_\alpha}3^\alpha$ and $K_2=L_1^{p_\alpha/\alpha}2^{\alpha p_\alpha}3^\alpha(1-\alpha)^{p_\alpha}$.
\end{proposition}
Proposition~\ref{lem: l0l1 smooth} eliminates the maxima from (\ref{l0l1 defn}), providing us with the more tractable upper bounds \eqref{a1 l0l1} and \eqref{a small l0l1}. 
In particular, note that in classical convergence analyses of \eg and \peg in the $L$-Lipschitz case, the stepsize is  set to $\gamma_k\propto\nicefrac{1}{L}$ to achieve the bound 
\vspace{-3mm}
\begin{equation}
    \gamma_k\|F(\xhatk)-F(\tilde{x}_k)\|\leq C\|\xhatk-\tilde{x}_k\|\label{eqn: backtrack cond}
\end{equation}
where $\tilde{x}_k=x_k$ for \eg and $\tilde{x}_k=\xhat_{k-1}$ for \peg, and $C\in\R$ is constant parameter that does not depend on $L$ or $\mu$. 
In the $(L_0,L_1)$-Lipschitz setting, our proposed stepsizes adapt to the bounds in Proposition~\ref{lem: l0l1 smooth} in order to achieve a similar result. 
\subsection{Stepsize selection}

For \eg, following a similar approach to \cite{choudhury2025extragradient} we use:
\vspace{-2mm}
\begin{equation}
    \gamma_{k,\eg}=\begin{cases}
        \frac{\nu}{L_0+L_1\|F(x_k)\|} &\text{ if $\alpha = 1$} \\
        \frac{\nu}{2K_0+(2K_1+2^{1-\alpha}K_2^{1-\alpha})\|F(x_k)\|^\alpha} &\text{ if $\alpha \in (0,1)$}\label{EG Stepsizes}
    \end{cases}
\end{equation}
where $\nu>0$ is a stepsize scaling parameter with an upper bound that depends on $\alpha$ and properties of the operator $F$ (e.g. monotone or strongly monotone). These upper bounds are specified in subsequent theorems and range from $0.17$ to $1$. The form of stepsizes similar to $\frac{A}{B+C\|F(x_k)\|^\alpha}$ is also seen in the minimization setting in \cite{gorbunov2025methods}. 

On the other hand, for \peg, we require the somewhat more complex stepsizes $\gamma_{k,\peg}$. Let $d_k:=\|x_k-\xhat_{k-1}\|$ and for $\alpha\in(0,1)$, define  $s_{k,\alpha}
:=d_k+\frac{\nu}{K_1}
       \|F(\xhat_{k-1})\|^{1-\alpha}$. Then
\vspace{-2mm}
\begin{equation}\label{Popov stepsizes}
\gamma_{k,\peg}:=
\begin{cases}
\displaystyle
\frac{\nu}
     {(L_0+L_1\|F(\xhat_{k-1})\|) e^{L_1d_k}},
& \alpha=1,\\[6pt]
\displaystyle
\frac{\nu}
     {K_0+K_1\|F(\xhat_{k-1})\|^\alpha
      +K_2s_{k,\alpha}^{p_\alpha}},
& \alpha\in (0,1).
\end{cases}
\end{equation}
For $\alpha=1$, this stepsize is similar to that proposed by \cite{pmlr-v235-vankov24a}, but it eliminates the dependence on the strong monotonicity parameter $\mu$.
While the stepsizes may not seem intuitive initially, note that as $\alpha\rightarrow 0$ or $L_1\rightarrow 0$ (equivalent to approaching the standard $L$-Lipschitz setting) both the \eg and \peg stepsizes become constant and proportional to the reciprocal of the effective $L$-Lipschitz constant, as expected (see Appendix~\ref{app:limiting behavior}).  Additionally, as \peg converges to a solution of the VI, resulting in $\|x_k-\xhat_{k-1}\|\rightarrow 0$, the stepsizes for \peg recover the form of $\frac{A}{B+C\|F(x_k)\|^\alpha}$. 

With these stepsizes, we obtain following lemma.
\begin{lemma}\label{lem:lipschitz analogue}
    For \eg and \popov, if $F$ is $\alpha$-symmetric $(L_0, L_1)$-Lipschitz and we choose the stepsize $\gamma_k$ according to equations (\ref{EG Stepsizes}) and (\ref{Popov stepsizes}) respectively, then the algorithm iterates satisfy
    \vspace{-2mm}
    \begin{equation}
    \gamma_k\|F(\xhatk)-F(\tilde{x}_k)\|\leq  C_\nu \|\xhatk-\tilde{x}_k\| \label{l0l1 half step general}
\end{equation}
where $\tilde{x}_k=x_k$ for \eg and $\tilde{x}_k=\xhat_{k-1}$ for \peg, $C_\nu=\nu e^\nu$ if $\alpha=1$, and $C_\nu=\nu$ if $\alpha \in (0,1)$, with $\nu<1$.
\end{lemma}
Lemma~\ref{lem:lipschitz analogue} allows us to perform a similar convergence analysis to the $L$-Lipschitz case. However, due to our use of non-constant stepsizes, our proofs will also require the derivation of stepsize lower bounds.
\subsection{Lower Bounds on Stepsizes}

We now provide these lower bounds for each of our stepsizes, dependent on the algorithm iterates being bounded. Here, bounded iterates is not an assumption, but an algorithmic property proved in the subsequent theorems. The first lower bound for \eg is similar to that obtained as an intermediate step in the proofs of \cite{choudhury2025extragradient} for root-finding problems (unconstrained setting), but it includes an additional dependence on $\|F(x^*)\|$ as we focus on the more complicated VI setting. 
    \begin{lemma}\label{lem:EG step lower bound}
Let $F$ be $\alpha$-symmetric $(L_0,L_1)$-Lipschitz.
Fix $x^*\in U^*$ and write
$D_0=\|x_0-x^*\|$.
Suppose the iterates of \eg with stepsizes
\eqref{EG Stepsizes} satisfy
$\|x_k-x^*\|\le D_0$ for all $k\ge0$.

Define
\vspace{-3mm}
\begin{align*}
M_1
&:= \|F(x^*)\|+D_0(L_0+L_1\|F(x^*)\|)e^{L_1D_0},\\
M_\alpha
&:= \|F(x^*)\|+D_0\bigl(K_0+K_1\|F(x^*)\|^\alpha +K_2D_0^{p_\alpha}\bigr),\\
A_\alpha
&:= 2K_1+2^{1-\alpha}K_2^{1-\alpha}.
\end{align*}
Then $\gamma_{k,\eg}\ge\ugamma$ for all $k\ge0$, where
\begin{equation}\label{EG step lower bound main}
\ugamma :=
\begin{cases}
\displaystyle\frac{\nu}{L_0+L_1M_1},
& \alpha=1,\\[6pt]
\displaystyle\frac{\nu}{2K_0+A_\alpha M_\alpha^\alpha},
& \alpha\in (0,1)
\end{cases}
\end{equation}
\end{lemma}
\vspace{-2mm}
Similarly, for \peg we have

\begin{lemma}\label{lem:popov step lower bound}
Let $F$ be $\alpha$-symmetric $(L_0,L_1)$-Lipschitz.
Fix $x^*\in U^*$ and define $D_k=\|x_k-x^*\|$,
$R_k=D_k^2+\|x_k-\xhat_{k-1}\|^2$, and $F_*=\|F(x^*)\|$

Suppose the iterates of \peg with stepsizes
\eqref{Popov stepsizes} and $C_\nu<1$ satisfy
$R_k\le R_0$ for all $k\ge0$.

For $\alpha=1$, define
\vspace{-2mm}
\begin{align*}
M_{1,\mathrm{PEG}}
&:= \sqrt{2}D_0
    \bigl(L_0+L_1 F_*\bigr) e^{\sqrt{2}L_1D_0}
    + F_*.
\end{align*}
For $\alpha\in (0,1)$, define
\vspace{-2mm}
\begin{align*}
q_\alpha
&:= \frac{\nu^2}
    {(1-\nu)\bigl[K_1+K_2(\nu/K_1)^{p_\alpha}\bigr]}
    +\frac{\nu}{K_1},\\
M_{\alpha,\mathrm{PEG}}
&:= \sqrt{2}D_0
    \bigl(K_0+K_1 F_*^{\alpha}
 +K_2(\sqrt{2}D_0)^{p_\alpha}\bigr)
    +F_*.
\end{align*}
Then, for all $k\ge0$,
\vspace{-2mm}
\begin{equation}\label{popov step lower bound main}
\gamma_{k,\peg}\ge\ugamma:=\frac{\nu}{\Delta},
\end{equation}
where
\[
\Delta :=
\begin{cases}
(L_0+L_1M_{1,\mathrm{PEG}})
\exp\!\left(\dfrac{\nu^2e^\nu}{1-\nu e^\nu}\right),
& \alpha=1,\\[6pt]
K_0+(K_1+K_2q_\alpha^{p_\alpha})
M_{\alpha,\mathrm{PEG}}^\alpha,
& \alpha\in (0,1).
\end{cases}
\]
\end{lemma}

For Lemma~\ref{lem:popov step lower bound}, let us note that the condition of $C_\nu<1$ is equivalent to choosing $\nu<1$ when $\alpha\in (0,1)$ and $\nu<0.56$ when $\alpha=1$. In addition, note that these lower bounds preserve the property of our proposed stepsizes that as $L_1\rightarrow0$ or $\alpha\rightarrow 0$, they are proportional to the reciprocal of the Lipschitz constant for each of these limiting cases. This will later allow our convergence guarantees to reduce to the $L$-Lipschitz case (up to a constant). The lower bound for the \popov stepsizes is smaller than that for \eg, largely due to our need to use the Lyapunov function $R_k$ for bounding the iterates rather than $\|x_0-x^*\|^2$ and our need to additionally bound the $\|x_k-\xhat_{k-1}\|$ term in the denominator. 

\section{Convergence Analysis}\label{sec:convergence analysis}
In this section, we present convergence guarantees for \eg and \peg in the $(L_0,L_1)$-Lipschitz setting for monotone (sublinear convergence) and strongly monotone operators (linear convergence).
As corollaries of our analysis, we provide refined convergence guarantees without exponential dependence on the initial distance from the solution for both \eg and \peg in the unconstrained setting as well as convergence guarantees for line-search variants of \eg and \peg. 

\subsection{Monotone VIs}

To measure convergence, we will use the \textit{restricted gap}
\begin{align*}
    \gap{x}=\max_{y\in B}\inprd{F(y)}{x-y}
\end{align*}
where $B=\set{y\in U: \|y-x^*\|\leq r\|x_0-x^*\|}$, $x_0$ is the initial point of the iterative algorithm, $x^*\in U^*$, and $r\in \R$.  
As long as the algorithm iterates are in $B$, then convergence of the restricted gap function to $0$ ensures convergence to a solution of a monotone VI \citep{nesterov_dual_2007}. Finally, here we will focus on the \textit{averaged iterates} $\bar{x}_K=\nicefrac{\sum_{k=0}^{K-1}(\gamma_k\xhatk)}{\sum_{k=0}^{K-1}\gamma_k}$ in the constrained setting and the best iterate in the unconstrained setting. For $L$-Lipschitz operators, \eg and \peg have slower $\bigo{\nicefrac{1}{\sqrt{K}}}$ last-iterate convergence rates \citep{cai_last_iterate_eg,popov_last_iterate} compared to the $\bigo{\nicefrac{1}{K}}$ convergence guarantees for the averaged iterate \citep{nemirovski_ergodic_eg,monteiro_ergodic_eg_guarantees,popov_ergodic_convergence,pmlr-v125-golowich20a}, and iterate averaging can also have practical benefits for GAN training in some settings \citep{gidel2018a,gan_tseng,yaz2018averaging}. Consequently, we leave last-iterate convergence in the $(L_0,L_1)$-Lipschitz setting for future work.

    \begin{theorem}\label{thm: EG monotone}
        Let $F$ be $\alpha$-symmetric $(L_0, L_1)$-Lipschitz and monotone. Additionally let $\xbark=\frac{\sum_{k=0}^{K-1}\gamma_k\xhatk}{\sum_{k=0}^{K-1}\gamma_k}$ and $x^*\in U^*$. Then the averaged iterates $\xbark$ of \eg with stepsizes given by (\ref{EG Stepsizes}) and $\nu$ chosen so that $C_\nu< 1$  satisfy
         \begin{equation*}
          \gap{\xbark}\leq \p{2+\frac{1}{\sqrt{1-C_\nu^2}}}^2\frac{\|x_0-x^*\|^2}{2\ugamma K}
        \end{equation*}
        where $\ugamma$ is defined in Lemma~\ref{lem:EG step lower bound} and $B=\{y\in U:\|y-x^*\|^2\leq  (1+(1-C_\nu^2)^{-\nicefrac{1}{2}})^2\|x_0-x^*\|^2\}$. Further, $\xbark\in B$.
    \end{theorem}

Theorem \ref{thm: EG monotone} provides the first convergence guarantees for \eg that apply to all monotone, $(L_0,L_1)$-Lipschitz operators in the constrained setting. Our permitted stepsizes are larger than those in \cite{choudhury2025extragradient} for best iterate convergence guarantees in the unconstrained case (there $C_\nu \leq\nicefrac{1}{\sqrt{2}}$). As discussed above, our use of the average iterate is common for \eg in monotone settings, and for $L$-Lipschitz operators, the average iterate achieves a $\bigo{\nicefrac{L}{K}}$ convergence rate \citep{nemirovski_ergodic_eg,monteiro_ergodic_eg_guarantees}. Here too, as either $L_1\rightarrow 0$ or $\alpha\rightarrow 0$, we obtain the corresponding $\bigo{\nicefrac{L}{K}}$ convergence guarantee, indicating the tightness of our analysis. While \cite{choudhury2025extragradient} also provided a second analysis where they are able to remove the $\exp(\|x_0-x^*\|)$ from the convergence guarantee for $\alpha=1$, central to their analysis is that $\|F(x_k)\|$ decreases to $0$ in the unconstrained case. Here in the constrained case, where the norm of the operator at the solution $\|F(x^*)\|$ may be arbitrarily large, we are unable to make similar guarantees of how long it will take for $\|F(x_k)\|$ to become sufficiently small for us to use a tighter lower bound on $\gamma_k\propto\nicefrac{1}{\|F(x_k)\|}$. Developing similar guarantees in the constrained setting is thus left for future work.

We next provide similar results for \peg.
    \begin{theorem}\label{thm: popov monotone}
        Let $F$ be $\alpha$-symmetric $(L_0,L_1)$-Lipschitz and monotone. Additionally, let $\xbark=\frac{\sum_{k=0}^{K-1}\gamma_k\xhatk}{\sum_{k=0}^{K-1}\gamma_k}$ and $x^*\in U^*$. Then the averaged iterates $\xbark$ of \peg with stepsize given by (\ref{Popov stepsizes}) and $\nu$ chosen so that $C_\nu\leq\frac{1}{2}$ satisfy 
        \begin{equation*}
            \gap{\xbark}\leq \frac{(3+2\sqrt{2})\|x_0-x^*\|^2}{2\ugamma K}
        \end{equation*}
        where $\ugamma$ is defined in Lemma~\ref{lem:popov step lower bound} and $B=\set{y\in U:\|y-x^*\|^2\leq 2\|x_0-x^*\|^2}$. Also, $\xbark\in B$.
    \end{theorem}

As with \eg, when $\alpha\rightarrow 0$ or $L_1\rightarrow 0$, we recover the expected $\bigo{\frac{L}{K}}$ ergodic convergence \citep{popov_ergodic_convergence}. We note that the change in stepsize constraints from \eg to \peg in the $(L_0,L_1)$-Lipschitz case parallels that of the $L$-Lipschitz case. While in the $L$-Lipschitz case, \eg requires $\gamma_k<\frac{1}{L}$ \citep{Korpelevich76} and \peg requires $\gamma_k<\frac{1}{2L}$ \citep{popov_ergodic_convergence}, in our setting \eg requires $C_\nu < 1$ while \popov requires $C_\nu\leq\half$. This parallel is a result of using proof techniques similar to the standard $L$-Lipschitz case, through Lemma~\ref{lem:lipschitz analogue}.

Finally let us consider the unconstrained setting with $\alpha=1$. For root-finding problems, $F(x^*)=0$, and $\|F(\xhat_{k-1})\|$ is the typical convergence measure. The following corollary establishes an $\bigo{1/K}$ bound
on the squared operator norm of the best iterate and
provides a refined guarantee without exponential dependence
on the initial distance $\|x_0-x^*\|$. A similar refinement was established for the monotone setting for \eg by \citep{choudhury2025extragradient}. To the best of our knowledge, the below is the first such guarantee for \peg.

\begin{corollary}\label{cor: popov monotone}
        Let $F$ be $1$-symmetric $(L_0,L_1)$-Lipschitz and monotone. In the unconstrained setting ($U=\R^d$), the iterates of \peg with stepsizes given by (\ref{Popov stepsizes}) and $\nu e^\nu=\frac{1}{2\sqrt{2}}$ ($\nu\approx 0.27$) satisfy
         \begin{equation*}
            \min_{0\leq k\leq K} \|F(\xhat_{k-1})\|^2 \leq \frac{2\|x_0-x^*\|^2}{\ugamma^2 (K+1)}
        \end{equation*}
        Also, for $K+1>\frac{2 L_1^2}{\nu^2}\exp\p{\frac{2\nu^2 e^\nu}{1-\nu e^\nu}}\|x_0-x^*\|^2$, 
         \begin{equation*}
            \mink \|F(\xhat_{k-1})\|\leq \frac{L_0 \sqrt{2}\exp\p{\frac{\nu^2 e^\nu}{1-\nu e^\nu}}\|x_0-x^*\|}{\nu\sqrt{K+1}-L_1\sqrt{2}\exp\p{\frac{\nu^2 e^\nu}{1-\nu e^\nu}}\|x_0-x^*\|}
        \end{equation*} 
    \end{corollary}

\subsection{Strongly Monotone VIs}
Next, we establish linear convergence of \eg and \peg for
strongly monotone VIs under the $\alpha$-symmetric
$(L_0,L_1)$-Lipschitz assumption.
\begin{theorem}\label{thm: EG sm}
         Let $F$ be $\alpha$-symmetric $(L_0, L_1)$-Lipschitz and $\mu$-strongly monotone. Then the iterates of \eg with stepsizes given by (\ref{EG Stepsizes}) and $\nu$ chosen so that $1-2\nu-\nu^2e^{2\nu}\geq0\ifaone$ and $1-\nu-\nu^2\geq0 \ifasmall$ satisfy
        \begin{equation*}
            \|x_k-x^*\|^2\leq (1-\ugamma\mu)^k \|x_0-x^*\|^2
        \end{equation*}
        where $\ugamma$ is defined in Lemma~\ref{lem:EG step lower bound}.
    \end{theorem}

The conditions on $\nu$ are satisfied by
$\nu\leq 0.36$ when $\alpha=1$ and $\nu\leq 0.61$
when $\alpha\in (0,1)$, matching the choices used in the
unconstrained analysis of \citep{choudhury2025extragradient}.
The resulting linear convergence bound yields
$\|x_K-x^*\|^2\leq\eps$ after
$\bigo{\frac{1}{\mu\ugamma}
\log\p{\frac{\|x_0-x^*\|^2}{\eps}}}$ iterations.
When $F(x^*)=0$, as in unconstrained root-finding,
we recover the convergence guarantees of
\cite{choudhury2025extragradient}.
For unconstrained problems with $\alpha=1$, the following
corollary further refines the iteration complexity bound
to remove exponential dependence on the initial distance
$\|x_0-x^*\|$.

\begin{corollary}\label{cor: eg usm}
      Let $F$ be $1$-symmetric $(L_0,L_1)$-Lipschitz and $\mu$-strongly monotone, and $U=\R^n$ (unconstrained case). Then the iterates of \eg with stepsizes given by (\ref{EG Stepsizes}) and $\nu$ chosen so that $\nu\approx0.21$ is the solution to the $1-4\nu-2\nu^2e^{2\nu}=0$ satisfy $\|x_{K+1}-x^*\|^2\leq\varepsilon$ after 
     \begin{equation*}
        \floor{\frac{8L_1^2}{\nu^2} \|x_0-x^*\|^2 }+ \ceil{\frac{2L_0}{\nu\mu}\log\frac{\|x_0-x^*\|^2}{\eps}} \text{ iterates.}
    \end{equation*}
\end{corollary}
The strongly monotone guarantee in
\cite{choudhury2025extragradient} removes exponential
dependence on the initial distance from the
$\eps$-dependent term, but retains it in the
$\eps$-independent term.
Corollary~\ref{cor: eg usm} removes this remaining
dependence, replacing the latter term with one quadratic
in $L_1\|x_0-x^*\|$.
Thus, our overall iteration complexity bound has no
exponential dependence on the initial distance to the solution.

We next present similar convergence guarantees for \peg in the strongly monotone regime.
    \begin{theorem}\label{thm: popov sm}
        Let $F$ be $\alpha$-symmetric $(L_0,L_1)$-Lipschitz and $\mu$-strongly monotone. Then the iterates of \peg with stepsizes given by (\ref{Popov stepsizes}) and $\nu$ chosen so that $1-2\nu-4C_\nu^2\geq 0$ satisfy
        \begin{equation*}
           R_k\leq (1-\ugamma\mu)^kR_0
        \end{equation*}
        where $R_k=\|x_k-x^*\|^2+\|x_k-\xhat_{k-1}\|^2$  and $\ugamma$ is defined in Lemma~\ref{lem:popov step lower bound}. 
    \end{theorem}
    
Our constraint on $\nu$ in Theorem~\ref{thm: popov sm} implies that $\nu\leq0.26\ifaone$ and $\nu\leq0.30\ifasmall$.
Theorem~\ref{thm: popov sm} provides a similar convergence guarantee to that of \cite{pmlr-v235-vankov24a} for $2$-quasi sharp operators, but with several advantages.  First, the stepsize in \cite{pmlr-v235-vankov24a} requires knowledge of the strong monotonicity parameter $\mu$, while the convergence of Theorem~\ref{thm: popov sm} is agnostic to this value. Second, when $\alpha=1$,
Theorem~\ref{thm: popov sm}  allows stepsizes at least twice as large as \cite{pmlr-v235-vankov24a}'s. Finally, when $\alpha\in (0,1)$ and $\nu$ is taken at its maximal value in the unconstrained setting, the stepsize in Theorem~\ref{thm: popov sm} is asymptotically upper bounded by $\nicefrac{0.30}{K_0}$ as $\|F(\xhatk)\|\rightarrow 0$, which is an improvement over the upper bound of $\min\set{\nicefrac{1}{(4\mu)},\nicefrac{1}{(6\sqrt{2}K_0)}}$ of \cite{pmlr-v235-vankov24a}. 
As a final note on Theorem~\ref{thm: popov sm}, as either $L_1\rightarrow 0$ or $\alpha\rightarrow0$, we recover the expected $\bigo{\frac{L}{\mu}\log\frac{1}{\varepsilon}}$ iteration complexity for \peg in the $L$-Lipschitz, $\mu$-strongly monotone setting \citep{gidel2018a}.

Similar to our results for \eg, we also obtain a refined iteration complexity bound
for \peg in the unconstrained setting with $\alpha=1$,
without exponential dependence on the initial distance
$\|x_0-x^*\|$.

\begin{corollary}\label{cor: popov sm}
      Let $F$ be $1$-symmetric $(L_0,L_1)$-Lipschitz and $\mu$-strongly monotone, and $U=\R^n$ (unconstrained case). Then the iterates of \peg with stepsizes given by (\ref{Popov stepsizes}) and $\nu$ chosen so that $\nu\approx0.17$ is the solution to the $1-4\nu-8\nu^2e^{2\nu}=0$ satisfy $R_k\leq\varepsilon$ after 
     \begin{equation*}
         \floor{\frac{8L_1^2}{\nu^2} \exp\p{\frac{2\nu^2 e^\nu}{1-\nu e^\nu}} R_0 } 
         +  \ceil{\frac{2L_0}{\nu\mu}\exp\p{\frac{\nu^2 e^\nu}{1-\nu e^\nu}} \log\frac{R_0}{\eps}}
     \end{equation*}
   iterates, where $R_k=\|x_k-x^*\|^2+\|x_k-\xhat_{k-1}\|^2$.
\end{corollary}
\vspace{-2mm}
\subsection{Line-search Variants}

In several applications, knowing the values of Lipschitz parameters $L_0$ and $L_1$ is impossible.
In these scenarios, when $L_0$ and $L_1$ are unknown, we propose using the line-search
variants \adaeg and \adapeg (Alg.~\ref{alg:AdaptiveAlgo} in Appendix~\ref{app: backtrack proofs}). For this direction we focus on the case of $\alpha=1$ for simplicity of exposition, and our approach retains the updates
\eqref{EG_update} and \eqref{PEG_update} with the stepsizes from \eqref{EG Stepsizes} and
\eqref{Popov stepsizes}, replacing $L_0,L_1$ with
positive estimates $\lambda_k^0,\lambda_k^1$.
In particular, our proposed \algname{(P)EG-LS} uses parameter updates similar to those
in the recent works of adaptive Frank--Wolfe algorithm
\citep{vyguzov2026frankwolfealgorithmsl0l1smooth}
and PolyakEG method
\citep{yoon2026polyaktypeextragradientmethodsmonotone}.

At each iteration of our line-search variants, a trial stepsize is accepted only
when \eqref{l0l1 half step general} holds with
$C_\nu=\nu e^\nu$.
Upon rejection, we alternate between increasing
$\lambda_k^0$ and $\lambda_k^1$, recompute the stepsize
and extrapolated point, and repeat the test.
The estimates and alternation state are carried across
iterations. We can prove that with this approach the while-loop of of the line-search is finite (when the line-search factor $\rho>2$). 
Together with boundedness of the (strongly) monotone iterates,
this yields positive uniform lower bounds on the
accepted stepsizes and the following convergence guarantees.
\begin{corollary}
Let $F$ be monotone and $1$-symmetric $(L_0,L_1)$-Lipschitz.
Run Algorithm~\ref{alg:AdaptiveAlgo} with positive
initial estimates, $\rho>2$, and $\nu>0$ satisfying
$\nu e^\nu<1$ for \adaeg or
$\nu e^\nu< \frac12$ for \adapeg. Then if 
\begin{itemize}[leftmargin=*]
    \item $F$ is \textbf{monotone}: the stepsize-weighted average $\bar{x}_K$
satisfies: 
$$\gap{\xbark}=\bigo{K^{-1}},$$ 
under the same conditions as 
Theorem~\ref{thm: EG monotone}
and Theorem~\ref{thm: popov monotone}.
\item $F$ is \textbf{strongly monotone}: the iterates $x_k$ have linear convergence in $\|x_k-x^*\|^2$ for \adaeg and in $R_k= \|x_k-x^*\|^2+\|x_k-\xhat_{k-1}\|^2$ for \adapeg.
\end{itemize} 
The explicit bounds are given in
Appendix~\ref{app: backtrack proofs}.
\end{corollary}
We note that unlike the backtracking results for \eg in \cite{pmlr-v235-vankov24a}, our convergence results do not make any additional assumptions on the relationship between $\gamma_k$ and $\mu$. Additionally, these variants also apply to $\alpha$-symmetric operators
with $\alpha\in(0,1)$: the inequality $t^\alpha\leq 1+t$
implies that such an operator is also
$1$-symmetric $(L_0+L_1,L_1)$-Lipschitz.

The pseudocode, rejection bound, explicit stepsize lower
bounds, and proofs appear in
Appendix~\ref{app: backtrack proofs}.

\section{Numerical Experiments}\label{sec: Numerical Experiments}
In this section, we perform several experiments to demonstrate the efficacy of our methods for solving several $(L_0,L_1)$-Lipschitz VIs in the constrained setting. Experiments were performed on a personal MacBook with an M2 Chip and 24 GB of RAM.

\begin{figure*}[t]
    \centering

    \begin{minipage}[t]{0.48\linewidth}
        \centering
        \includegraphics[width=\linewidth]{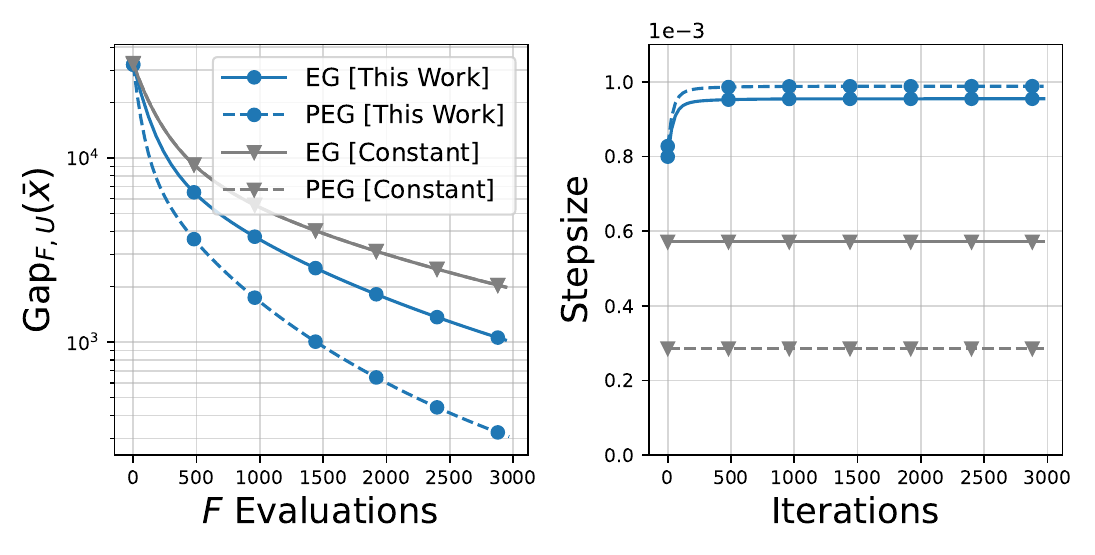}
        \caption{We compare \algname{EG} and \algname{PEG} using our stepsize and constant stepsizes based on the corresponding $L$-Lipschitz constant for a traffic equilibrium problem.}
        \label{fig:traffic plot}
    \end{minipage}
    \hfill
    \begin{minipage}[t]{0.48\linewidth}
        \centering
        \includegraphics[width=\linewidth]{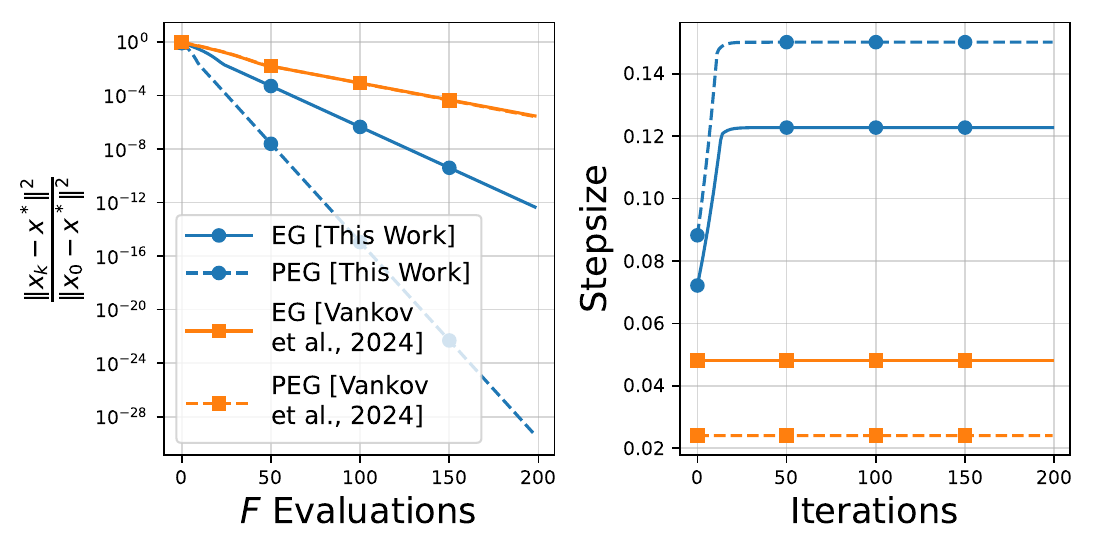}
        \caption{We compare the performance of \algname{EG} and \algname{PEG} with our stepsizes against those proposed by \cite{pmlr-v235-vankov24a} on a strongly monotone problem.}
        \label{fig:strongly monotone}
    \end{minipage}

\end{figure*}

\textbf{Performance on Traffic Equilibrium Problem} First, we consider the traffic equilibrium problem \citep{dafermos_traffic_1980} to demonstrate our stepsize performance on a monotone problem. We evaluate a traffic network inspired by \cite{traffic_problem_source} using the Bureau of Public Roads (BPR) cost function \citep{bureau1964traffic}. As detailed in Appendix~\ref{app:traffic}, the corresponding VI is a monotone $\frac{3}{4}$-symmetric $(0,0.103)$-Lipschitz operator, while the classical $L$-Lipschitz constant for the network under consideration is $1750$. As seen in Figure~\ref{fig:traffic plot}, our methods demonstrate the expected sublinear convergence of the \texttt{Gap} function. Utilizing the theoretical $(L_0,L_1)$-Lipschitz constants with our stepsize strategy allows us to take larger steps than the classical constant stepsizes. Consequently, our methods have the fastest convergence. Surprisingly, our theoretical stepsizes for \peg are larger than those of \eg. Combined with the fact that \peg requires only one operator evaluation per iteration, \peg with our stepsizes exhibits the fastest convergence in terms of operator evaluations.

\textbf{Performance on a Strongly Monotone Problem} We next compare our stepsizes with those proposed by \cite{pmlr-v235-vankov24a} on a strongly monotone problem. We consider the operator $F(x) = (x_1^2\quad x_2^2)^\top$ with $x$ constrained to the disc of radius $\nicefrac{1}{4}$ centered at $(\nicefrac{1}{2},\nicefrac{1}{2})$. This problem is $\frac{1}{2}$-symmetric $(0,2)$-Lipschitz \citep{choudhury2025extragradient}. As shown in Figure \ref{fig:strongly monotone}, our method exhibits linear convergence to the solution as expected. As seen in the left panel, our methods exhibit faster convergence to the solution, and in the right panel, we see our methods exhibit a greater increase in stepsizes over the course of the optimization trajectory. Like the traffic equilibrium problem, \peg with our stepsizes performs best overall.

\textbf{Performance on Min-max Problem} We conclude by demonstrating the performance of our line-search methods on a quartic $1$-symmetric $(L_0,L_1)$-Lipschitz min-max problem based on \cite{choudhury2025extragradient}:
\begin{equation}\label{eq:min_max_cubic_Rd}
    \min_{w_1 \in U_1} \max_{w_2 \in U_2}  \frac{1}{4} \|w_1\|^4+ w_1^{\top} \B w_2 - \frac{1}{4}\|w_2\|^4
\end{equation}

where $\B\in \R^{50 \times 50}$ is a positive definite matrix. We perform the experiment with constraint sets $U_1,U_2\subset\R^{50}$ defined to be balls centered at the ones vectors with radius $\sqrt{50}$ for the monotone experiment and radius $3$ for the strongly monotone experiment. We compare our methods against both \algname{aGRAAL} \citep{malitsky_golden_2020}, an adaptive algorithm for solving monotone locally Lipschitz VIs, and \cite{pmlr-v235-vankov24a}'s \algname{EG Backtracking} algorithm which has convergence guarantees for strongly monotone $(L_0,L_1)$-Lipschitz VIs. As can be seen in Figure~\ref{fig:minimax}, our methods show the expected sublinear convergence in the monotone setting and linear convergence in the strongly monotone setting, and have the fastest convergence of the algorithms tested. Additional details on this problem and hyperparameter tuning for each of the algorithms tested can be found in Appendix~\ref{app:minimax}.

\begin{figure*}[t]
    \centering
        \includegraphics[width=\linewidth]{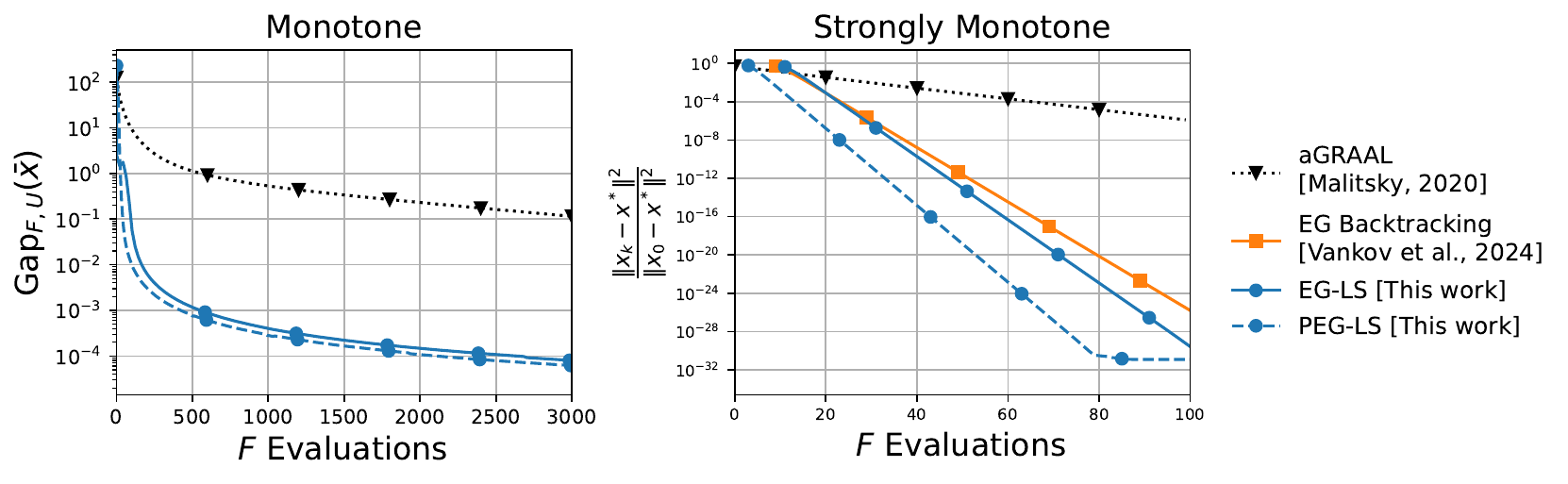}
    \caption{In Figure \ref{fig:minimax}, we assess the performance of \pegls on a min-max problem (\ref{eq:min_max_cubic_Rd}) whose monotonicity can be controlled by the constraint set $U_1\times U_2$.}
    \label{fig:minimax}
\end{figure*}

\section{Conclusion}

We establish sublinear convergence of \eg and \peg for
monotone $\alpha$-symmetric
$(L_0,L_1)$-Lipschitz VIs and linear convergence for
strongly monotone VIs, using stepsizes that do not require
knowledge of the strong monotonicity parameter.
For unconstrained problems, we further obtain iteration complexity bounds
without exponential dependence on the initial distance
to the solution. We also develop line-search variants
that retain these convergence rates without requiring
knowledge of $L_0$ or $L_1$.

While this work provides convergence analysis applicable to non-Lipschitz settings, relaxing monotonicity for $(L_0,L_1)$-Lipschitz is an important direction for future work. Specifically, extending \cite{choudhury2025extragradient}'s results for Weak Minty root-finding problems to the VI setting is a natural next step for non-monotone convergence guarantees beyond the $p$-quasi sharp operators treated in \cite{pmlr-v235-vankov24a}. Additionally, while initial progress has been made in  the stochastic setting for $(L_0,L_1)$-Lipschitz operators \citep{vankov2025stochasticvip}, convergence guarantees for stochastic monotone operators remain an open question. 

\section*{Acknowledgments}
The authors’ contribution to this work was supported by NSF CCF 2504626, NSF CAREER
2542902, and AFOSR YIP FA9550-26-1-B353.

\newpage

\subsection*{AI use statement}

In this work, we used generative AI tools to verify our already derived theoretical proofs and to provide feedback on our final experiments. We also used ChatGPT to help compress our exposition for the final version of this work to fit in the given space.  

We have not used generative AI tools to generate novel convergence results (that authors did not obtain first), formulate mathematical claims, or implement our proposed methods. 

 We take responsibility for the final content of this work, including text, claims, or artifacts produced with the aid of generative AI.

\bibliographystyle{abbrv}
\bibliography{references}

@inproceedings{
    Zhang2020Why,
    title={Why Gradient Clipping Accelerates Training: A Theoretical Justification for Adaptivity},
    author={Jingzhao Zhang and Tianxing He and Suvrit Sra and Ali Jadbabaie},
    booktitle={ICLR},
    year={2020},
    furl={https://openreview.net/forum?id=BJgnXpVYwS}
}

@article{yoon2026same,
  title={On Same-Sample and Independent-Sample Stochastic Extragradient for Monotone Variational Inequalities},
  author={Yoon, TaeHo and Loizou, Nicolas},
  journal={arXiv preprint arXiv:2608.06182},
  year={2026}
}

@inproceedings{
    gidel2018a,
    title={A Variational Inequality Perspective on Generative Adversarial Networks},
    author={Gauthier Gidel and Hugo Berard and Gaëtan Vignoud and Pascal Vincent and Simon Lacoste-Julien},
    booktitle={ICLR},
    year={2019},
    furl={https://openreview.net/forum?id=r1laEnA5Ym},
}

@inproceedings{
    choudhury2023singlecall,
    title={Single-Call Stochastic Extragradient Methods for Structured Non-monotone Variational Inequalities: Improved Analysis under Weaker Conditions},
    author={Sayantan Choudhury and Eduard Gorbunov and Nicolas Loizou},
    booktitle={NeurIPS},
    year={2023},
    furl={https://openreview.net/forum?id=tFeaLw9AWn}
}

@article{Korpelevich76,
    author = {G.M. Korpelevich},
    title = {The extragradient method for finding saddle points and other
problems},
    journal = {Matecon},
    year ={1976} ,
    volume={12},
    number={4},
    pages={747-756}
}

@inproceedings{
    gorbunov2025methods,
    title={Methods for Convex {$(L_0,L_1)$}-Smooth Optimization: Clipping, Acceleration, and Adaptivity},
    author={Eduard Gorbunov and Nazarii Tupitsa and Sayantan Choudhury and Alen Aliev and Peter Richt{\'a}rik and Samuel Horv{\'a}th and Martin Tak{\'a}{\v{c}}},
    booktitle={ICLR},
    year={2025},
    furl={https://openreview.net/forum?id=0wmfzWPAFu}
}

@InProceedings{pmlr-v235-vankov24a,
  title = 	 {Generalized Smooth Variational Inequalities: Methods with Adaptive Stepsizes},
  author =       {Vankov, Daniil and Nedich, Angelia and Sankar, Lalitha},
  booktitle = 	 {ICML},
  year = 	 {2024},
  furl = 	 {https://proceedings.mlr.press/v235/vankov24a.html},
}

@InProceedings{pmlr-v202-chen23ar,
  title = 	 {Generalized-Smooth Nonconvex Optimization is As Efficient As Smooth Nonconvex Optimization},
  author =       {Chen, Ziyi and Zhou, Yi and Liang, Yingbin and Lu, Zhaosong},
  booktitle = 	 {ICML},
  year = 	 {2023},
  furl = 	 {https://proceedings.mlr.press/v202/chen23ar.html},
}

@InProceedings{pmlr-v151-gorbunov22a,
  title = 	 { Extragradient Method: O(1/K) Last-Iterate Convergence for Monotone Variational Inequalities and Connections With Cocoercivity },
  author =       {Gorbunov, Eduard and Loizou, Nicolas and Gidel, Gauthier},
  booktitle = 	 {AISTATS},
  year = 	 {2022},
  furl = 	 {https://proceedings.mlr.press/v151/gorbunov22a.html},
}

@article{popov1980,
    author = {Popov, L. D.},
    title = {A modification of the arrow-hurwicz method for search of saddle points.},
    journal = {Mathematical notes of the Academy of Sciences of the USSR},
    year = {1980},
    volume = {28},
    pages= {845-848}
}

@InProceedings{pmlr-v108-mishchenko20a,
  title = 	 {Revisiting Stochastic Extragradient},
  author =       {Mishchenko, Konstantin and Kovalev, Dmitry and Shulgin, Egor and Richtarik, Peter and Malitsky, Yura},
  booktitle = 	 {AISTATS},
  year = 	 {2020},
  furl = 	 {https://proceedings.mlr.press/v108/mishchenko20a.html},
}

@inproceedings{
sokota2023a,
title={A Unified Approach to Reinforcement Learning, Quantal Response Equilibria, and Two-Player Zero-Sum Games},
author={Samuel Sokota and Ryan D'Orazio and J Zico Kolter and Nicolas Loizou and Marc Lanctot and Ioannis Mitliagkas and Noam Brown and Christian Kroer},
booktitle={ICLR},
year={2023},
furl={https://openreview.net/forum?id=DpE5UYUQzZH}
}

@inproceedings{goodfellow14,
 author = {Goodfellow, Ian J. and Pouget-Abadie, Jean and Mirza, Mehdi and Xu, Bing and Warde-Farley, David and Ozair, Sherjil and Courville, Aaron and Bengio, Yoshua},
 booktitle = {NeurIPS},
 title = {Generative Adversarial Nets},
 furl = {https://proceedings.neurips.cc/paper_files/paper/2014/file/f033ed80deb0234979a61f95710dbe25-Paper.pdf},
 year = {2014}
}

@InProceedings{diakonikolas21a_weak_mvi_eg,
  title = 	 { Efficient Methods for Structured Nonconvex-Nonconcave Min-Max Optimization },
  author =       {Diakonikolas, Jelena and Daskalakis, Constantinos and Jordan, Michael},
  booktitle = 	 {AISTATS},
  year = 	 {2021},
  furl = 	 {https://proceedings.mlr.press/v130/diakonikolas21a.html},
}

@article{nesterov_dual_2007,
	title = {Dual extrapolation and its applications to solving variational inequalities and related problems},
	volume = {109},
	issn = {1436-4646},
	furl = {https://doi.org/10.1007/s10107-006-0034-z},
	doi = {10.1007/s10107-006-0034-z},
	language = {en},
	number = {2},
	furldate = {2025-10-30},
	journal = {Mathematical Programming},
	author = {Nesterov, Yurii},
	year = {2007},
	pages = {319--344},
}

@article{monteiro_ergodic_eg_guarantees,
    author = {Monteiro, Renato D. C. and Svaiter, B. F.},
    title = {On the Complexity of the Hybrid Proximal Extragradient Method for the Iterates and the Ergodic Mean},
    journal = {SIAM Journal on Optimization},
    volume = {20},
    number = {6},
    pages = {2755-2787},
    year = {2010},
    doi = {10.1137/090753127},
    furl = {https://doi.org/10.1137/090753127},
    eprint = {https://doi.org/10.1137/090753127},
}

@article{nemirovski_ergodic_eg,
    author = {Nemirovski, Arkadi},
    title = {Prox-Method with Rate of Convergence O(1/t) for Variational Inequalities with Lipschitz Continuous Monotone Operators and Smooth Convex-Concave Saddle Point Problems},
    journal = {SIAM Journal on Optimization},
    volume = {15},
    number = {1},
    pages = {229-251},
    year = {2004},
    doi = {10.1137/S1052623403425629},
    furl = {https://doi.org/10.1137/S1052623403425629},
    eprint = {https://doi.org/10.1137/S1052623403425629}
}

@InProceedings{pmlr-v125-golowich20a,
  title = 	 {Last Iterate is Slower than Averaged Iterate in Smooth Convex-Concave Saddle Point Problems},
  author =       {Golowich, Noah and Pattathil, Sarath and Daskalakis, Constantinos and Ozdaglar, Asuman},
  booktitle = 	 {Proceedings of Thirty Third Conference on Learning Theory},
  pages = 	 {1758--1784},
  year = 	 {2020},
  editor = 	 {Abernethy, Jacob and Agarwal, Shivani},
  volume = 	 {125},
  series = 	 {Proceedings of Machine Learning Research},
  publisher =    {PMLR},
  furl = 	 {https://proceedings.mlr.press/v125/golowich20a.html}
}

@inproceedings{cai_last_iterate_eg,
 author = {Cai, Yang and Oikonomou, Argyris and Zheng, Weiqiang},
 booktitle = {NeurIPS},
 title = {Finite-Time Last-Iterate Convergence for Learning in Multi-Player Games},
 furl = {https://proceedings.neurips.cc/paper_files/paper/2022/file/db2d2001f63e83214b08948b459f69f0-Paper-Conference.pdf},
 year = {2022}
}

@article{tseng2000,
author = {Tseng, Paul},
title = {A Modified Forward-Backward Splitting Method for Maximal Monotone Mappings},
journal = {SIAM Journal on Control and Optimization},
volume = {38},
number = {2},
pages = {431-446},
year = {2000},
doi = {10.1137/S0363012998338806},
furl = {https://doi.org/10.1137/S0363012998338806},
eprint = { https://doi.org/10.1137/S0363012998338806}
}

@article{malitsky2015,
author = {Malitsky, Yu.},
title = {Projected Reflected Gradient Methods for Monotone Variational Inequalities},
journal = {SIAM Journal on Optimization},
volume = {25},
number = {1},
pages = {502-520},
year = {2015},
doi = {10.1137/14097238X},
furl = { https://doi.org/10.1137/14097238X},
eprint = { https://doi.org/10.1137/14097238X}
}

@article{malitsky2020,
author = {Malitsky, Yura and Tam, Matthew K.},
title = {A Forward-Backward Splitting Method for Monotone Inclusions Without Cocoercivity},
journal = {SIAM Journal on Optimization},
volume = {30},
number = {2},
pages = {1451-1472},
year = {2020},
doi = {10.1137/18M1207260},
furl = {  https://doi.org/10.1137/18M1207260 },
eprint = {  https://doi.org/10.1137/18M1207260 }
}

@article{yardim2025,
    author = {Yardim, Batuhan and Cayci, Semih and He, Niao},
    title = {A Variational Inequality Approach to Independent Learning in Static Mean-Field Games},
    year = {2025},
    issue_date = {June 2025},
    publisher = {Association for Computing Machinery},
    address = {New York, NY, USA},
    volume = {2},
    number = {2},
    furl = {https://doi.org/10.1145/3731553},
    doi = {10.1145/3731553},
    journal = {ACM / IMS J. Data Sci.},
    articleno = {7},
    numpages = {43}
}

@inproceedings{
    sidahmed2025addressing,
    title={Addressing Rotational Learning Dynamics in Multi-Agent Reinforcement Learning},
    author={Baraah A. M. Sidahmed and Tatjana Chavdarova},
    booktitle={Eighteenth European Workshop on Reinforcement Learning},
    year={2025},
    furl={https://openreview.net/forum?id=RtbH7KnH0n}
}

@article{
    vankov2025stochasticvip,
    title={Generalized Smooth Stochastic Variational Inequalities:  Almost Sure Convergence and Convergence Rates},
    author={Daniil Vankov and Angelia Nedich and Lalitha Sankar},
    journal={Transactions on Machine Learning Research},
    issn={2835-8856},
    year={2025},
    furl={https://openreview.net/forum?id=EjqSpbUBWU},
    note={}
}

@article{beznosikov_distributed_2025,
	title = {Distributed saddle point problems: lower bounds, near-optimal and robust algorithms},
	volume = {40},
	issn = {1055-6788},
	shorttitle = {Distributed saddle point problems},
	furl = {https://doi.org/10.1080/10556788.2025.2463986},
	doi = {10.1080/10556788.2025.2463986},
	number = {5},
	furldate = {2026-01-29},
	journal = {Optimization Methods and Software},
	publisher = {Taylor \& Francis},
	author = {Beznosikov, Aleksandr and Samokhin, Valentin and Gasnikov, Alexander},
	year = {2025},
	pages = {1249--1266},
}

@inproceedings{
    cai2023accelerated,
    title={Accelerated Single-Call Methods for Constrained Min-Max Optimization},
    author={Yang Cai and Weiqiang Zheng},
    booktitle={ICLR},
    year={2023},
    furl={https://openreview.net/forum?id=HRwN7IQLUKA}
}

@article{dafermos_traffic_1980,
	title = {Traffic {Equilibrium} and {Variational} {Inequalities}},
	volume = {14},
	issn = {0041-1655},
	furl = {https://pubsonline.informs.org/doi/10.1287/trsc.14.1.42},
	doi = {10.1287/trsc.14.1.42},
	number = {1},
	furldate = {2026-03-19},
	journal = {Transportation Science},
	publisher = {INFORMS},
	author = {Dafermos, Stella},
	year = {1980},
	pages = {42--54},
}

@article{conical_traffic,
    author = {Spiess, Heinz},
    title = {Technical Note—Conical Volume-Delay Functions},
    journal = {Transportation Science},
    volume = {24},
    number = {2},
    pages = {153-158},
    year = {1990},
    doi = {10.1287/trsc.24.2.153},
    furl = { 
            https://doi.org/10.1287/trsc.24.2.153
    },
    eprint = { 
            https://doi.org/10.1287/trsc.24.2.153
    }
}

@inproceedings{
    choudhury2025extragradient,
    title={Extragradient Method for {$(L_0, L_1)$}-Lipschitz Root-finding Problems},
    author={Sayantan Choudhury and Nicolas Loizou},
    booktitle={NeurIPS},
    year={2025},
    furl={https://openreview.net/forum?id=jRnGy4ZzK5}
}

@article{malitsky_golden_2020,
	title = {Golden ratio algorithms for variational inequalities},
	volume = {184},
	issn = {1436-4646},
	furl = {https://doi.org/10.1007/s10107-019-01416-w},
	doi = {10.1007/s10107-019-01416-w},
	language = {en},
	number = {1},
	furldate = {2026-01-20},
	journal = {Mathematical Programming},
	author = {Malitsky, Yura},
	year = {2020},
	pages = {383--410},
}

@book{facchinei_finite-dimensional_2003,
	address = {New York, NY},
	series = {Springer {Series} in {Operations} {Research} and {Financial} {Engineering}},
	title = {Finite-{Dimensional} {Variational} {Inequalities} and {Complementarity} {Problems}},
	copyright = {http://www.springer.com/tdm},
	isbn = {978-0-387-95580-3},
	furl = {http://link.springer.com/10.1007/b97543},
	doi = {10.1007/b97543},
	language = {en},
	furldate = {2026-04-05},
	publisher = {Springer},
	editor = {Facchinei, Francisco and Pang, Jong-Shi},
	year = {2003},
}

@inproceedings{popov_ergodic_convergence,
 author = {Hsieh, Yu-Guan and Iutzeler, Franck and Malick, J\'{e}r\^{o}me and Mertikopoulos, Panayotis},
 booktitle = {NeurIPS},
 title = {On the convergence of single-call stochastic extra-gradient methods},
 furl = {https://proceedings.neurips.cc/paper_files/paper/2019/file/4625d8e31dad7d1c4c83399a6eb62f0c-Paper.pdf},
 year = {2019}
}

@article{gan_tseng,
author = {B\"{o}hm, Axel and Sedlmayer, Michael and Csetnek, Ern\"{o} Robert and Bo\c{t}, Radu Ioan},
title = {Two Steps at a Time---Taking GAN Training in Stride with Tseng's Method},
journal = {SIAM Journal on Mathematics of Data Science},
volume = {4},
number = {2},
pages = {750-771},
year = {2022},
doi = {10.1137/21M1420939},
furl = {  https://doi.org/10.1137/21M1420939 },
eprint = {  https://doi.org/10.1137/21M1420939 },
}

@InProceedings{pmlr-v139-yoon21d,
  title = 	 {Accelerated Algorithms for Smooth Convex-Concave Minimax Problems with $O(1/k^2)$ Rate on Squared Gradient Norm},
  author =       {Yoon, Taeho and Ryu, Ernest K},
  booktitle = 	 {ICML},
  year = 	 {2021},
  furl = 	 {https://proceedings.mlr.press/v139/yoon21d.html},
}

@inproceedings{popov_last_iterate,
 author = {Gorbunov, Eduard and Taylor, Adrien and Gidel, Gauthier},
 booktitle = {NeurIPS},
 title = {Last-Iterate Convergence of Optimistic Gradient Method for Monotone Variational Inequalities},
 furl = {https://proceedings.neurips.cc/paper_files/paper/2022/file/893cd874ba98afa54ae9e385a24a83ac-Paper-Conference.pdf},
 year = {2022}
}

@article
    {traffic_problem_source, 
    author={Van Hieu, Dang and Dung Muu, Le and Kim Quy, Pham},
    title = {Strong convergence of multi-parameter projection methods for variational inequality problems}, 
    volume={27}, 
    furl={https://journals.vilniustech.lt/index.php/MMA/article/view/14479}, 
    DOI={10.3846/mma.2022.14479}, 
    abstractNote={
        In this paper, we introduce a multi-parameter projection method for solving a variational inequality problem, and establish its strong convergence in a Hilbert space under appropriate conditions. The method involves two projectionsteps with different variable stepsizes where one of them is computed explicitly on a specifically structural half-space. The proof of strong convergence of the method is based on the regularization solutions depending on parameters of the original problem. It turns out that the solution obtained by the method is the solution of a bilevel variational inequality problem whose constraint is the solution set of our considered problem. In order to support the obtained theoretical results, we perform some experiments on transportation equilibrium and optimal control problems, and also involve comparisons. Numerical results show the computational effectiveness and the fast convergence of the new method over some existing ones.
    }, number={2}, 
    journal={Mathematical Modelling and Analysis}, 
    year={2022}, 
    pages={242–262} 
}

@book{bureau1964traffic,
  author    = {{Bureau of Public Roads}},
  title     = {Traffic Assignment Manual},
  year      = {1964},
  publisher = {U.S. Department of Commerce, Urban Planning Division},
  address   = {Washington, DC}
}

@article{trandinh2023sublinearconvergenceratesextragradienttype,
    author = {Quoc Tran-Dinh},
    title = {Sublinear Convergence Rates of Extragradient-Type Methods: A Survey on Classical and Recent Developments},
    journal = {arXiv preprint arXiv:2303.17192},
    year = {2023}
}

@article{beyond_golden_alacaoglu,
  author  = {Alacaoglu, Ahmet and B{\"o}hm, Axel and Malitsky, Yura},
  title   = {Beyond the Golden Ratio for Variational Inequality Algorithms},
  journal = {Journal of Machine Learning Research},
  year    = {2023},
  volume  = {24},
  number  = {172},
  pages   = {1--33},
  furl     = {http://jmlr.org/papers/v24/22-1488.html}
}

@inproceedings{
    daskalakis2018training,
    title={Training {GAN}s with Optimism},
    author={Constantinos Daskalakis and Andrew Ilyas and Vasilis Syrgkanis and Haoyang Zeng},
    booktitle={ICLR},
    year={2018},
    furl={https://openreview.net/forum?id=SJJySbbAZ},
}

@InProceedings{pmlr-v151-gorbunov22b,
  title = 	 { Stochastic Extragradient: General Analysis and Improved Rates },
  author =       {Gorbunov, Eduard and Berard, Hugo and Gidel, Gauthier and Loizou, Nicolas},
  booktitle = 	 {AISTATS},
  year = 	 {2022},
  furl = 	 {https://proceedings.mlr.press/v151/gorbunov22b.html}
}

@inproceedings{beznosikov_decentralized_seg22,
 author = {Beznosikov, Aleksandr and Dvurechenskii, Pavel and Koloskova, Anastasiia and Samokhin, Valentin and Stich, Sebastian U and Gasnikov, Alexander},
 booktitle = {NeurIPS},
 title = {Decentralized Local Stochastic Extra-Gradient for Variational Inequalities},
 furl = {https://proceedings.neurips.cc/paper_files/paper/2022/file/f9379afacdbabfdc6b060972b60f9ab8-Paper-Conference.pdf},
 year = {2022}
}

@inproceedings{sgda_consensus21,
 author = {Loizou, Nicolas and Berard, Hugo and Gidel, Gauthier and Mitliagkas, Ioannis and Lacoste-Julien, Simon},
 booktitle = {NeurIPS},
 title = {Stochastic Gradient Descent-Ascent and Consensus Optimization for Smooth Games: Convergence Analysis under Expected Co-coercivity},
 furl = {https://proceedings.neurips.cc/paper_files/paper/2021/file/9f96f36b7aae3b1ff847c26ac94c604e-Paper.pdf},
 year = {2021}
}

@InProceedings{sgda_unified_beznosikov23a,
  title = 	 {Stochastic Gradient Descent-Ascent: Unified Theory and New Efficient Methods},
  author =       {Beznosikov, Aleksandr and Gorbunov, Eduard and Berard, Hugo and Loizou, Nicolas},
  booktitle = 	 {AISTATS},
  year = 	 {2023},
  furl = 	 {https://proceedings.mlr.press/v206/beznosikov23a.html},
}

@inproceedings{
    borodich2026nesterov,
    title={Nesterov Finds {GRAAL}: Optimal and Adaptive Gradient Method for Convex Optimization},
    author={Ekaterina Borodich and Dmitry Kovalev},
    booktitle={ICLR},
    year={2026},
    furl={https://openreview.net/forum?id=vPSiCA3CkD}
}

@inproceedings{
vankov2025optimizing,
title={Optimizing {$(L_0, L_1)$}-Smooth Functions by Gradient Methods},
author={Daniil Vankov and Anton Rodomanov and Angelia Nedich and Lalitha Sankar and Sebastian U Stich},
booktitle={ICLR},
year={2025},
furl={https://openreview.net/forum?id=GQ1Tc3vHbt}
}

@inproceedings{
    riabinin2026from,
    title={From Muon to Gluon: Bridging Theory and Practice of {LMO}-based Optimizers for {LLM}s},
    author={Artem Riabinin and Egor Shulgin and Kaja Gruntkowska and Peter Richt{\'a}rik},
    booktitle={ICML},
    year={2026},
    furl={https://openreview.net/forum?id=IelAHU5MVz}
}

@inproceedings{
    tyurin2025toward,
    title={Toward a Unified Theory of Gradient Descent under Generalized Smoothness},
    author={Alexander Tyurin},
    booktitle={ICML},
    year={2025},
    furl={https://openreview.net/forum?id=mzle2Jnt72}
}

@InProceedings{pmlr-v235-xian24a,
  title = 	 {Delving into the Convergence of Generalized Smooth Minimax Optimization},
  author =       {Xian, Wenhan and Chen, Ziyi and Huang, Heng},
  booktitle = 	 {ICML},
  year = 	 {2024},
  editor = 	 {Salakhutdinov, Ruslan and Kolter, Zico and Heller, Katherine and Weller, Adrian and Oliver, Nuria and Scarlett, Jonathan and Berkenkamp, Felix},
  furl = 	 {https://proceedings.mlr.press/v235/xian24a.html},
}

@inproceedings{
    yaz2018averaging,
    title={The Unusual Effectiveness of Averaging in {GAN} Training},
    author={Yasin Yaz{\i}c{\i} and Chuan-Sheng Foo and Stefan Winkler and Kim-Hui Yap and Georgios Piliouras and Vijay Chandrasekhar},
    booktitle={ICLR},
    year={2019},
    furl={https://openreview.net/forum?id=SJgw_sRqFQ},
}

@article{vyguzov2026frankwolfealgorithmsl0l1smooth,
    author = {A. A. Vyguzov and F. S. Stonyakin},
    title = {Frank-Wolfe Algorithms for (L0, L1)-smooth functions},
    journal = {arXiv preprint arxiv:2510.16468},
    year = {2026} 
}

@article{yoon2026polyaktypeextragradientmethodsmonotone,
  title={Polyak-Type Extragradient Methods for Monotone Root-Finding Problems},
  author={Yoon, TaeHo and Choudhury, Sayantan and Greenberg, Ezra and Loizou, Nicolas},
  journal={arXiv preprint arXiv:2609.26581},
  year={2026}
}

\newpage

\appendix
\part*{Supplementary Material}

\section{Further Related Work}\label{app:related work}
\textbf{Algorithms for Solving VIs.} The simplest method for solving variational inequalities is the projection method \citep{sgda_consensus21,sgda_unified_beznosikov23a}, though as noted in the main text, this diverges for simple monotone problems \citep{gidel2018a}. Instead, \eg \citep{Korpelevich76} has proven a popular solution to solving monotone VIs, and has seen extensions in stochastic \citep{pmlr-v108-mishchenko20a,pmlr-v151-gorbunov22b}, non-monotone \citep{diakonikolas21a_weak_mvi_eg}, and distributed settings \citep{beznosikov_decentralized_seg22}. \peg, initially proposed by Popov \citep{popov1980}, was the first method to reduce the per-iteration computational cost to only one operator call per iteration. Like \eg, it has seen various stochastic extensions \citep{daskalakis2018training,popov_last_iterate,choudhury2023singlecall}. Beyond \eg and \peg, other methods have been introduced that instead reduce the projections per iteration \citep{tseng2000} or both the operator and projection calls per iteration \citep{malitsky2015,malitsky2020,malitsky_golden_2020,beyond_golden_alacaoglu}. Accelerated methods have also been developed with faster convergence rates for monotone \citep{pmlr-v139-yoon21d} and non-monotone \citep{cai2023accelerated} settings. \citep{trandinh2023sublinearconvergenceratesextragradienttype} provides a survey of many of the relevant methods.

\textbf{Relaxed Lipschitz Conditions.} The $(L_0,L_1)$-smoothness condition was first introduced by \cite{Zhang2020Why} and was further developed to the more general $\alpha$-symmetric setting in \cite{pmlr-v202-chen23ar}. Various convergence guarantees have been developed for gradient descent methods in this setting \citep{vankov2025optimizing}. New generalizations of this condition have been proposed \citep{tyurin2025toward}, with a layer-wise variant providing a theoretical basis for the success of Gluon for example \citep{riabinin2026from}. 

\cite{pmlr-v235-xian24a} extended these ideas to min-max optimization while \citep{pmlr-v235-vankov24a} first introduced the $(L_0,L_1)$-Lipschitz assumption for the VI setting, where \cite{pmlr-v235-vankov24a} analyzed \eg and \peg under this setting for $p$-quasi sharp operators. \cite{choudhury2025extragradient} extended the convergence guarantees for \eg in the unconstrained setting to monotone and weak Minty problems while \cite{vankov2025stochasticvip} gave the first results for \eg in the stochastic $(L_0,L_1)$-Lipschitz setting. Most recently, \cite{yoon2026polyaktypeextragradientmethodsmonotone} demonstrated the use of Polyak-type stepsizes for \eg when solving $(L_0,L_1)$-Lipschitz root-finding problems.

Prior to the introduction of the $(L_0,L_1)$-Lipschitz assumption, \cite{malitsky_golden_2020,beyond_golden_alacaoglu} created and analyzed the ``Adaptive Golden Ratio Algorithm'' (\algname{aGRAAL}), which only requires local Lipschitzness for solving VIs. An accelerated version of GRAAL for the minimization setting with convergence guarantees for $(L_0,L_1)$-smooth functions was introduced more recently \citep{borodich2026nesterov}. Additionally, several works have analyzed the class of star-cocoercive operators \citep{sgda_consensus21,pmlr-v151-gorbunov22a}, which includes some non-Lipschitz operators.

\newpage

\section{Demonstration of $(L_0,L_1)$-Lipschitz Assumption Along $\algname{EG}$ and $\algname{PEG}$ Trajectories}
\label{app: l0l1 trajectory}
As discussed in the main text, the spectral norm of the Jacobian $\|\J\|$ has been observed to scale with the operator norm $\|F\|$ for \eg in the unconstrained setting in \cite{choudhury2025extragradient}. Here, we provide the first such plots for \eg and \peg in the constrained setting. 

Consider the quartic min-max problem
\begin{equation*}
    \min_{w_1\in U_1} \max_{w_2\in U_2} \mathcal{L}(w_1, w_2) = \frac{1}{4}w_1^4 + w_1w_2 - \frac{1}{4}w_2^4,
\end{equation*}

where $U_1=U_2=[0,2]$. As seen in Figure~\ref{fig:jacobian_vs_gradient_plot}, the Jacobian norm decreases with the operator norm sublinearly at the beginning of each algorithm's trajectory, and then continues decreasing in a more linear relationship after the projections are applied due to the constraints. Such behavior provides motivation for the $\alpha$-symmetric $(L_0,L_1)$-Lipschitz assumption. 

\begin{figure}[H]
    \centering
 \includegraphics[width=0.8\textwidth]{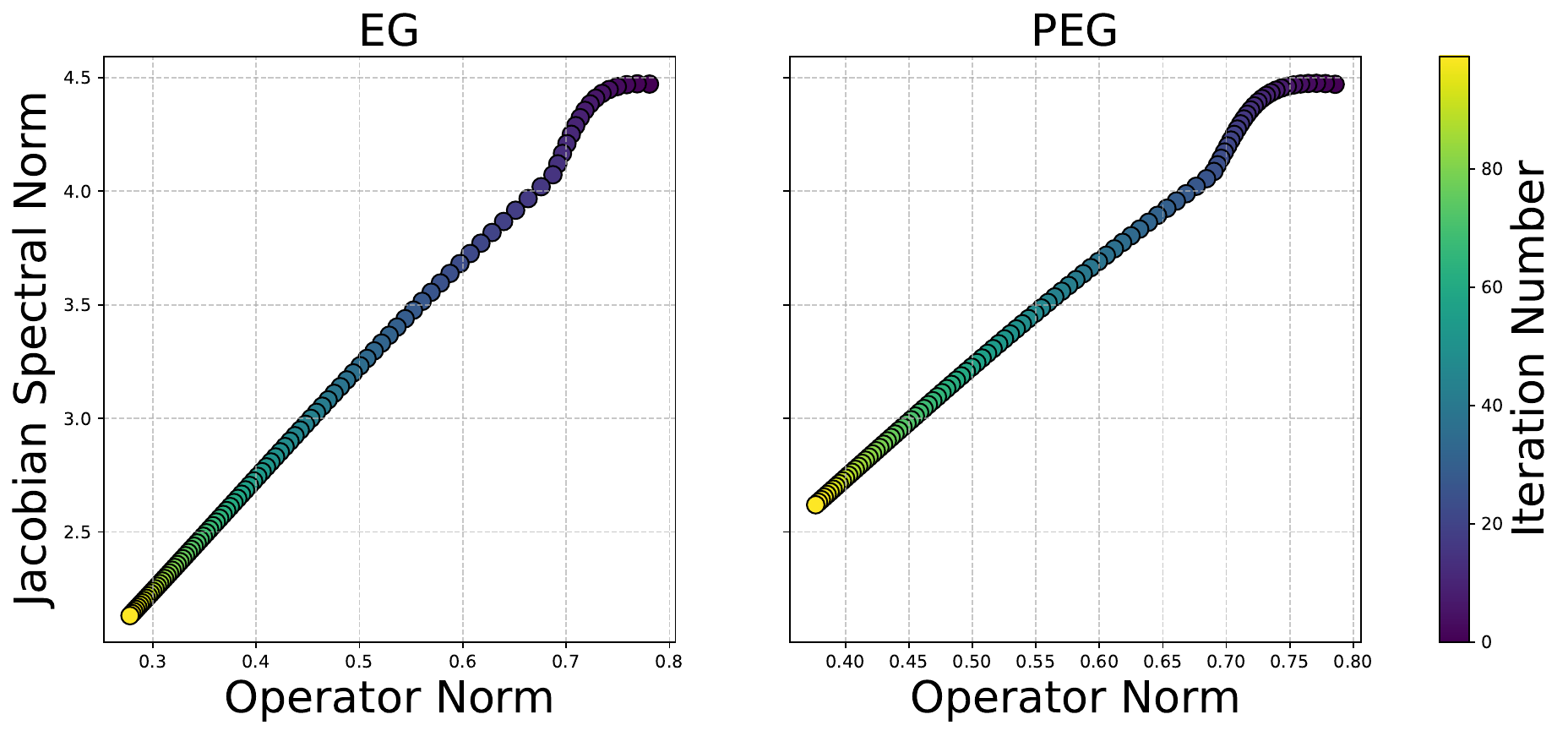}
    \caption{\small{Scatter plot of $\| \J(x_k)\|$ against $\|F(x_k)\|$.}}
    \label{fig:jacobian_vs_gradient_plot}
    \vspace{-2mm}
\end{figure}

\newpage

\section{Technical Lemmas}\label{app:technical lemmas}
In this section, we present various technical lemmas which we will use to prove the main results of our work in later sections.

First, some standard inequalities:
\begin{lemma}[Young's Inequality]
    For $a,b\in\R^d$ and $\beta> 0$ we have
    \begin{equation}
        \|a+b\|^2 \leq(1+\beta^{-1})\|a\|^2+(1+\beta)\|b\|^2 \label{full young}
    \end{equation}
    When $\beta=1$, we have
    \begin{equation}
        \|a+b\|^2\leq 2\|a\|^2+2\|b\|^2 \label{young ineq}
    \end{equation}
    We can also rearrange \eqref{full young} to obtain
    \begin{equation}
        -\|a\|^2\leq-\frac{1}{1+\beta}\|a+b\|^2+\frac{1}{\beta}\|b\|^2
        \label{Young-like}
    \end{equation}
\end{lemma}
\begin{lemma}
    For $a,b\in\R$, we have
    \begin{equation}
        2ab\leq a^2+b^2 \label{trivial ineq}
    \end{equation}
\end{lemma}

\begin{lemma}[Cauchy-Schwarz Inequality]
    For $a,b\in\R^d$, we have
    \begin{equation}
        \inprd{a}{b}\leq \|a\|\|b\|\label{cauchy-schwarz}
    \end{equation}
\end{lemma}
\begin{lemma}[Projection properties]
    Given a nonempty convex closed set $U\subset \R^d$, the projection operator $P_U$ satisfies 
     \begin{align}
        \|P_U(x)-P_U(y)\| \leq \|x-y\| \hspace{0.5cm}&\forall x,y\in\R^d        \label{Nonexpansive} \\
        \inprd{x-P_U(x)}{u-P_U(x)}\leq 0 \hspace{0.5cm}&\forall u\in U, x\in\R^d \label{proj inprd} \\
        \|u-P_U(x)\|^2\leq \|u-x\|^2-\|x-P_U(x)\|^2 \hspace{0.5cm}&\forall u\in U, x\in\R^d \label{proj norm diff}
    \end{align}
\end{lemma}

We now prove two basic inequalities relating $(L_0, L_1)$-Lipschitz constants and the strong monotonicity parameter $\mu$:
\begin{lemma}
    If $F:U\rightarrow \R^d$ is a $\mu$-strongly monotone  $\alpha$-symmetric $(L_0,L_1)$-Lipschitz operator, then for any $x\in U$
    \begin{equation}
        \mu\leq \begin{cases}
            L_0+L_1\|F(x)\| &\ifaone \\
             K_0+K_1\|F(x)\|^\alpha &\ifasmall
        \end{cases}  \label{mu L0 L1}   
    \end{equation}
\end{lemma}
\begin{proof}
    From strong monotonicity and the Cauchy-Schwarz Inequality, we have
    \begin{align*}
    \mu\|x-y\|^2 &\leqfrom{\ref{muSC}} \inprd{F(x)-F(y)}{x-y} \\
    &\leqfrom{\ref{cauchy-schwarz}}\|F(x)-F(y)\|\|x-y\|  \\
    \end{align*}
    From here, we split into two cases depending on $\alpha$. First, if $\alpha =1$, we have
    \begin{align*}
   \mu\|x-y\|^2 &\leqfrom{\ref{a1 l0l1}} (L_0+L_1\|F(y)\|)\exp(L_1\|x-y\|)\|x-y\|^2 \\
    \Rightarrow \mu&\leq (L_0+L_1\|F(y)\|)\exp(L_1\|x-y\|)
\end{align*}
while if $\alpha \in(0,1)$, we have
\begin{align*}
    \mu\|x-y\|^2  &\leq(K_0+K_1\|F(y)\|^\alpha + K_2\|x-y\|^{\alpha/(1-\alpha)})\|x-y\|^2 \\
    \Rightarrow \mu&\leq (K_0+K_1\|F(y)\|^\alpha + K_2\|x-y\|^{\alpha/(1-\alpha)})
\end{align*}
In either case, the inequality holds for all $x,y$, so by taking $y\rightarrow x$ we attain the desired results.
\end{proof}

The following lemma will be useful in our analysis of monotone problems:
\begin{lemma}
    Let $z_k,\gamma_k$ be sequences with $\gamma_k\geq \ugamma$ for all $k$.  Define the average iterate $\bar{z}_K=\frac{\sum_{k=0}^{K-1}\gamma_k z_k}{\sum_{k=0}^{K-1}\gamma_k}$. Then if $\bar{z}_K\in B$,   
    \begin{equation}
        \supg \sum_{k=0}^{K-1} \gamma_k \inprd{F(y)}{z_k - y} \geq K\ugamma\gap{\bar{z}_K}\label{Gap lemma}
    \end{equation}
\end{lemma}
\begin{proof}
    Let $\Gamma_K=\sum_{k=0}^{K-1} \gamma_k \geq K\ugamma$. We have
\begin{align*}
    \supg \sum_{k=0}^{K-1} \gamma_k \inprd{F(y)}{z_k - y} &=  \Gamma_K\supg \sum_{k=0}^{K-1} \inprd{F(y)}{\frac{z_k\gamma_k}{\Gamma_K} - \frac{\gamma_ky}{\Gamma_K}} \\
    &= \Gamma_K\supg  \inprd{F(y)}{\sum_{k=0}^{K-1}\p{\frac{z_k\gamma_k}{\Gamma_K} - \frac{\gamma_ky}{\Gamma_K}}}\\
    &= \Gamma_K\supg  \inprd{F(y)}{\sum_{k=0}^{K-1}\p{\frac{z_k\gamma_k}{\Gamma_K}} -y} \\
    &=\Gamma_K\gap{\bar{z}_K} \\
    &\geq K\ugamma \gap{\bar{z}_K}
\end{align*}
where for the last step $\bar{z}_K\in B$ ensures that $\gap{\bar{z}_K}$ is nonnegative.
\end{proof}

Lastly, the following descent lemma will allow us to obtain our descent inequalities for \eg and \peg in a more unified fashion. Our proof closely follows \cite[Lemma. ~5.1]{pmlr-v235-vankov24a}, though we note such results are standard, and a similar derivation can be found in \cite[Lemma.~12.1.10]{facchinei_finite-dimensional_2003}. 

\begin{lemma}\label{lem:unified predescent}
   For all $y\in U$, the iterates of \eg and \peg algorithms satisfy 
    \begin{equation}
    \begin{aligned}
        \|x_{k+1}-y\|^2 &\leq \|x_k-y\|^2 - \|x_{k+1}-\xhatk\|^2-\|\xhatk -x_k\|^2 \\
        &\qquad + 2\gamma_k\|F(\xhatk)-F(\tilde{x}_k)\|\|x_{k+1}-\xhatk\| - 2\gamma_k\inprd{F(\xhatk)}{\xhatk - y}
    \end{aligned} \label{eg popov descent}
    \end{equation}
    where $\tilde{x}_k=\begin{cases}
        x_k &\text{ for \eg} \\
        \xhat_{k-1} &\text{ for \peg}
    \end{cases}$
\end{lemma}

\begin{proof}
    Let $F(\tilde{x}_k)$ denote $F(x_k)$ for \eg method and $F(\hat{x}_{k-1})$ for \peg method. Then we have the following:
    \begin{align}
        \|x_{k+1}-y\|^2 &= \|P_U(x_k-\gamma_k F(\xhatk)) - y\|^2 \notag\\
        &\leqfrom{\ref{proj norm diff}} \|x_k - \gamma_kF(\xhatk)-y\|^2 - \|x_k-\gamma_kF(\xhatk)-x_{k+1}\|^2 \notag\\
        &= \|x_k-y\|^2 - \|x_{k+1}-x_k\|^2-2\gamma_k\inprd{F(\xhatk)}{x_{k+1}-y} \notag\\
        &= \|x_k-y\|^2 - \|x_{k+1}-\xhatk+\xhatk -x_k\|^2-2\gamma_k\inprd{F(\xhatk)}{x_{k+1}-y} \notag\\
        &= \|x_k-y\|^2 - \|x_{k+1}-\xhatk\|^2-\|\xhatk -x_k\|^2 \notag\\
        &\qquad- 2\inprd{x_{k+1}-\xhatk}{\xhatk -x_k}-2\gamma_k\inprd{F(\xhatk)}{x_{k+1}-y} \notag\\
        &= \|x_k-y\|^2 - \|x_{k+1}-\xhatk\|^2 - \|\xhatk -x_k\|^2  \notag\\
        &\qquad- 2\inprd{x_{k+1}-\xhatk}{\xhatk -(x_k-\gamma_k F(\tilde{x}_k))} \notag\\
        &\qquad+2\gamma_k\inprd{x_{k+1}-\xhatk}{F(\tilde{x}_k)}-2\gamma_k\inprd{F(\xhatk)}{x_{k+1}-y} \notag\\
        &\leqfrom{\ref{proj inprd}} \|x_k-y\|^2 - \|x_{k+1}-\xhatk\|^2-\|\xhatk -x_k\|^2  \notag\\
        &\qquad + 2\gamma_k\inprd{x_{k+1}-\xhatk}{F(\tilde{x}_k)}-2\gamma_k\inprd{F(\xhatk)}{x_{k+1}-y}  \notag\\
        &= \|x_k-y\|^2 - \|x_{k+1}-\xhatk\|^2-\|\xhatk -x_k\|^2 \notag\\
        &\qquad + 2\gamma_k\inprd{F(\tilde{x}_k)-F(\xhatk)}{x_{k+1}-\xhatk} - 2\gamma_k\inprd{F(\xhatk)}{\xhatk - y} \notag\\
        &\leqfrom{\ref{cauchy-schwarz}} \|x_k-y\|^2 - \|x_{k+1}-\xhatk\|^2-\|\xhatk -x_k\|^2 \notag\\
        &\qquad + 2\gamma_k\|F(\xhatk)-F(\tilde{x}_k)\|\|x_{k+1}-\xhatk\| - 2\gamma_k\inprd{F(\xhatk)}{\xhatk - y} \notag
    \end{align}
    as desired.
\end{proof}
\newpage

\section{$(L_0,L_1)$-Lipschitz: Setting and Stepsizes}\label{app:stepsize section}
In this section, we first provide a discussion of the limiting behavior of the $(L_0,L_1)$-Lipschitz assumption and our stepsizes, and then we provide proofs of the Lemmas from this section.

\subsection{Limiting Behavior of $(L_0,L_1)$-Lipschitz Setting and Stepsizes}\label{app:limiting behavior}
In this subsection, we discuss the limiting behavior of the $\alpha$-symmetric $(L_0,L_1)$-Lipschitz setting and stepsizes, specifically as $L_1\rightarrow 0$ or $\alpha\rightarrow0$. First, note that as $L_1\rightarrow 0$, we have that Assumption~\ref{assmpt: l0 l1} simplifies to
\begin{equation*}
    \| F(x) - F(y)\| \leq L_0 \| x- y\| \quad \forall x, y \in U,
\end{equation*}
which is simply an $L_0$-Lipschitz assumption. Similarly, as $\alpha\rightarrow 0$, Assumption~\ref{assmpt: l0 l1} becomes
\begin{equation*}
    \| F(x) - F(y)\| \leq \left( L_0 + L_1  \right) \| x- y\| \quad \forall x, y \in U,
\end{equation*}
which is a $(L_0+L_1)$-Lipschitz assumption.

We now consider the limiting behavior of each of our stepsizes. First, recall from Lemma~\ref{lem: l0l1 smooth} that 
\begin{align*}
    K_0&=L_0(2^{\alpha^2/(1-\alpha)}+1) \\
    K_1&=L_12^{\alpha^2/(1-\alpha)}3^\alpha\\
    K_2&=L_1^{1/(1-\alpha)}2^{\alpha^2/(1-\alpha)}3^\alpha(1-\alpha)^{\alpha/(1-\alpha)}
\end{align*}
Thus, as $L_1\rightarrow 0$, $K_1,K_2\rightarrow 0$ while $K_0$ remains unchanged. As $\alpha\rightarrow 0$, we have $K_0\rightarrow2 L_0$ while $K_1,K_2\rightarrow L_1$. For \eg, this implies that as $L_1\rightarrow 0$,
\begin{equation*}
    \gamma_{k,\eg}\rightarrow\begin{cases}
        \frac{\nu}{L_0}&\ifaone \\
        \frac{\nu}{2K_0}&\ifasmall
    \end{cases}
\end{equation*}
When $\alpha\rightarrow 0$, we have
\begin{equation*}
    \gamma_{k,\eg}\rightarrow \frac{\nu}{4(L_0+L_1)}
\end{equation*}
For \popov, when $\alpha=1$ the derivations follow a similarly straightforward manner. However, when $\alpha\in (0,1)$, we have to treat the denominator of the stepsize more carefully due to the presence of the term
\begin{equation*}
    K_2\p{\|x_k-\xhat_{k-1}\|+\frac{\nu}{K_1}\|F(\xhat_{k-1})\|^{1-\alpha}}^{\nicefrac{\alpha}{(1-\alpha)}}
\end{equation*}
Here, as we take the limit of $L_1\rightarrow 0$, the $\frac{\nu}{K_1}$ term will become unbounded, but we will see that this term overall will still go to zero due to the presence of $K_2$. Note in the following $C$ represents all constant terms that do not include $L_1$.
\begin{align*}
    \lim_{L_1\rightarrow 0}    K_2\p{\|x_k-\xhat_{k-1}\|+\frac{\nu}{K_1}\|F(\xhat_{k-1})\|^{1-\alpha}}^{\nicefrac{\alpha}{(1-\alpha)}}&= C\lim_{L_1\rightarrow 0 }K_2\p{K_1^{-1}}^{\nicefrac{\alpha}{(1-\alpha)}}\\
    &= C\lim_{L_1\rightarrow 0 } L_1^{1/(1-\alpha)}\p{L_1^{-1}}^{\nicefrac{\alpha}{(1-\alpha)}} \\
    &= C\lim_{L_1\rightarrow 0}L_1^{\nicefrac{(1-\alpha)}{(1-\alpha)}} \\
    &= 0
\end{align*}
Using this and the definition of the \popov stepsizes from (\ref{Popov stepsizes}), we have that as $L_1\rightarrow 0$,
\begin{align*}
    \gamma_{k,\popov}\rightarrow\begin{cases}
        \frac{\nu}{L_0}&\ifaone \\
        \frac{\nu}{K_0} &\ifasmall
    \end{cases}
\end{align*}
Lastly, as $\alpha\rightarrow 0$, we obtain
\begin{equation*}
    \gamma_{k,\popov} \rightarrow \frac{\nu}{2(L_0+L_1)}
\end{equation*}
To summarize, for each of our stepsizes, as $L_1\rightarrow 0$ and we thereby approach the $L_0$-Lipschitz setting, our stepsizes approach stepsizes proportional to $\frac{1}{L_0}$, while as $\alpha\rightarrow 0$ and we approach the $(L_0+L_1)$-Lipschitz setting, our stepsizes become proportional to $\frac{1}{L_0+L_1}$. These are the expected stepsizes for $L$-Lipschitz VIs. As a final note, we only make asymptotic claims regarding $L_1\rightarrow 0$, since as noted above, when $\alpha \in (0,1)$ our stepsizes for \peg become undefined when $L_1=0$. Additionally, our proofs for both \eg and \peg use the fact that $L_1>0$.

    \subsection{Proof of Lemma~\ref{lem:lipschitz analogue}}\label{app:proof_of_lipschitz_analogue}
        We separate the proposition into two lemmas, one corresponding to \eg and the other to \peg.
        \begin{lemma}
        Let $F$ be $\alpha$-symmetric $(L_0,L_1)$-Lipschitz. Then iterates of \eg with stepsizes given by (\ref{EG Stepsizes}) with $\nu<1$ satisfy
        \begin{equation}
            \gamma_k\|F(x_k)-F(\xhatk)\| \leq C_\nu \|x_k-\xhatk\| \label{EG f bound}
        \end{equation}
        where $C_\nu=\nu e^\nu$ if $\alpha=1$ and $\nu$ if $\alpha \in(0,1)$.
    \end{lemma}
        \begin{proof}
        First, note from the \eg algorithm definition and since $x_k\in U$, we have
        \begin{align}
            \|x_k-\hat{x}_k\|&=\|P_U(x_k)-P_U(x_k-\gamma_kF(x_k))\| \notag\\
            &\leqfrom{\ref{Nonexpansive}} \gamma_k\|F(x_k)\|\label{EG half step f}
        \end{align}
      Now, starting with the case of $\alpha = 1$, we have
      \begin{align*}
          \gamma_k\|F(\xhatk)-F(x_k)\|&\leqfrom{\ref{a1 l0l1}} \gamma_k\|x_k-\hat{x}_k\|(L_0+L_1\|F(x_k)\|)\exp(L_1\|x_k-\hat{x}_k\|) \\
          &\eqfrom{\ref{EG Stepsizes}} \nu\|x_k-\xhatk\|\exp(L_1\|x_k-\xhatk\|) \\
          &\leqfrom{\ref{EG half step f}} \nu \|x_k-\xhatk\|\exp(L_1\gamma_k\|F(x_k)\|) \\ 
          &\eqfrom{\ref{EG Stepsizes}} \nu \|x_k-\xhatk\|\exp\p{\frac{\nu L_1\|F(x_k)\|}{L_0+L_1\|F(x_k)\|}}\\
          &\leqfrom{\ref{EG Stepsizes}} \nu e^\nu \|x_k-\xhatk \|
      \end{align*}
    
    We now turn to $\alpha \in (0,1)$. We have 
    \begin{align*}
    \gamma_k\|F(\xhatk)-F(x_k)\|
            &\leqfrom{\ref{a small l0l1}} \gamma_k(K_0+K_1\|F(x_k)\|^\alpha + K_2(\gamma_k\|F(x_k)\|)^{\alpha/(1-\alpha)})\|\xhatk-x_k\| \\ &= \p{\gamma_k(K_0+K_1\|F(x_k)\|^\alpha)+\gamma_k^{1/(1-\alpha)}K_2\|F(x_k)\|^{\alpha/(1-\alpha)}}\|\xhatk-x_k\| \\
            &= \frac{\nu(K_0+K_1\|F(x_k)\|^\alpha)\|\xhatk-x_k\|}{2K_0+(2K_1+2^{1-\alpha}K_2^{1-\alpha})\|F(x_k)\|^\alpha} \\
            &\qquad+ \frac{\nu^{1/(1-\alpha)}K_2\|F(x_k)\|^{\alpha/(1-\alpha)}\|\xhatk-x_k\|}{(2K_0+(2K_1+2^{1-\alpha}K_2^{1-\alpha})\|F(x_k)\|^\alpha)^{1/(1-\alpha)}} \\
            &\leq \p{\frac{\nu}{2}+\frac{\nu^{1/(1-\alpha)}}{2}}\|\xhatk-x_k\| \\
            &\leq \nu \|\xhatk-x_k\|
        \end{align*}
    where the last step uses $\nu< 1$.
    \end{proof}
    We note that the above lemma for \eg follows the intermediate steps of proofs in \cite{choudhury2025extragradient} where the unconstrained case of $\alpha$-symmetric $(L_0,L_1)$-Lipschitz operators was treated, though we must also use the nonexpansiveness of the projection since we are working in the constrained setting.
    
        \begin{lemma}
        Let $F$ be $\alpha$-symmetric $(L_0,L_1)$-Lipschitz. Then iterates of \peg with stepsizes given by \eqref{Popov stepsizes} satisfy
        \begin{equation}
        \gamma_k\|F(\xhatk)-F(\xhat_{k-1})\| \leq C_\nu\|\xhatk-\xhat_{k-1} \|\label{popov f bound}
        \end{equation}
        where $C_\nu=\nu e^\nu$ if $\alpha=1$ and $\nu$ if $\alpha\in (0,1)$, and $\nu<1$.
        \end{lemma}
        \begin{proof}
          First, note that by the projection inequality and the definition of \popov, we have 
        \begin{align}
            \|\xhatk-\xhat_{k-1}\|&=\|P_U(x_k-\gamma_kF(\xhat_{k-1}))-\xhat_{k-1}\| \notag\\
            &\leqfrom{\ref{Nonexpansive}} \|x_k-\xhat_{k-1} - \gamma_kF(\xhat_{k-1})\|\notag \\
            &\leq \|x_k-\xhat_{k-1}\|+\gamma_k\|F(\xhat_{k-1})\|\label{popov half step diff}
        \end{align}
        We now treat each $\alpha$ separately. Starting with $\alpha = 1$:
            \begin{align}
            \gamma_k&\|F(\xhatk)-F(\xhat_{k-1})\| \leqfrom{\ref{a1 l0l1}} \gamma_k\|\xhatk-\xhat_{k-1}\|(L_0+L_1\|F(\xhat_{k-1})\|)\exp(L_1\|\xhatk-\xhat_{k-1}\|) \notag\\
            &\leqfrom{\ref{popov half step diff}} \gamma_k\|\xhatk-\xhat_{k-1}\|(L_0+L_1\|F(\xhat_{k-1})\|)\exp(L_1\|x_k-\xhat_{k-1}\|+L_1\gamma_k\|F(\xhat_{k-1})\|)\notag\\ 
            &\leq \gamma_k\|\xhatk-\xhat_{k-1}\|(L_0+L_1\|F(\xhat_{k-1})\|)\exp(L_1\|x_k-\xhat_{k-1}\|+\nu)\notag\\
            &= \nu e^\nu\|\xhatk-\xhat_{k-1}\| \label{f bound a1 popov}
        \end{align}
        Now, for $\alpha \in (0,1)$:
        \begin{align}
            \gamma_k &\|F(\xhatk)-F(\xhat_{k-1})\| \leqfrom{\ref{a small l0l1}} \gamma_k \|\xhatk-\xhat_{k-1}\|\p{K_0+K_1\|F(\xhat_{k-1})\|^\alpha + K_2\|\xhatk-\xhat_{k-1}\|^{\alpha/(1-\alpha)}}\notag\\
            &\leqfrom{\ref{popov half step diff}} \resizebox{0.9\linewidth}{!}{$\gamma_k \|\xhatk-\xhat_{k-1}\|\p{K_0+K_1\|F(\xhat_{k-1})\|^\alpha + K_2\p{\|x_k-\xhat_{k-1}\|+\gamma_k\|F(\xhat_{k-1})\|}^{\alpha/(1-\alpha)}}$}\notag\\
            &\leq \resizebox{0.925\linewidth}{!}{$\gamma_k \|\xhatk-\xhat_{k-1}\|\p{K_0+K_1\|F(\xhat_{k-1})\|^\alpha + K_2\p{\|x_k-\xhat_{k-1}\|+\frac{\nu}{K_1}\|F(\xhat_{k-1})\|^{1-\alpha}}^{\alpha/(1-\alpha)}}$}\notag\\
            &=\nu\|\xhatk-\xhat_{k-1}\|\label{f bound 2}
        \end{align}
    \end{proof}

    \subsection{Proof of Lemma~\ref{lem:EG step lower bound}}
    \begin{lemma}
        Let $F$ be $\alpha$-symmetric $(L_0,L_1)$-Lipschitz. Then if the iterates of \eg with stepsizes given by (\ref{EG Stepsizes}) satisfy $\|x_{k+1}-x^*\|\leq\|x_0-x^*\|$ for all $k\in 0,1,2,...$ we have 
        \begin{equation}
            \gamma_k\geq \ugamma = \begin{cases}
                \frac{\nu}{L_0+L_1\p{ \|x_0-x^*\|(L_0+L_1\|F(x^*)\|)\exp(L_1\|x_0-x^*\|) + \|F(x^*)\|}} &\ifaone \\
                \frac{\nu}{2K_0+(2K_1 + 2^{1-\alpha}K_2^{1-\alpha})\overF^{\alpha}} &\ifasmall
            \end{cases} \label{EG step lower bound}
        \end{equation}
        where $\overF=\|x_0-x^*\|(K_0+K_1\|F(x^*)\|^\alpha + K_2\|x_0-x^*\|^{\alpha/(1-\alpha)}) + \|F(x^*)\|$.
    \end{lemma}
    \begin{proof}
      For each value of $\alpha$, we upper bound $\|F(x_k)\|$ in order to lower bound $\gamma_k$ , starting with
      \begin{equation*}
          \|F(x_k)\|\leq\|F(x_k)-F(x^*)\|+\|F(x^*)\|
      \end{equation*}
      and using the smoothness of the operator to bound the first term. 
        For $\alpha=1$, this yields an upper bound on the norm of
        \begin{align}
        \|F(x_k)\| &\leqfrom{\ref{a1 l0l1}} \|x_k-x^*\|(L_0+L_1\|F(x^*)\|)\exp(L_1\|x_k-x^*\|) + \|F(x^*)\| \notag \\
        &\leq \|x_0-x^*\|(L_0+L_1\|F(x^*)\|)\exp(L_1\|x_0-x^*\|) + \|F(x^*)\| \label{eg f upper bound}
        \end{align}
    Substituting this upper bound on $\|F(x_k)\|$ into the stepsize expression for \eg (\ref{EG Stepsizes}), we find
    \begin{align*}
        \gamma_k&\geq \frac{\nu}{L_0+L_1\p{ \|x_0-x^*\|(L_0+L_1\|F(x^*)\|)\exp(L_1\|x_0-x^*\|) + \|F(x^*)\|}}
        \end{align*}
        Similarly, for $\alpha \in (0,1)$, we have
        \begin{align*}
         \|F(x_k)\|   &\leqfrom{\ref{a small l0l1}} \|x_k-x^*\|(K_0+K_1\|F(x^*)\|^\alpha + K_2\|x_k-x^*\|^{\alpha/(1-\alpha)}) + \|F(x^*)\| \\
        &\leq \|x_0-x^*\|(K_0+K_1\|F(x^*)\|^\alpha + K_2\|x_0-x^*\|^{\alpha/(1-\alpha)}) + \|F(x^*)\| \\
        &=\overF
        \end{align*}
        as our upper bound for the operator norm, which gives us the desired stepsize lower bound:
        \begin{align*}
         \gamma_k&\geq \frac{\nu}{2K_0+(2K_1 + 2^{1-\alpha}K_2^{1-\alpha})\overF^{\alpha}}
        \end{align*}
    \end{proof}

\subsection{Proof of Lemma~\ref{lem:popov step lower bound}}\label{app: proof of lem popov step lower bound}
    \begin{lemma}
        Let $F$ be $\alpha$-symmetric $(L_0,L_1)$-Lipschitz. Then if the iterates of \peg with stepsizes given by (\ref{Popov stepsizes}) and $\nu$ chosen so that $C_\nu< 1$  satisfy $R_{k+1}\leq R_0$ for all $k\in 0,1,2,...$, we have
    \begin{equation}
        \gamma_k\geq\ugamma=\begin{cases}
            \frac{\nu}{(L_0+L_1(\sqrt{2}\|x_0-x^*\|(L_0+L_1\|F(x^*)\|)\exp(\sqrt{2}L_1\|x_0-x^*\|) + \|F(x^*)\|))\exp\p{\frac{\nu^2 e^\nu}{1-\nu e^\nu}}} &\ifaone\\
            \frac{\nu}{K_0+\p{K_1 + K_2\p{\frac{\nu^2}{\p{K_1+K_2\p{\frac{\nu}{K_1}}^{\alpha/(1-\alpha)}}(1-\nu)}+\frac{\nu}{K_1}}^{\alpha/(1-\alpha)}}\overF^{\alpha}} &\ifasmall
        \end{cases}\label{popov step lower bound}
    \end{equation}
        where $R_k=\|x_k-x^*\|^2+\|x_k-\xhat_{k-1}\|^2$  and $\overF= \sqrt{2}\|x_0-x^*\|(K_0+K_1\|F(x^*)\|^\alpha+2^{\alpha/(2(1-\alpha))}K_2\|x_0-x^*\|^{\alpha/(1-\alpha)}) + \|F(x^*)\|$.
    \end{lemma}
    
    \begin{proof}
    To lower bound our stepsizes, we must upper bound each of $\|F(\xhat_{k-1})\|$ and $\|x_k-\xhat_{k-1}\|$. Starting with upper bounding $\|F(\xhat_{k-1})\|$, we have by the definition of $R_k$ and by initializing the algorithm with $\xhat_{-1}=x_0$ that
    \begin{equation}
        \|x_k-x^*\|^2+\|x_k-\xhat_{k-1}\|^2 \leq \|x_0-x^*\|^2 \label{popov descent exp}
    \end{equation}
    and similarly,
    \begin{align}
        \|\xhat_{k-1} - x^*\|^2 &\leqfrom{\ref{young ineq}} 2\|x_k-x^*\|^2 +2 \|x_k-\xhat_{k-1}\|^2 \notag\\
        &\leqfrom{\ref{popov descent exp}} 2\|x_0-x^*\|^2 \label{popov half from sol}
    \end{align}
    
    We now treat the cases of $\alpha=1$ and $\alpha\in (0,1)$ separately. In the case of $\alpha=1$, we have
    \begin{align}
        \|F(\xhat_{k-1})\|&\leq \|F(\xhat_{k-1})-F(x^*)\|+\|F(x^*)\| \notag\\
        &\leqfrom{\ref{a1 l0l1}} \|\xhat_{k-1}-x^*\|(L_0+L_1\|F(x^*)\|)\exp(L_1\|\xhat_{k-1}-x^*\|) + \|F(x^*)\|\notag \\
        &\leqfrom{\ref{popov half from sol}} \sqrt{2}\|x_0-x^*\|(L_0+L_1\|F(x^*)\|)\exp(\sqrt{2}L_1\|x_0-x^*\|) + \|F(x^*)\| \label{popov final f bound a1}
    \end{align}
    As for when $\alpha \in (0,1)$,
    \begin{align}
        \|F&(\xhat_{k-1})\|\leq \|F(\xhat_{k-1})-F(x^*)\|+\|F(x^*)\| \notag\\
        &\leqfrom{\ref{a small l0l1}} \|\xhat_{k-1}-x^*\|(K_0+K_1\|F(x^*)\|^\alpha+K_2\|\xhat_{k-1}-x^*\|^{\alpha/(1-\alpha)}) + \|F(x^*)\| \notag\\
        &\leqfrom{\ref{popov half from sol}} \sqrt{2}\|x_0-x^*\|(K_0+K_1\|F(x^*)\|^\alpha+2^{\alpha/(2(1-\alpha))}K_2\|x_0-x^*\|^{\alpha/(1-\alpha)}) + \|F(x^*)\|\notag\\
        &=\overF \label{popov final f bound asmall}
    \end{align}
    With these bounds on our operator norms established, we move onto bounding $\|x_k-\xhat_{k-1}\|$.  We have for any value of $\alpha$ that
    \begin{align}
        \|x_k-\xhat_{k-1}\| &\leqfrom{\ref{Nonexpansive}} \gamma_{k-1}\|F(\xhat_{k-1})-F(\xhat_{k-2})\| \notag\\
        &\leqfrom{\ref{popov f bound}} C_{\nu}\|\xhat_{k-1}-\xhat_{k-2}\| \notag\\
        &\leqfrom{\ref{popov half step diff}} C_{\nu}\p{\|x_{k-1}-\xhat_{k-2}\|+\gamma_{k-1}\|F(\xhat_{k-2})\|} \label{popov half step bound intermediate}
        \end{align}
    We now must bound $\gamma_{k-1}\|F(\xhat_{k-2})\|$ separately for each $\alpha$. In the case where $\alpha=1$, we have
    \begin{align*}
            \gamma_{k-1}\|F(\xhat_{k-2})\| &= \frac{\nu \|F(\xhat_{k-2})\|}{(L_0+L_1\|F(\xhat_{k-2})\|)\exp(L_1\|x_{k-1}-\xhat_{k-2}\|)} \\
            &\leq \frac{\nu}{L_1}
    \end{align*}
    and similarly for $\alpha \in (0,1)$, we have
    \begin{align*}
     \gamma_{k-1}\|F(\xhat_{k-2})\| &=   \frac{\nu\|F(\xhat_{k-2})\|}{K_0+K_1\|F(\xhat_{k-2})\|^\alpha + K_2(\|x_{k-1}-\xhat_{k-2}\|+\frac{\nu}{K_1}\|F(\xhat_{k-2})\|^{1-\alpha})^\frac{\alpha}{1-\alpha}} \\
     &\leq \frac{\nu\|F(\xhat_{k-2})\|}{\p{K_1+K_2\p{\frac{\nu}{K_1}}^{\alpha/(1-\alpha)}}\|F(\xhat_{k-2})\|^\alpha} \\
     &=\frac{\nu \|F(\xhat_{k-2})\|^{1-\alpha}}{K_1+K_2\p{\frac{\nu}{K_1}}^{\alpha/(1-\alpha)}} \\
     &\leq \frac{\nu \overF^{1-\alpha} }{K_1+K_2\p{\frac{\nu}{K_1}}^{\alpha/(1-\alpha)}}
    \end{align*}
    
    Defining $M=\begin{cases}
        \frac{\nu}{L_1}\ifaone \\
        \frac{\nu \overF^{1-\alpha}}{K_1+K_2\p{\frac{\nu}{K_1}}^{\alpha/(1-\alpha)}} \ifasmall
    \end{cases}$ we have from (\ref{popov half step bound intermediate}) 
    
    \begin{equation*}
        \|x_k-\xhat_{k-1}\| \leq  C_{\nu}\p{\|x_{k-1}-\xhat_{k-2}\|+M} 
    \end{equation*}
    We then unwind the recursion:
    \begin{align}
      \|x_k-\xhat_{k-1}\|  &\leq C_{\nu}^k\|x_0-\xhat_{-1}\|+\sum_{i=1}^k MC_\nu^i\notag\\
      &= M\sum_{i=1}^k C_\nu^i \notag\\
      &\leq \frac{MC_\nu}{1-C_\nu}\notag\\
        &=\begin{cases}
            \frac{\nu^2 e^\nu}{L_1(1-\nu e^\nu)}\ifaone \\
            \frac{\nu^2 \overF^{1-\alpha} }{\p{K_1+K_2\p{\frac{\nu}{K_1}}^{\alpha/(1-\alpha)}}(1-\nu)} \ifasmall
        \end{cases} \label{popov half step bound} 
    \end{align}
    where in the second to last line we used that $C_\nu<1$. 
    
    Now that we have upper bounded both $\|F(\xhat_{k-1})\|$ and $\|x_k-\xhat_{k-1}\|$, we can lower bound our stepsizes.
    
    For $\alpha=1$, we have
    \begin{align*}
        \gamma_k&=\frac{\nu}{(L_0+L_1\|F(\xhat_{k-1})\|)\exp(L_1\|x_k-\xhat_{k-1}\|)} \\
        &\geqfrom{\ref{popov half step bound}} \frac{\nu}{(L_0+L_1\|F(\xhat_{k-1})\|)\exp\p{\frac{\nu^2 e^\nu}{1-\nu e^\nu}}} \\
        &\geqfrom{\ref{popov final f bound a1}} \resizebox{0.9\linewidth}{!}{$\frac{\nu}{(L_0+L_1(\sqrt{2}\|x_0-x^*\|(L_0+L_1\|F(x^*)\|)\exp(\sqrt{2}L_1\|x_0-x^*\|) + \|F(x^*)\|))\exp\p{\frac{\nu^2 e^\nu}{1-\nu e^\nu}}}$}
    \end{align*}
    While for $\alpha \in (0,1)$, we find 
    \begin{align*}
        \gamma_k &=\frac{\nu}{K_0+K_1\|F(\xhat_{k-1})\|^\alpha + K_2(\|x_k-\xhat_{k-1}\|+\frac{\nu}{K_1}\|F(\xhat_{k-1})\|^{1-\alpha})^\frac{\alpha}{1-\alpha}} \\
        &\geqfrom{\ref{popov final f bound asmall}} \frac{\nu}{K_0+K_1\overF^\alpha + K_2\p{\|x_k-\xhat_{k-1}\|+\frac{\nu}{K_1}\overF^{1-\alpha}}^\frac{\alpha}{1-\alpha}}\\
        &\geqfrom{\ref{popov half step bound}} \frac{\nu}{K_0+\p{K_1 + K_2\p{\frac{\nu^2}{\p{K_1+K_2\p{\frac{\nu}{K_1}}^{\alpha/(1-\alpha)}}(1-\nu)}+\frac{\nu}{K_1}}^\frac{\alpha}{1-\alpha}}\overF^{\alpha}}
    \end{align*}
    
    \end{proof}

\newpage

\section{Proofs of Convergence Results}\label{app:convergence}

In this section, we provide the proofs for our main convergence theorems from Section \ref{sec:convergence analysis}. 
Before treating the individual cases of strongly monotone and monotone operators, we provide a general descent lemma for each of \eg and \peg that can be used in either setting. We first provide the descent lemma for \eg.

    \begin{lemma}\label{lem: eg descent}
          Let $F$ be $\alpha$-symmetric $(L_0,L_1)$-Lipschitz and $\mu$-strongly monotone. Then iterates of \eg with stepsizes given by (\ref{EG Stepsizes}) with $\nu<1$ satisfy
          \begin{equation}
              \|x_{k+1}-y\|^2 \leqfrom{\ref{Young-like}} \p{1-\frac{2\gamma_k\mu}{1+\beta}}\|x_k-y\|^2-\p{1-C_\nu^2- \frac{2\gamma_k\mu}{\beta}}\|x_k-\xhatk\|^2 - 2\gamma_k\inprd{F(y)}{\hat{x}_k-y}   \label{EG predescent}
          \end{equation}
              where $y\in U$, $\beta>0$, and $C_\nu =\begin{cases}
            \nu e^\nu \text{ if $\alpha=1$} \\
            \nu \text{ if $\alpha \in (0,1)$}
        \end{cases}$.
    \end{lemma}
    \begin{proof}
    First, from Lemma~\ref{lem:unified predescent}, we have
        \begin{align}
            \|x_{k+1}-y\|^2 &\leqfrom{\ref{eg popov descent}} \|x_k-y\|^2 - \|x_{k+1}-\xhatk\|^2-\|\xhatk -x_k\|^2 \notag\\
            &\qquad + 2\gamma_k\|F(\xhatk)-F(x_k)\|\|x_{k+1}-\xhatk\| - 2\gamma_k\inprd{F(\xhatk)}{\xhatk - y} \notag\\
            &\leqfrom{\ref{trivial ineq}} \|x_k-y\|^2 -\|\xhatk -x_k\|^2  + \gamma_k^2\|F(\xhatk)-F(x_k)\|^2 - 2\gamma_k\inprd{F(\xhatk)}{\xhatk - y} \label{vankov eg lemma}
        \end{align}
        Up to here, we have reproduced standard results for \eg for completeness (see \cite[Lemma 5.1]{pmlr-v235-vankov24a}).
    
        We now use our lemmas specific to the $\mu$-strongly monotone $(L_0, L_1)$-Lipschitz setting to get our final descent lemma. For our proofs in the monotone setting, we will simply take $\mu=0$.
        \begin{align*}
            \|x_{k+1}&-y\|^2 \leqfrom{\ref{vankov eg lemma}} \|x_k-y\|^2-\|x_k-\hat{x}_k\|^2\\
            &\qquad-2\gamma_k\inprd{F(\hat{x}_k)}{\hat{x}_k-y}+\gamma_k^2\|F(x_k)-F(\hat{x}_k)\|^2 \notag\\
            &\leqfrom{\ref{EG f bound}} \|x_k-y\|^2-(1-C_\nu^2)\|x_k-\hat{x}_k\|^2-2\gamma_k\inprd{F(\hat{x}_k)}{\hat{x}_k-y} \\
            &= \|x_k-y\|^2-(1-C_\nu^2)\|x_k-\hat{x}_k\|^2 \\
            &\qquad-2\gamma_k\inprd{F(\hat{x}_k) - F(y)}{\hat{x}_k-y} - 2\gamma_k\inprd{F(y)}{\hat{x}_k-y} \\
            &\leqfrom{\ref{muSC}} \|x_k-y\|^2-(1-C_\nu^2)\|x_k-\xhatk\|^2 -2\gamma_k\mu\|\xhatk-y\|^2  - 2\gamma_k\inprd{F(y)}{\hat{x}_k-y}
        \end{align*}

        We conclude the proof by applying Young's Lemma, so that for any $\beta>0$,
        \begin{align*}
            \|x_{k+1}&-y\|^2 \leqfrom{\ref{Young-like}} \p{1-\frac{2\gamma_k\mu}{1+\beta}}\|x_k-y\|^2-\p{1-C_\nu^2- \frac{2\gamma_k\mu}{\beta}}\|x_k-\xhatk\|^2 - 2\gamma_k\inprd{F(y)}{\hat{x}_k-y} 
        \end{align*}  
    \end{proof}

    \begin{lemma}
        Let $F$ be $\alpha$-symmetric $(L_0,L_1)$-Lipschitz and $\mu$-strongly monotone. Then iterates of \peg with the stepsizes given by (\ref{Popov stepsizes}) satisfy
    \begin{equation}
    \begin{aligned}
        R_{k+1}(y)&\leq \p{1-\frac{2\gamma_k\mu}{1+\beta}}R_k(y)-\p{1-4C_\nu^2-\frac{2\gamma_k\mu}{\beta}}\|x_k-\xhatk\|^2 \notag\\
        &\qquad+\p{4C_\nu^2-1+\frac{2\gamma_k\mu}{1+\beta}}\|x_k-\xhat_{k-1}\|^2  -2\gamma_k\inprd{F(y)}{\xhatk-y}
    \end{aligned}
         \label{popov predescent}
    \end{equation}
      where $y\in U$, $\beta>0$, and $R_k(y)=\|x_k-y\|^2+\|x_k-\xhat_{k-1}\|^2$.
    \end{lemma}
    \begin{proof}
        Again starting with Lemma~\ref{lem:unified predescent}, we have,
           \begin{align}
            \|x_{k+1}&-y\|^2 \leqfrom{\ref{eg popov descent}} \|x_k-y\|^2 - \|x_{k+1}-\xhatk\|^2-\|\xhatk -x_k\|^2 \notag\\
            &\qquad + 2\gamma_k\|F(\xhatk)-F(\xhat_{k-1})\|\|x_{k+1}-\xhatk\| - 2\gamma_k\inprd{F(\xhatk)}{\xhatk - y} \notag\\
            &= \|x_k-y\|^2 - \|x_{k+1}-\xhatk\|^2-\|\xhatk -x_k\|^2 - 2\gamma_k\inprd{F(\xhatk)}{\xhatk - y}\notag\\
            &\qquad + 2\gamma_k\|F(\xhatk)-F(\xhat_{k-1})\|\|P_U(x_k-\gamma_k F(\xhatk)) - P_U(x_k-\gamma_k F(\xhat_{k-1}))\| \notag\\
            &\leqfrom{\ref{Nonexpansive}} \|x_k-y\|^2 -\|\xhatk - x_k\|^2   - \|x_{k+1}-\xhatk\|^2 \notag\\
            &\qquad- 2\gamma_k\inprd{F(\xhatk)}{\xhatk - y} + 2\gamma_k^2\|F(\xhatk)-F(\xhat_{k-1})\|^2 \label{popov lemma}
        \end{align}
        At this point, we have recovered \cite[Lemma 6.1]{pmlr-v235-vankov24a} for completeness. Rearranging and defining $R_k(y)=\|x_k-y\|^2+\|x_k-\xhat_{k-1}\|^2$, we have
        \begin{align}
        R_{k+1}(y) &\leq \|x_k-y\|^2 - \|x_k-\xhatk\|^2-2\gamma_k\inprd{F(\xhatk)}{\xhatk-y}+2\gamma_k^2\|F(\xhatk)-F(\xhat_{k-1})\|^2 \notag \\
        &\leqfrom{\ref{popov f bound}}  \|x_k-y\|^2 - \|x_k-\xhatk\|^2-2\gamma_k\inprd{F(\xhatk)}{\xhatk-y} +2 C_\nu^2\|\xhatk-\xhat_{k-1}\|^2\notag \\
        &=\|x_k-y\|^2 - \|x_k-\xhatk\|^2-2\gamma_k\inprd{F(\xhatk)}{\xhatk-y} +2 C_\nu^2\|\xhatk-x_k+x_k-\xhat_{k-1}\|^2 \notag\\
        &\leq \|x_k-y\|^2-(1-4C_\nu^2)\|x_k-\xhatk\|^2+4C_\nu^2\|x_k-\xhat_{k-1}\|^2 -2\gamma_k\inprd{F(\xhatk)}{\xhatk-y} \notag\\
        &=\|x_k-y\|^2-(1-4C_\nu^2)\|x_k-\xhatk\|^2+4C_\nu^2\|x_k-\xhat_{k-1}\|^2 \notag\\ &\qquad-2\gamma_k\inprd{F(\xhatk)-F(y)}{\xhatk-y} -2\gamma_k\inprd{F(y)}{\xhatk-y}\notag \\
        &\leqfrom{\ref{muSC}} \|x_k-y\|^2-(1-4C_\nu^2)\|x_k-\xhatk\|^2+4C_\nu^2\|x_k-\xhat_{k-1}\|^2 \notag \\
        &\qquad-2\gamma_k\mu\|\xhatk-y\|^2 -2\gamma_k\inprd{F(y)}{\xhatk-y} \notag 
    \end{align}

    Setting $\beta >0$, we now apply Young's Lemma and then rearrange terms to achieve our final result.
    \begin{align}
        R_{k+1}(y)&\leqfrom{\ref{Young-like}} \p{1-\frac{2\gamma_k\mu}{1+\beta}}\|x_k-y\|^2 - \p{1-4C_\nu^2 -\frac{2\gamma_k\mu}{\beta}}\|x_k-\xhatk\|^2 \notag \\
        &\qquad+ 4C_\nu^2\|x_k-\xhat_{k-1}\|^2 -2\gamma_k\inprd{F(y)}{\xhatk-y} \notag\\
        &= \p{1-\frac{2\gamma_k\mu}{1+\beta}}R_k(y)-\p{1-4C_\nu^2-\frac{2\gamma_k\mu}{\beta}}\|x_k-\xhatk\|^2 \notag\\
        &\qquad+\p{4C_\nu^2-1+\frac{2\gamma_k\mu}{1+\beta}}\|x_k-\xhat_{k-1}\|^2  -2\gamma_k\inprd{F(y)}{\xhatk-y} \notag 
    \end{align} 
    \end{proof}

\subsection{Proofs of Results for Monotone VIs}

\subsubsection{Proof of Theorem~\ref{thm: EG monotone}}
    \begin{theorem}
        Let $F$ be $\alpha$-symmetric $(L_0, L_1)$-Lipschitz and monotone. Additionally let $\xbark=\frac{\sum_{k=0}^{K-1}\gamma_k\xhatk}{\sum_{k=0}^{K-1}\gamma_k}$ and $x^*\in U^*$. Then the averaged iterates $\xbark$ of \eg with stepsizes given by (\ref{EG Stepsizes}) and $\nu$ chosen so that $C_\nu< 1$  satisfy
         \begin{equation*}
          \gap{\xbark}\leq \p{2+\frac{1}{\sqrt{1-C_\nu^2}}}^2\frac{\|x_0-x^*\|^2}{2\ugamma K}
        \end{equation*}
        where $\ugamma$ is defined in Lemma~\ref{lem:EG step lower bound} and $B=\{y\in U:\|y-x^*\|^2\leq  (1+(1-C_\nu^2)^{-\nicefrac{1}{2}})^2\|x_0-x^*\|^2\}$. Further, $\xbark\in B$.
    \end{theorem}
    
    \begin{proof}
        Starting with (\ref{EG predescent}) with $\mu=0$ from monotonicity, we have
        \begin{align}
              \|x_{k+1}-y\|^2 &\leq \|x_k-y\|^2-(1-C_\nu^2)\|x_k-\xhatk\|^2 - 2\gamma_k\inprd{F(y)}{\hat{x}_k-y} \label{EG monotone predescent}
        \end{align}
        We use this inequality twice. First, taking $y=x^*$, we have
        \begin{align}
            \|x_{k+1}-x^*\|^2 &\leqfrom{\ref{VIP Solution}}  \|x_k-x^*\|^2-(1-C_\nu^2)\|x_k-\xhatk\|^2 \label{eg pre half bound}\\
            &\leq \|x_0-x^*\|^2 \label{eg bounded iterates}
        \end{align}
      where the last inequality comes from our choice of $\nu$. This indicates that our iterates are bounded and allows us to apply (\ref{EG step lower bound}). Note also from rearranging (\ref{eg pre half bound}), we have
      \begin{align}
         (1-C_\nu^2)\|x_k-\xhatk\|^2  &\leq \|x_k-x^*\|^2-\|x_{k+1}-x^*\|^2 \notag\\
         &\leqfrom{\ref{eg bounded iterates}} \| x_0-x^*\|^2 \notag \\
         \Rightarrow \|x_k-\xhatk\|^2&\leq  \frac{\|x_0-x^*\|^2}{1-C_\nu^2} \label{eg half bound}
      \end{align}
      From Young's Inequality and combining (\ref{eg bounded iterates}) and (\ref{eg half bound}), we thus have for any $\beta> 0$
      \begin{align*}
          \|\xhatk-x^*\|^2 &\leqfrom{\ref{full young}} (1+\beta^{-1})\|x_k-x^*\|^2+(1+\beta)\|x_k-\xhatk\|^2 \\  
          &\leqfrom{\ref{eg half bound}} \p{1+\frac{1}{1-C_\nu^2}+\frac{1}{\beta}+\frac{\beta}{1-C_\nu^2}}\|x_0-x^*\|^2 
          \end{align*}
          We choose $\beta=\sqrt{1-C_\nu^2}$ to minimize the upper bound, obtaining the following
          
          \begin{align*}
          \|\xhatk-x^*\|^2&\leq \p{1+\frac{1}{\sqrt{1-C_\nu^2}}}^2\|x_0-x^*\|^2
      \end{align*}
      
      For convenience, let $r=1+\frac{1}{\sqrt{1-C_\nu^2}}$ and $B$ be the set $\set{y\in U:\|y-x^*\|^2\leq r^2\|x_0-x^*\|^2}$. Note that since $\xbark$ is a convex combination of $\xhat_i\in B$ for $i\in \set{0,1,...,K-1}$, $\xbark\in B$.
        
        Now returning to the case of general $y$, we rearrange and telescope to obtain
        \begin{align*}
          2\gamma_k  \inprd{F(y)}{\xhatk-y} &\leq \|x_k-y\|^2-\|x_{k+1}-y\|^2 \\
            \Rightarrow 2\sum_{k=0}^{K-1} \gamma_k\inprd{F(y)}{\xhatk-y} &\leq \|x_{0}-y\|^2- \|x_{K}-y\|^2  \\
            &\leq \|x_0-y\|^2
        \end{align*}
       
        We take $\supg$ over both sides and applying (\ref{Gap lemma}) to obtain
        \begin{align}
              \supg\|x_{0}-y\|^2 &\geq 2\supg\sum_{k=0}^{K-1} \gamma_k\inprd{F(y)}{\xhatk-y} \notag\\
              &\geqfrom{\ref{Gap lemma}}  2\ugamma K\gap{\xbark}\label{eg gap setup}
        \end{align}
All that remains is to upper bound $ \supg\|x_{0}-y\|^2$. To do this, we again apply Young's inequality with some $\beta>0$:
\begin{align*}
    \supg\|x_0-y\|^2 &\leqfrom{\ref{full young}} \supg\p{(1+\beta^{-1})\|x_0-x^*\|^2+ (1+\beta)\|x^*-y\|^2} \\
    &\leq \p{1+\beta^{-1}}\|x_0-x^*\|^2+(1+\beta)r^2\|x_0-x^*\|^2 \\
    &=(1+r^2+\beta^{-1}+\beta r^2)\|x_0-x^*\|^2
\end{align*}
        We can readily find that the value for $\beta$ that minimizes this expression is $\beta=r^{-1}$, which gives us
        \begin{equation}
            \supg \|x_0-y\|^2\leq \p{1+r}^2\|x_0-x^*\|^2\label{eg sup bound}
        \end{equation}
       Rearranging (\ref{eg gap setup}) and substituting in (\ref{eg sup bound}), we complete the proof:
        \begin{align*}
            \gap{\xbark} &\leqfrom{\ref{eg gap setup}} \frac{\supg \|x_0-y\|^2}{2\ugamma K} \\
            &\leqfrom{\ref{eg sup bound}} \frac{(1+r)^2\|x_0-x^*\|^2}{2\ugamma K} \\
            &=\p{2+\frac{1}{\sqrt{1-C_\nu^2}}}^2\frac{\|x_0-x^*\|^2}{2\ugamma K}
        \end{align*}
    \end{proof}

    \subsubsection{Proof of Theorem~\ref{thm: popov monotone}}
   \begin{theorem}
        Let $F$ be $\alpha$-symmetric $(L_0,L_1)$-Lipschitz and monotone. Additionally, let $\xbark=\frac{\sum_{k=0}^{K-1}\gamma_k\xhatk}{\sum_{k=0}^{K-1}\gamma_k}$ and $x^*\in U^*$. Then the averaged iterates $\xbark$ of \peg with stepsizes given by (\ref{Popov stepsizes}) and $\nu$ chosen so that $C_\nu\leq\frac{1}{2}$ satisfy 
        \begin{equation*}
            \gap{\xbark}\leq \frac{(3+2\sqrt{2})\|x_0-x^*\|^2}{2\ugamma K}
        \end{equation*}
        where $\ugamma$ is defined in Lemma~\ref{lem:popov step lower bound} and $B=\set{y\in U:\|y-x^*\|^2\leq 2\|x_0-x^*\|^2}$. Further, $\xbark\in B$.
    \end{theorem}

    \begin{proof}
    From (\ref{popov predescent}) with $\mu=0$, we have
    \begin{align}
         R_{k+1}(y)\leq R_k(y)-(1-4C_\nu^2)\|x_k-\xhatk\|^2-(1-4C_\nu^2)\|x_k-\xhat_{k-1}\|^2 -2\gamma_k\inprd{F(y)}{\xhatk-y} \label{popov monotone predescent}
    \end{align}   
    Like with \eg, we use this last inequality twice. First, in order to lower bound $\gamma_k$, we choose $y=x^*$, obtaining \begin{align}
        R_{k+1}(x^*)&\leq R_k(x^*)-(1-4C_\nu^2)\|x_k-\xhatk\|^2-(1-4C_\nu^2)\|x_k-\xhat_{k-1}\|^2  \label{popov monotone descent} \\
        &\leq R_k(x^*) \notag\\
        &\leq R_0(x^*) \notag
    \end{align}
  where in the second step we used that $C_\nu\leq\half$. This result allows us to apply (\ref{popov step lower bound}). Note from here we can also obtain the following bound obtained as an intermediate step in Lemma~\ref{lem:popov step lower bound}:
  \begin{equation*}
       \|\xhat_{k-1} - x^*\|^2  \leqfrom{\ref{popov half from sol}} 2\|x_0-x^*\|^2
  \end{equation*}
 Defining $B=\set{y\in U:\|y-x^*\|^2\leq 2\|x_0-x^*\|^2}$, we have that all $\xhatk\in B$, ensuring that $\xbark\in B$.
    
    Now, rearranging (\ref{popov monotone predescent}), we have
    \begin{align*}
        \gamma_k\inprd{F(y)}{\xhatk-y}&\leq \frac{1}{2}\p{R_k(y)-R_{k+1}(y)} \\
        \Rightarrow \sum_{k=0}^{K-1} \gamma_k\inprd{F(y)}{\xhat_k-y}&\leq \frac{1}{2}\p{R_0(y)-R_{K}(y)}\\
        &\leq \frac{1}{2}\p{\|x_0-y\|^2 +\|x_0-\xhat_{-1}\|^2} \\
        &=\half\|x_0-y\|^2
    \end{align*}
    Taking the supremum of both sides and applying (\ref{Gap lemma}), we obtain
    \begin{align}
       \frac{1}{2}\supg\|x_0-y\|^2 &\geq  \supg \sum_{k=0}^{K-1} \gamma_k\inprd{F(y)}{\xhatk - y} \notag\\
     &\geqfrom{\ref{Gap lemma}}  \ugamma K \gap{\xbark} 
        \label{popov sup ineq 1}
    \end{align}
    where $\xbark =\frac{ \sum_{k=0}^{K-1}\xhatk\gamma_k}{ \sum_{k=0}^{K-1}\gamma_k}$. We now turn to upper bounding the remaining supremum on the left-hand side of the inequality by using Young's inequality:
    \begin{align*}
        \supg \|x_0-y\|^2 &\leqfrom{\ref{full young}} \supg\p{(1+\beta^{-1})\|x_0-x^*\|^2 + (1+\beta)\|y-x^*\|^2} \\
        &=(1+\beta^{-1})\|x_0-x^*\|^2+ 2(1+\beta)\|x_0-x^*\|^2 \\
        &=(3+\beta^{-1}+2\beta)\|x_0-x^*\|^2
    \end{align*}
    We choose $\beta=\nicefrac{1}{\sqrt{2}}$ to minimize our upper bound, finding that
    \begin{equation}
        \supg\|x_0-y\|^2\leq (3+2\sqrt{2})\|x_0-x^*\|^2\label{popov sup bound}
    \end{equation}
    Substituting into (\ref{popov sup ineq 1}) and dividing by $\ugamma K$, we conclude the proof with
    \begin{align*}
        \gap{\xbark} &\leqfrom{\ref{popov sup ineq 1}} \frac{\supg \|x_0-y\|^2}{2\ugamma K}\\
        &\leqfrom{\ref{popov sup bound}}\frac{(3+2\sqrt{2})\|x_0-x^*\|^2}{2\ugamma K}
    \end{align*}
    \end{proof}

\subsubsection{Proof of Corollary~\ref{cor: popov monotone}}
    \begin{corollary}
        Let $F$ be $1$-symmetric $(L_0,L_1)$-Lipschitz and monotone. In the unconstrained setting ($U=\R^d$), the iterates of \peg with stepsizes given by (\ref{Popov stepsizes}) and $\nu e^\nu=\frac{1}{2\sqrt{2}}$ ($\nu\approx 0.27$) satisfy
         \begin{align*}
            \min_{0\leq k\leq K} \|F(\xhat_{k-1})\|^2 &\leq \frac{2\|x_0-x^*\|^2}{\ugamma^2 (K+1)}
        \end{align*}
        Furthermore, for $K+1>\frac{2 L_1^2}{\nu^2}\exp\p{\frac{2\nu^2 e^\nu}{1-\nu e^\nu}}\|x_0-x^*\|^2$, 
         \begin{equation*}
            \mink \|F(\xhat_{k-1})\|\leq \frac{L_0 \sqrt{2}\exp\p{\frac{\nu^2 e^\nu}{1-\nu e^\nu}}\|x_0-x^*\|}{\nu\sqrt{K+1}-L_1\sqrt{2}\exp\p{\frac{\nu^2 e^\nu}{1-\nu e^\nu}}\|x_0-x^*\|}
        \end{equation*} 
    \end{corollary}
    \begin{proof}
        From (\ref{popov monotone predescent}), we have
        \begin{align*}
            R_{k+1}(x^*) &\leq R_k(x^*)-(1-4\nu^2e^{2\nu})(\|x_k-\xhatk\|^2+\|x_k-\xhat_{k-1}\|^2) -2\gamma_k\inprd{F(x^*)}{\xhatk-x^*} \\
            &= R_k(x^*)-(1-4\nu^2e^{2\nu})(\gamma_k^2\|F(\xhat_{k-1})\|^2+\|x_k-\xhat_{k-1}\|^2)
        \end{align*}
        We choose $\nu e^\nu = \frac{1}{2\sqrt{2}}$ ($\nu\approx 0.27$) to obtain
        \begin{align*}
            R_{k+1}(x^*)&\leq R_k(x^*)-\frac{1}{2}(\gamma_k^2\|F(\xhat_{k-1})\|^2+\|x_k-\xhat_{k-1}\|^2) \\
            &\leq R_k(x^*)-\frac{1}{2}\gamma_k^2\|F(\xhat_{k-1})\|^2
        \end{align*}
        From here, we have that $\gamma_k$ is lower bounded using (\ref{popov step lower bound}). We continue by rearranging and telescoping to obtain
        \begin{align*}
            \sum_{k=0}^K\gamma_k^2\|F(\xhat_{k-1})\|^2 &\leq 2\p{R_0(x^*)-R_{K+1}(x^*)} \\
            &\leq 2R_0(x^*) \\
            &= 2\|x_0-x^*\|^2
        \end{align*}
       Taking the minimum in our summand, we have that
       \begin{equation}
           (K+1)\mink \gamma_k^2\|F(\xhat_{k-1})\|^2 \leq 2\|x_0-x^*\|^2 \label{Popov min gF bound}
       \end{equation}
        
        We can then readily obtain best iterate convergence using $\ugamma\leq\gamma_k$
        \begin{align*}
            \min_{0\leq k\leq K} \|F(\xhat_{k-1})\|^2 &\leq \frac{2\|x_0-x^*\|^2}{\ugamma^2 (K+1)}
        \end{align*}
        giving us our first convergence result. 
    
        Alternatively, we do not use our lower bound for $\gamma_k$. Then we can note that from (\ref{Popov min gF bound})
    \begin{align*}
         \frac{2\|x_0-x^*\|^2}{K+1} &\geq   \mink \gamma_k^2\|F(\xhat_{k-1})\|^2 \\
            &=\mink \frac{\nu^2\|F(\xhat_{k-1})\|^2}{(L_0+L_1\|F(\xhat_{k-1})\|)^2\exp(2L_1\|x_k-\xhat_{k-1}\|)} \\
            &\geqfrom{\ref{popov half step bound}} \mink \frac{\nu^2\|F(\xhat_{k-1})\|^2}{(L_0+L_1\|F(\xhat_{k-1})\|)^2\exp\p{\frac{2\nu^2 e^\nu}{1-\nu e^\nu}}}\\
        \end{align*}
        Taking the square roots of both sides, we obtain
        \begin{equation*}
            \mink \frac{\nu\|F(\xhat_{k-1})\|}{(L_0+L_1\|F(\xhat_{k-1})\|)\exp\p{\frac{\nu^2 e^\nu}{1-\nu e^\nu}}}\leq \frac{\sqrt{2}\|x_0-x^*\|}{\sqrt{K+1}}
        \end{equation*}
        Then there exists some $0\leq k_0\leq K$ where
        \begin{equation*}
            \frac{\|F(\xhat_{k_0-1})\|}{L_0+L_1\|F(\xhat_{k_0-1})\|}\leq \frac{\sqrt{2}\exp\p{\frac{\nu^2 e^\nu}{1-\nu e^\nu}}\|x_0-x^*\|}{\nu\sqrt{K+1}}
        \end{equation*}
        Rearranging, we obtain
        \begin{align*}
            \p{\nu\sqrt{K+1}-L_1\sqrt{2}\exp\p{\frac{\nu^2 e^\nu}{1-\nu e^\nu}}\|x_0-x^*\|}\|F(\xhat_{k_0-1})\|\\
            \leq L_0 \sqrt{2}\exp\p{\frac{\nu^2 e^\nu}{1-\nu e^\nu}}\|x_0-x^*\|
        \end{align*}
        Thus, whenever $K+1>\frac{2 L_1^2}{\nu^2}\exp\p{\frac{2\nu^2 e^\nu}{1-\nu e^\nu}}\|x_0-x^*\|^2$, we can rearrange the terms a final time to find
        \begin{equation*}
            \mink \|F(\xhat_{k-1})\|\leq \frac{L_0 \sqrt{2}\exp\p{\frac{\nu^2 e^\nu}{1-\nu e^\nu}}\|x_0-x^*\|}{\nu\sqrt{K+1}-L_1\sqrt{2}\exp\p{\frac{\nu^2 e^\nu}{1-\nu e^\nu}}\|x_0-x^*\|}
        \end{equation*} 
    \end{proof}

    \subsection{Proofs of Results for Strongly Monotone VIs}
    \subsubsection{Proof of Theorem~\ref{thm: EG sm}}
    \begin{theorem}
        Let $F$ be $\alpha$-symmetric $(L_0, L_1)$-Lipschitz and $\mu$-strongly monotone. Then the iterates of \eg with stepsizes given by (\ref{EG Stepsizes}) and $\nu$ chosen so that $1-2\nu-\nu^2e^{2\nu}\geq0\ifaone$ and $1-\nu-\nu^2\geq0 \ifasmall$ satisfy
        \begin{equation*}
            \|x_k-x^*\|^2\leq (1-\ugamma\mu)^k \|x_0-x^*\|^2
        \end{equation*}
        where $\ugamma$ is defined in Lemma~\ref{lem:EG step lower bound}.
    \end{theorem}
    \begin{proof}
        Using (\ref{EG predescent}) with $y=x^*$, we have
        \begin{align}
            \|x_{k+1}-x^*\|^2 &\leqfrom{\ref{EG predescent}} \p{1-\frac{2\gamma_k\mu}{1+\beta}}\|x_k-x^*\|^2-\p{1-C_\nu^2 -\frac{2\gamma_k\mu}{\beta}}\|x_k-\xhatk\|^2 \notag\\
            &\qquad- 2\gamma_k\inprd{F(x^*)}{\hat{x}_k-x^*}\notag \\
            &\leqfrom{\ref{VIP Solution}} \p{1-\frac{2\gamma_k\mu}{1+\beta}}\|x_k-x^*\|^2-\p{1-C_\nu^2 -\frac{2\gamma_k\mu}{\beta}}\|x_k-\xhatk\|^2\label{EG usm predescent}
        \end{align}

        We now choose $\beta=1$, finding
        \begin{equation}
            \|x_{k+1}-x^*\|^2\leq  \p{1-\gamma_k\mu }\|x_k-x^*\|^2 -(1-2\gamma_k\mu-C_\nu^2)\|x_k-\xhatk\|^2 \label{eg usm predescent beta1}
        \end{equation}
Now note that when $\alpha=1$, we have
        \begin{align}
            \gamma_k\mu &= \frac{\nu\mu}{L_0+L_1\|F(x_k)\|} \notag\\
            &\leqfrom{\ref{mu L0 L1}} \nu \label{eqn: eg usm gmu nu}
        \end{align}
        while when $\alpha \in (0,1)$
        \begin{align*}
            \gamma_k\mu &=\frac{\nu\mu}{2K_0+(2K_1+2^{1-\alpha}K_2^{1-\alpha})\|F(x_k)\|^\alpha} \\ 
            &\leq \frac{\nu\mu}{2(K_0+K_1\|F(x_k)\|^\alpha)} \\
            &\leqfrom{\ref{mu L0 L1}} \frac{\nu}{2}
        \end{align*}
        Using these inequalities in (\ref{EG usm predescent}) our choice of $\nu$ guarantees that $(1-2\gamma_k\mu-C_\nu^2)$ is positive. Using this and noting that $\gamma_k\mu<1$ from our choices of $\nu$, we have 
        \begin{align}
            \|x_{k+1} - x^*\|^2 &\leq (1-\gamma_k\mu)\|x_k-x^*\|^2 \label{eg descent gamma k}\\
            &\leq \|x_k-x^*\|^2 \notag \\
            &\leq \|x_0-x^*\|^2
            \end{align}
   Now that our iterates are bounded, we can apply Lemma~\ref{lem:EG step lower bound} to equation (\ref{eg descent gamma k}), giving us
            \begin{align*}
          \|x_{k+1} - x^*\|^2  &\leqfrom{\ref{EG step lower bound}} (1-\ugamma\mu)\|x_k-x^*\|^2
        \end{align*}
        which we can unroll to obtain our final result.
    \end{proof}

\subsubsection{Proof of Corollary~\ref{cor: eg usm}}
\begin{corollary}
      Let $F$ be $1$-symmetric $(L_0,L_1)$-Lipschitz and $\mu$-strongly monotone, and $U=\R^n$ (unconstrained case). Then the iterates of \eg with stepsizes given by (\ref{EG Stepsizes}) and $\nu$ chosen so that $\nu\approx0.21$ is the solution to the $1-4\nu-2\nu^2e^{2\nu}=0$ satisfy $\|x_{K+1}-x^*\|^2\leq\varepsilon$ after 
     \begin{equation*}
        K= \floor{\frac{8L_1^2}{\nu^2} \|x_0-x^*\|^2 }+ \ceil{\frac{2L_0}{\nu\mu}\log\frac{\|x_0-x^*\|^2}{\eps}}
    \end{equation*}
    iterates.
\end{corollary}
\begin{proof}
    Starting from \eqref{eg usm predescent beta1}, we have
    \begin{align}
        \|x_{k+1}-x^*\|^2 &\leqfrom{\ref{eg usm predescent beta1}}  \p{1-\gamma_k\mu}\|x_k-x^*\|^2-\p{1-C_\nu^2 -2\gamma_k\mu}\|x_k-\xhatk\|^2 \notag\\
        &=  \p{1-\gamma_k\mu}\|x_k-x^*\|^2-\p{1-\nu^2 e^{2\nu} -2\gamma_k\mu}\gamma_k^2\|F(x_k)\|^2  \notag\\
        &\leqfrom{\ref{eqn: eg usm gmu nu}} \p{1-\gamma_k\mu}\|x_k-x^*\|^2-\p{1-\nu^2 e^{2\nu} -2\nu}\gamma_k^2\|F(x_k)\|^2 
        \label{eqn: eg cor usm start}
    \end{align}
    where in the second line we used that $C_\nu=\nu e^\nu$ when $\alpha=1$ and from the definition of \eg, $\|x_k-\xhatk\|=\gamma_k\|F(x_k)\|$ in the unconstrained setting.
    
Furthermore, from our choice of $\nu$ as the solution to $1-4\nu-2\nu^2 e^{2\nu}=0$, we have that
    \begin{equation}
        1-2\nu-\nu^2e^{2\nu}=\frac{1}{2} \label{eqn: eg cor usm nu choice}
    \end{equation}
Substituting this into \eqref{eqn: eg cor usm start}, we obtain
 \begin{align}
     \|x_{k+1}-x^*\|^2 &\leq\p{1-\gamma_k\mu}\|x_k-x^*\|^2-\half\gamma_k^2\|F(x_k)\|^2  \label{eqn: eg descent usm cor}
 \end{align}
We now separate into two cases depending on the magnitude of $\|F(x_k)\|^2$.

    \textbf{Case 1: $\|F(x_k)\| > \nicefrac{L_0}{L_1}$}. In this setting, since $L_1\|F(x_k)\|> L_0$, we have
    \begin{align*}
        \gamma_k \|F(x_k)\|&= \frac{\nu\|F(x_k)\|}{L_0+L_1\|F(x_k)\|} \\
        &> \frac{\nu}{2L_1}
    \end{align*}
  Note that from the second term of (\ref{eqn: eg descent usm cor}), this implies that $\|x_k-x^*\|^2$ decreases by at least $\frac{\nu^2}{8L_1^2}$ at each iteration in this regime. Consequently, at most
\begin{equation}
    K'=\floor{\frac{8L_1^2}{\nu^2} \|x_0-x^*\|^2 }
\end{equation}
iterations can occur with $\|F(x_k)\| > \nicefrac{L_0}{L_1}$.

\textbf{Case 2: $\|F(x_k)\|\leq \nicefrac{L_0}{L_1}$}: In this regime, we have,
\begin{align*}
        \gamma_k & \geq \frac{\nu }{2L_0}
    \end{align*}
    Using the first term of (\ref{eqn: eg descent usm cor}) and accounting for the number of iterations that Case 1 is satisfied, we then obtain that
    \begin{align*}
        \|x_{k+1}-x^*\|^2 \leq \p{1-\frac{\nu\mu}{2L_0}}^{\max\set{k+1-K', 0}}\|x_0-x^*\|^2
    \end{align*}
    Hence, we conclude that $\|x_{K+1}-x^*\|^2\leq \eps$ after at most
    \begin{equation*}
        K= \floor{\frac{8L_1^2}{\nu^2} \|x_0-x^*\|^2 }+ \ceil{\frac{2L_0}{\nu\mu}\log\frac{\|x_0-x^*\|^2}{\eps}}
    \end{equation*}
    iterations, as desired.    
\end{proof}

    \subsubsection{Proof of Theorem~\ref{thm: popov sm}}
    \begin{theorem}
        Let $F$ be $\alpha$-symmetric $(L_0,L_1)$-Lipschitz and $\mu$-strongly monotone. Then the iterates of \peg with stepsizes given by (\ref{Popov stepsizes}) and $\nu$ chosen so that $1-2\nu-4C_\nu^2\geq 0$ satisfy
        \begin{equation*}
           R_k\leq (1-\ugamma\mu)^kR_0
        \end{equation*}
        where $R_k=\|x_k-x^*\|^2+\|x_k-\xhat_{k-1}\|^2$  and $\ugamma$ is defined in Lemma~\ref{lem:popov step lower bound}. 
    \end{theorem}

    \begin{proof}
    From (\ref{popov predescent}), taking $y=x^*$ we have
    \begin{align}
      R_{k+1}  &\leqfrom{\ref{popov predescent}}  \p{1-\frac{2\gamma_k\mu}{1+\beta}}R_k-\p{1-4C_\nu^2-\frac{2\gamma_k\mu}{\beta}}\|x_k-\xhatk\|^2 \notag\\
        &\qquad+\p{4C_\nu^2-1+\frac{2\gamma_k\mu}{1+\beta}}\|x_k-\xhat_{k-1}\|^2  -2\gamma_k\inprd{F(x^*)}{\xhatk-x^*}\notag\\
      &\leqfrom{\ref{VIP Solution}}\p{1-\frac{2\gamma_k\mu}{1+\beta}}R_k(y)-\p{1-4C_\nu^2-\frac{2\gamma_k\mu}{\beta}}\|x_k-\xhatk\|^2 \notag\\
        &\qquad+\p{4C_\nu^2-1+\frac{2\gamma_k\mu}{1+\beta}}\|x_k-\xhat_{k-1}\|^2  \label{eqn: popov usm prebeta descent} \\
      \end{align}
We now set $\beta=1$ and use the stepsize definition to obtain the following:
        \begin{align}
        R_{k+1}&\leq (1-\gamma_k\mu)R_k-(1-2\gamma_k\mu-4C_\nu^2)\|x_k-\xhatk\|^2 \notag\\
      &\qquad+(4C_\nu^2+\gamma_k\mu-1)\|x_k-\xhat_{k-1}\|^2 \notag \\
            &\leqfrom{\ref{mu L0 L1}} (1-\gamma_k\mu)R_k-(1-2\nu-4C_\nu^2)\|x_k-\xhatk\|^2-(1-\nu-4C_\nu^2)\|x_k-\xhat_{k-1}\|^2  \label{popov predescent usm}
        \end{align}
  
    We continue the proof by choosing $\nu$ so that $ 1-\nu-4C_\nu^2\geq 1-2\nu-4C_\nu^2\geq 0$. When $\alpha=1$ this corresponds to $\nu <0.26$, and when $\alpha\in (0,1)$, this is $\nu\leq\frac{-1+\sqrt{5}}{4} \approx 0.31$. Then we have
    \begin{align}
        R_{k+1} \leq (1-\gamma_k\mu)R_k \label{popov descent gammak}
    \end{align}
    Note from our step size definition and (\ref{mu L0 L1}), we have $\gamma_k\mu <1$, indicating that $R_{k+1}\leq R_k\leq R_0$. This allows us to apply Lemma~\ref{lem:popov step lower bound} to equation (\ref{popov descent gammak}), giving us that
    \begin{equation*}
        R_{k+1}\leq (1-\ugamma \mu)R_k
    \end{equation*}
  Unrolling this inequality gives us the final result.
    \end{proof}

    \subsubsection{Proof of Corollary~\ref{cor: popov sm}}
    \begin{corollary}
          Let $F$ be $1$-symmetric $(L_0,L_1)$-Lipschitz and $\mu$-strongly monotone, and $U=\R^n$ (unconstrained case). Then the iterates of \peg with stepsizes given by (\ref{Popov stepsizes}) and $\nu$ chosen so that $\nu\approx0.17$ is the solution to the $1-4\nu-8\nu^2e^{2\nu}=0$ satisfy $R_k\leq\varepsilon$ after 
         \begin{equation*}
              K= \floor{\frac{8L_1^2}{\nu^2} \exp\p{\frac{2\nu^2 e^\nu}{1-\nu e^\nu}} R_0 } +  \ceil{\frac{2L_0}{\nu\mu}\exp\p{\frac{\nu^2 e^\nu}{1-\nu e^\nu}}\log\frac{R_0}{\eps}}
         \end{equation*}
       iterates, where $R_k=\|x_k-x^*\|^2+\|x_k-\xhat_{k-1}\|^2$.
    \end{corollary}
    \begin{proof}
       We take $\nu$ to be the solution to $1-4\nu-8\nu^2e^{2\nu}=0$. Then we have that
          \begin{align*}
              1-2\nu-4\nu^2e^{2\nu} &=\frac{1}{2} \\
              1-\nu-4\nu^2e^{2\nu}&=\frac{1}{2}+\nu
          \end{align*}
        Combining this with (\ref{popov predescent usm}), we have
        \begin{align}
            R_{k+1} &\leq (1-\gamma_k\mu)R_k-\frac{1}{2}\|x_k-\xhatk\|^2-\p{\frac{1}{2}+\nu}\|x_k-\xhat_{k-1}\|^2\notag \\
            &= (1-\gamma_k\mu)R_k-{\frac{\gamma_k^2}{2}\|F(\xhat_{k-1})\|^2}-\p{\frac{1}{2}+\nu}\|x_k-\xhat_{k-1}\|^2  \notag\\
            &\leq (1-\gamma_k\mu)R_k-{\frac{\gamma_k^2}{2}\|F(\xhat_{k-1})\|^2}\label{eqn: popov descent usm cor}
        \end{align}
    Note also since the criteria of Lemma~\ref{lem:popov step lower bound} are satisfied, we can use the intermediate equation~(\ref{popov half step bound}):
    \begin{equation}
         \|x_k-\xhat_{k-1}\|\leq \frac{\nu^2 e^\nu}{L_1(1-\nu e^\nu)} \label{eqn: popov half step lower bound usm cor} 
    \end{equation}
    We now separate our convergence into two cases depending on the magnitude of $\|F(\xhat_{k-1})\|$.
    
    \textbf{Case 1: $\|F(\xhat_{k-1})\| > \nicefrac{L_0}{L_1}$}. In this setting, note that since we have $L_1\|F(\xhat_{k-1}\|>L_0$, we readily obtain
    \begin{align*}
        \gamma_k\|F(\xhat_{k-1})\| &= \frac{\nu\|F(\xhat_{k-1})\|}{(L_0+L_1\|F(\xhat_{k-1})\|)\exp\p{L_1\|x_k-\xhat_{k-1}\|}} \\ 
        &> \frac{\nu}{2L_1\exp(L_1\|x_k-\xhat_{k-1}\|)}
        \end{align*}
    Then using (\ref{eqn: popov half step lower bound usm cor}), we have
    \begin{align*}
        \gamma_k\|F(\xhat_{k-1})\| &\geq \frac{\nu}{2L_1 \exp\p{\frac{\nu^2 e^\nu}{1-\nu e^\nu}}}
    \end{align*}
    Note that from the second term of (\ref{eqn: popov descent usm cor}), this implies that $R_k$ decreases by at least $\frac{\nu^2}{8L_1^2 \exp\p{\frac{2\nu^2 e^\nu}{1-\nu e^\nu}}}$ at each iteration in this regime. Consequently, at most
    \begin{equation}
        K'=\floor{\frac{8L_1^2}{\nu^2} \exp\p{\frac{2\nu^2 e^\nu}{1-\nu e^\nu}} R_0 }
    \end{equation}
    iterations can occur with $\|F(\xhat_{k-1})\| > \nicefrac{L_0}{L_1}$.
    
    \textbf{Case 2: $\|F(\xhat_{k-1})\| \leq \nicefrac{L_0}{L_1}$}. In this regime, we have
    \begin{align*}
        \gamma_k &\geq \frac{\nu}{2L_0\exp(L_1\|x_k-\xhat_{k-1}\|)} \\
        &\geq \frac{\nu}{2L_0\exp\p{\frac{\nu^2 e^\nu}{1-\nu e^\nu}}} 
    \end{align*}
    Using the first term of (\ref{eqn: popov descent usm cor}) and combining this result with our iteration count from Case 1, we then obtain that
    \begin{equation*}
        R_{k+1} \leq R_0\p{1-\frac{\nu\mu}{2L_0\exp\p{\frac{\nu^2 e^\nu}{1-\nu e^\nu}}}}^{\max\set{k+1-K', 0}}
    \end{equation*}
         Thus, we conclude that $R_{K+1}\leq\varepsilon$ after at most 
         \begin{equation*}
             K= \floor{\frac{8L_1^2}{\nu^2} \exp\p{\frac{2\nu^2 e^\nu}{1-\nu e^\nu}} R_0 }+  \ceil{\frac{2L_0}{\nu\mu}\exp\p{\frac{\nu^2 e^\nu}{1-\nu e^\nu}}\log\frac{R_0}{\eps}}
         \end{equation*}
         iterations.
    \end{proof}

    \subsection{Proofs of Line-Search Results and Algorithm Details}
    \label{app: backtrack proofs}
    In this section, we provide details on our line-search method and proofs of the convergence results. The full algorithm is detailed in Algorithm~\ref{alg:AdaptiveAlgo}. The stepsize forms are the same as for our methods without line-search. However, each time the backtracking condition in Line~\ref{alg: while line} (based on (\ref{eqn: backtrack cond})) is violated, we alternate between updating our estimates for $L_0$ and $L_1$ ($\lambda_k^0$ and $\lambda^1_k$) using the update rules in Lines \ref{alg: L0 update} and \ref{alg: L1 update}. These updates allow each Lipschitzness parameter to be updated in a way roughly proportional to its contribution to the stepsize denominator, and have been used both for \eg in the unconstrained $(L_0,L_1)$-Lipschitz setting \citep{yoon2026polyaktypeextragradientmethodsmonotone} and for the Franke-Wolfe algorithm in the $(L_0,L_1)$-smooth setting \citep{vyguzov2026frankwolfealgorithmsl0l1smooth}. 

\begin{algorithm}
    \small
        \caption{\algname{(P)EG-LS} \newline $\tilde{x}_k=x_k$ for \algname{EG-LS} and $\xhat_{k-1}$ for \algname{PEG-LS}; \newline$\gamma(\lambda_k^0,\lambda_k^1)=\frac{\nu}{\lambda_k^0+\lambda_k^1\|F(x_k)\|}$ for \algname{EG-LS}, $\frac{\nu}{(\lambda_k^0+\lambda_k^1\|F(\xhat_{k-1})\|)\exp(\lambda_k^1\|x_k-\xhat_{k-1}\|)}$ for \algname{PEG-LS}}
        \label{alg:AdaptiveAlgo}
        \begin{algorithmic}[1]
            \REQUIRE Initial  point $x_0 \in U$; Initial line search estimates $\lambda_{-1}^0,\lambda_{-1}^1 > 0$; Line search factor $\rho > 2$; Stepsize scaling parameter $\nu>0$
            \STATE $\textrm{toggle} = 0$
            \FOR{$k = 0, 1,...,K-1$}
            \STATE $\lambda_k^0 = \lambda_{k-1}^0$
            \STATE $\lambda_k^1 = \lambda_{k-1}^1$
            \STATE $\gamma_k = \gamma(\lambda_k^0,\lambda_k^1)$
            \STATE $\xhatk = P_U(x_k - \gamma_k F(\tilde{x}_k))$.
            \WHILE{$\gamma_k\|F(\hat{x}_k)-F(\tilde{x}_{k})\| \label{alg: while line}
            > \nu e^\nu \|\hat{x}_k-\tilde{x}_{k}\|$}
                \IF{$\textrm{toggle}=0$}
                    \STATE 
                    $\lambda_k^0=\lambda_k^0\p{\rho-\frac{\lambda_k^0}{\lambda_k^0+\lambda_k^1\|F(\tilde{x}_k)\|}}$ \label{alg: L0 update}
                    \STATE $\textrm{toggle}=1$
                \ELSIF{$\textrm{toggle}=1$}
                    \STATE
                    $\lambda_k^1=\lambda_k^1\p{\rho-\frac{\lambda_k^1\|F(\tilde{x}_k)\|}{\lambda_k^0+\lambda_k^1\|F(\tilde{x}_k)\|}}$  \label{alg: L1 update}
                    \STATE $\textrm{toggle}=0$
                \ENDIF
                \STATE $\gamma_k = \gamma(\lambda_k^0,\lambda_k^1)$
                \STATE $\xhatk=P_U(x_k-\gamma_kF(\tilde{x}_k))$.
            \ENDWHILE
            \STATE $x_{k+1} = P_U(x_k - \gamma_k F(\xhatk))$
            \ENDFOR
        \end{algorithmic}
    \end{algorithm}

 We can readily show that the algorithm will only backtrack a finite number of times.

\begin{lemma}
\label{lem: finite while loops}
Suppose $F$ is 1-symmetric $(L_0,L_1)$-Lipschitz. Then the total number of while-loop calls of Algorithm~\ref{alg:AdaptiveAlgo} with backtracking parameter $\rho>2$ is at most 
    \begin{equation}
    \label{eqn:linesearch-finite-operator-call-bound}
        N=2\max\set{
            0, 
            \ceil{\frac{\log(L_0/\lambda_{-1}^0)}{\log(\rho-1)}},
            \ceil{\frac{\log(L_1/\lambda_{-1}^1)} {\log(\rho-1)}}
        }
    \end{equation}
    
    Consequently, for all accepted steps we have
    \begin{equation}
        \lambda_k^0 \le \lamz
        =
        \rho^{\lceil N/2 \rceil}\lambda_{-1}^0,
        \qquad
        \lambda_k^1 \le \lamo
        =
        \rho^{\lceil N/2 \rceil}\lambda_{-1}^1. \label{eqn: lam max app}
    \end{equation}
    \end{lemma} 
    \begin{proof}
     First, if both $\lambda_k^0\geq L_0$ and $\lambda_k^1 \geq L_1$, then from Lemma~\ref{lem:lipschitz analogue} we have
        \begin{equation*}
            \gamma_k\|F(\xhatk)-F(\tilde{x}_k)\|\leq C_\nu\|\xhatk-\tilde{x}_k\|
        \end{equation*}
        where $\tilde{x}_k=x_k$ for \eg and $\xhat_{k-1}$ for \peg. Note that this is the while loop condition for Algorithm~\ref{alg:AdaptiveAlgo}, so once $\lambda_k^0\geq L_0$ and $\lambda_k^1 \geq L_1$, no additional while loop calls will occur. 
        
       Now, note that at each update step in the while loop, the update factor applied to $\lambda_k^0$ or $\lambda_k^1$ is lower bounded by $\rho-1$ and upper bounded by $\rho$. Thus, after $k$ while-loop calls, the denominator estimates are lower bounded by
        
    \begin{equation*}
       \lambda_k^0\geq  (\rho-1)^{\lfloor k/2 \rfloor}\lambda_{-1}^0,
        \qquad
        \lambda_k^1\geq(\rho-1)^{\lfloor k/2 \rfloor}\lambda_{-1}^1
    \end{equation*}
    
    Consequently, by the definition of $N$, we have that
    \begin{equation*}
        \lambda_N^0\geq (\rho-1)^{\ceil{N/2}}\lambda_{-1}^0\geq L_0,
        \qquad
        \lambda_N^1\geq (\rho-1)^{\ceil{N/2}}\lambda_{-1}^1\geq L_1.
    \end{equation*}
    
    At this point, no further while loop calls will occur, leading to at most $N$ while loop calls over all the algorithm's iterations.  The upper bounds bounds on $\lambda_k^0$ and $\lambda_k^1$ follow immediately since both sequences only increase during while-loop calls. As a final note, these bounds on the number of while loop calls imply that over $K$ iterations of each algorithm, there are at most $2K+N$ operator evaluations for \adaeg and $K+N+1$ operator evaluations for \adapeg. 
    \end{proof}

    We now continue into our convergence proofs, treating \adaeg and \adapeg separately. For each algorithm, we present a lemma lower bounding the stepsizes and then the corollaries with the convergence guarantee for the corresponding setting. 

\subsubsection{\adaeg Convergence Proofs}     \begin{lemma}
 \label{lem: adaeg step lower bound}
  Let $F$ be $1$-symmetric $(L_0,L_1)$-Lipschitz. Then if the iterates of \adaeg satisfy $\|x_{k+1}-x^*\|\leq\|x_0-x^*\|$ for all $k\in 0,1,2,...$ we have 
        \begin{equation}
            \gamma_k\geq \ugamma = 
                \frac{\nu}{\lamz+\lamo\p{ \|x_0-x^*\|(L_0+L_1\|F(x^*)\|)\exp(L_1\|x_0-x^*\|) + \|F(x^*)\|}} 
         \label{adaEG step lower bound}
        \end{equation} 
        where $\lamz$ and $\lamo$ are defined in Lemma~\ref{lem: finite while loops}
        \end{lemma}
    \begin{proof}
    This result readily follows from Lemma~\ref{lem:EG step lower bound}. Specifically, note that from the assumed descent inequality, we can reuse the following intermediate inequality from Lemma~\ref{lem:EG step lower bound}:
        \begin{equation}
            \|F(x_k)\| \leqfrom{\ref{eg f upper bound}} \|x_0-x^*\|(L_0+L_1\|F(x^*)\|)\exp(L_1\|x_0-x^*\|) + \|F(x^*)\|  \label{eqn: eg f upper bound backtrack}
        \end{equation}
            Thus, we have that for \adaeg, the stepsizes will satisfy
    \begin{align*}
        \gamma_k &= \frac{\nu}{\lambda_k^0 + \lambda_k^1\|F(x_k)\|} \\
        &\geqfrom{\ref{eqn: lam max app}} \frac{\nu}{\lamz +\lamo\|F(x_k)\|} \\
        &\geqfrom{\ref{eqn: eg f upper bound backtrack}} \frac{\nu}{\lamz+ \lamo ( \|x_0-x^*\|(L_0+L_1\|F(x^*)\|)\exp(L_1\|x_0-x^*\|) + \|F(x^*)\|)}
    \end{align*}
    \end{proof}
 \begin{corollary}
 \label{cor: adaeg convergence app}
      Let $F$ be $1$-symmetric $(L_0, L_1)$-Lipschitz and monotone. Additionally let $\xbark=\frac{\sum_{k=0}^{K-1}\gamma_k\xhatk}{\sum_{k=0}^{K-1}\gamma_k}$ and $x^*\in U^*$. Then the averaged iterates $\xbark$ of \adaeg with $\nu$ chosen so that $\nu e^\nu< 1$  satisfy
         \begin{equation*}
          \gap{\xbark}\leq \p{2+\frac{1}{\sqrt{1-C_\nu^2}}}^2\frac{\|x_0-x^*\|^2}{2\ugamma K}
        \end{equation*}
     where $\ugamma$ is defined in Lemma~\ref{lem: adaeg step lower bound} and $B=\{y\in U:\|y-x^*\|^2\leq  (1+(1-C_\nu^2)^{-\nicefrac{1}{2}})^2\|x_0-x^*\|^2\}$. Further, $\xbark\in B$.
 \end{corollary}
 \begin{proof}
        Note that in the proofs of Theorems \ref{thm: EG monotone}, the specific stepsize choice for \eg is used in only two locations. The first is in showing that \eqref{EG f bound} holds, and the second is in lower bounding the stepsizes $\gamma_k$. With respect to the first use, Equations~\ref{EG f bound} is the backtracking conditions of \adaeg, guaranteeing that this equation holds at each iterate. With respect to the stepsize lower bounds, we simply use the new lower bound derived in Lemma~\ref{lem: adaeg step lower bound}. Thus, we obtain identical convergence results but with a different value for $\ugamma$.
     \end{proof}

\begin{corollary}
         Let $F$ be $1$-symmetric $(L_0, L_1)$-Lipschitz and $\mu$-strongly monotone. Then the iterates of \adaeg with $\nu$ chosen so that $\nu e^\nu <1$ satisfy
        \begin{equation*}
            \|x_k-x^*\|^2\leq \p{1-\frac{2\ugamma\mu(1-\nu^2 e^{2\nu})}{2\ugamma\mu+1-\nu^2 e^{2\nu}}}^k \|x_0-x^*\|^2
        \end{equation*}
        where $\ugamma$ is defined in Lemma~\ref{lem: adaeg step lower bound}.
\end{corollary}
\begin{proof}
    We start with \eqref{EG usm predescent} when $\alpha=1$,
    \begin{equation*}
        \|x_{k+1}-x^*\|^2\leq \underbrace{\p{1-\frac{2\gamma_k\mu}{1+\beta}}\|x_k-x^*\|^2}_{\text{Term I}}- \underbrace{\p{1-\nu^2 e^{2\nu} -\frac{2\gamma_k\mu}{\beta}}\|x_k-\xhatk\|^2 }_{\text{Term II}}
    \end{equation*}
    for any $\beta>0$.
    Note that this result holds for \adaeg since the only use of the specific stepsize for \eg in acquiring this bound is in applying \eqref{EG f bound}, which is guaranteed for \adaeg by the backtracking condition. 
    While for \eg we chose $\beta=1$, this required us to later use $\gamma_k\mu\leq \nu$ in the proof to eliminate Term II of  \eqref{EG usm predescent}. Since we have no such guarantee that $\gamma_k\mu \leq \nu$ for the line-search algorithm, we instead choose $\beta=\frac{2\gamma_k\mu}{1-C_\nu^2}$ in order to immediately eliminate this term of \eqref{EG usm predescent}, finding
        \begin{equation}
            \|x_{k+1}-x^*\|^2\leq  \p{1-\frac{2\gamma_k\mu(1-\nu^2 e^{2\nu})}{2\gamma_k\mu+1-\nu^2 e^{2\nu}}}\|x_k-x^*\|^2 \label{eg descent gamma k2}
        \end{equation}

        Noting that $\frac{2\gamma_k\mu(1-\nu^2 e^{2\nu})}{2\gamma_k\mu+1-\nu^2 e^{2\nu}}<1$ since $1-\nu^2 e^{2\nu}<1$, we trivially have that our iterates are bounded by $\|x_{k+1}-x^*\|^2\leq \|x_0-x^*\|^2$. Noting also that $\frac{2\gamma_k\mu(1-\nu^2 e^{2\nu})}{2\gamma_k\mu+1-\nu^2 e^{2\nu}}$ is an increasing function with $\gamma_k$, we can apply Lemma~\ref{lem: adaeg step lower bound} to equation (\ref{eg descent gamma k2}), giving us
            \begin{align*}
         \|x_{k+1}-x^*\|^2\leqfrom{\ref{EG step lower bound}}  \p{1-\frac{2\ugamma\mu(1-\nu^2 e^{2\nu})}{2\ugamma\mu+1-\nu^2 e^{2\nu}}}\|x_k-x^*\|^2
        \end{align*}
        which we can unroll to obtain our final result.
\end{proof}

\subsubsection{\adapeg Convergence Proofs}
\begin{lemma}\label{lem: adapeg step lower app}
    Let $F$ be $1$-symmetric $(L_0,L_1)$-Lipschitz. Then if the iterates of \adapeg with $\nu$ chosen so that $\nu e^\nu<1$ satisfy $R_{k+1}\leq R_0$ for all $k\in 0,1,2,...,$ we have
    \begin{equation}
        \shrink{0.95}{\gamma_k\geq\ugamma=\frac{\nu}{(\lamz+\lamo(\sqrt{2}\|x_0-x^*\|(L_0+L_1\|F(x^*)\|)\exp(\sqrt{2}L_1\|x_0-x^*\|) + \|F(x^*)\|))\exp\p{\lamo M}}}
    \end{equation}
    where $R_k=\|x_k-x^*\|^2+\|x_k-\hat{x}_{k-1}\|^2$, $M=\min\set{\|x_0-x^*\|, \frac{\nu^2 e^\nu}{\lambda^1_{-1}(1-\nu e^\nu)}}$, and $\lamz$ and $\lamo$ are defined in Lemma~\ref{lem: finite while loops}.
\end{lemma}
\begin{proof}
    We are largely able to reuse our proof of Lemma~\ref{lem:popov step lower bound}. From our assumption that $R_{k+1}\leq R_0$, we have
    \begin{equation}
        \|x_k-x^*\|^2+\|x_k-\hat{x}_{k-1}\|^2\leq \|x_0-x^*\|^2 \label{adapeg step lower bound 1}
    \end{equation}
 From the same reasoning as in the proof of Lemma~\ref{lem:popov step lower bound}, we have that Equation \ref{popov final f bound a1} applies to \adapeg:
    \begin{equation}
       \|F(\xhat_{k-1})\|\leq \sqrt{2}\|x_0-x^*\|(L_0+L_1\|F(x^*)\|)\exp(\sqrt{2}L_1\|x_0-x^*\|) + \|F(x^*)\| \label{adapeg f bound}
    \end{equation}

    Lastly, we must bound $\|x_k-\xhat_{k-1}\|$. We derive two distinct bounds and take the minima between them. For the first bound, we can use a similar telescoping argument to the proof of Lemma~\ref{lem:popov step lower bound}. Specifically, we again have,
    \begin{align}
        \|x_k-\xhat_{k-1}\| &\leqfrom{\ref{Nonexpansive}} \gamma_{k-1}\|F(\xhat_{k-1})-F(\xhat_{k-2})\| \notag\\
        &\leq \nu e^\nu\|\xhat_{k-1}-\xhat_{k-2}\| \notag\\
        &\leqfrom{\ref{popov half step diff}} \nu e^\nu \p{\|x_{k-1}-\xhat_{k-2}\|+\gamma_{k-1}\|F(\xhat_{k-2})\|} \notag
    \end{align}
    where the second inequality is a result of our backtracking condition. Then since $\lambda_k^1\geq \lambda_{-1}^1$ for all $k$, we have
    \begin{align*}
            \gamma_{k-1}\|F(\xhat_{k-2})\| &= \frac{\nu \|F(\xhat_{k-2})\|}{(\lambda_{k-1}^0+\lambda_{k-1}^1\|F(\xhat_{k-2})\|)\exp(\lambda_{k-1}^1\|x_{k-1}-\xhat_{k-2}\|)} \\
            &\leq \frac{\nu}{\lambda_{-1}^1}
    \end{align*}
    We thus obtain the recursion of 
    \begin{equation*}
        \|x_k-\xhat_{k-1}\|\leq \nu e^\nu \p{\|x_{k-1}-\xhat_{k-2}\|+\frac{\nu}{\lambda_{-1}^1}}
    \end{equation*}
    Unwinding this in the same manner as in the proof of Lemma~\ref{lem:popov step lower bound}, we find
    $$\|x_k-\xhat_{k-1}\|\leq \frac{\nu^2 e^\nu}{\lambda^1_{-1}(1-\nu e^\nu)}$$.

    For the second bound, we use a simple but somewhat more crude result from \eqref{adapeg step lower bound 1} of
    \begin{equation*}
        \|x_k-\hat{x}_{k-1}\|^2\leq \|x_0-x^*\|^2 
    \end{equation*}
     This bound in particular may be tighter when we use a small initial estimate $\lambda_{-1}^1$ for $L_1$.

    Combining the above two bounds, we define $M=\min\set{\|x_0-x^*\|, \frac{\nu^2 e^\nu}{\lambda^1_{-1}(1-\nu e^\nu)}}$ and thus obtain
    \begin{equation}
        \|x_k-\hat{x}_{k-1}\|\le M  \label{adapeg halfstep bound}
    \end{equation}

Putting all of our bounds together, we conclude that for all $k$,
\begin{align*}
    \gamma_k &=\frac{\nu}{({\lambda}_k^0+{\lambda}_k^1\|F(\xhat_{k-1})\|)\exp({\lambda}_k^1\|x_k-\xhat_{k-1}\|)}\\
    &\geqfrom{\ref{eqn: lam max app}}\frac{\nu}{(\lamz+\lamo\|F(\xhat_{k-1})\|)\exp(\lamo\|x_k-\xhat_{k-1}\|)} \\
     &\geqfrom{\ref{adapeg halfstep bound}}\frac{\nu}{(\lamz+\lamo\|F(\xhat_{k-1})\|)\exp\p{\lamo M}} \\
    &\geqfrom{\ref{adapeg f bound}} \frac{\nu}{(\lamz+\lamo(\sqrt{2}\|x_0-x^*\|(L_0+L_1\|F(x^*)\|)\exp(\sqrt{2}L_1\|x_0-x^*\|) + \|F(x^*)\|))\exp\p{\lamo M}}
\end{align*}

\end{proof}

\begin{corollary}
\label{cor: adapeg convergence app}
         Let $F$ be $1$-symmetric $(L_0,L_1)$-Lipschitz and monotone. Additionally, let $\xbark=\frac{\sum_{k=0}^{K-1}\gamma_k\xhatk}{\sum_{k=0}^{K-1}\gamma_k}$ and $x^*\in U^*$. Then the averaged iterates $\xbark$ of \adapeg with $\nu$ chosen so that $\nu e^\nu\leq\frac{1}{2}$ satisfy
        \begin{equation*}
            \gap{\xbark}\leq \frac{(3+2\sqrt{2})\|x_0-x^*\|^2}{2\ugamma K}
        \end{equation*}
       where $\ugamma$ is defined in Lemma~\ref{lem: adapeg step lower app} and $B=\set{y\in U:\|y-x^*\|^2\leq 2\|x_0-x^*\|^2}$. Further, $\xbark\in B$. 
    \end{corollary}
    \begin{proof}
        Note that in the proof of Theorem~\ref{thm: popov monotone}, the specific stepsize choices for \peg are used in only two locations. The first is in showing that Equation~ \ref{popov f bound} hold, and the second is in lower bounding the stepsizes $\gamma_k$. With respect to the first use, Equations \ref{popov f bound} is the backtracking conditions of \adapeg, guaranteeing that this equation holds at each iterate. With respect to the stepsize lower bounds, Lemma~\ref{lem: adapeg step lower app} provides a new lower bound. Thus, we obtain identical results to Theorem~\ref{thm: popov monotone} but with a different stepsize lower bound.
     \end{proof}  

     \begin{theorem}
         Let $F$ be $1$-symmetric $(L_0,L_1)$-Lipschitz and $\mu$-strongly monotone. Then the iterates of \adapeg with $\nu$ chosen so that $\nu e^\nu<\half$ satisfy
        \begin{equation*}
           R_k\leq \p{1-\frac{2\ugamma\mu(1-4\nu^2 e^{2\nu})}{2\ugamma\mu+1-4\nu^2 e^{2\nu}}}^kR_0
        \end{equation*}
        where $R_k=\|x_k-x^*\|^2+\|x_k-\xhat_{k-1}\|^2$  and $\ugamma$ is defined in Lemma~\ref{lem: adapeg step lower app}.
     \end{theorem}
     \begin{proof}
        We begin with \eqref{eqn: popov usm prebeta descent} in the case of $\alpha=1$:
     \begin{align*}
         R_{k+1} &\leq {\p{1-\frac{2\gamma_k\mu}{1+\beta}}R_k(y)}-\p{1-4\nu^2 e^{2\nu}-\frac{2\gamma_k\mu}{\beta}}\|x_k-\xhatk\|^2 \notag\\
        &\qquad+{\p{4\nu^2 e^{2\nu}-1+\frac{2\gamma_k\mu}{1+\beta}}\|x_k-\xhat_{k-1}\|^2}
     \end{align*}
        Note that this result holds for \adapeg since the only use of the specific stepsize for \peg in acquiring this bound is in applying \eqref{popov f bound}, which is guaranteed for \adapeg by the backtracking condition. Similar to our strongly monotone convergence proof for \adaeg, we choose $\beta=\frac{2\gamma_k\mu}{1-4C^2_\nu}$ to eliminate the $\|x_k-\xhatk\|$ term in \eqref{eqn: popov usm prebeta descent}, finding
    \begin{align}
        R_{k+1}\leqfrom{\ref{eqn: popov usm prebeta descent}} \p{1-\frac{2\gamma_k\mu(1-4\nu^2 e^{2\nu})}{2\gamma_k\mu+1-4\nu^2 e^{2\nu}}}R_k + \p{4\nu^2 e^{2\nu} - 1}\p{1-\frac{2\gamma_k\mu}{2\gamma_k\mu+1-4\nu^2 e^{2\nu}}}\|x_k-\xhat_{k-1}\|^2  \label{eqn: popov usm big nu}
    \end{align}
    Then since $\nu e^\nu<\frac{1}{2}$ and $\frac{2\gamma_k\mu}{2\gamma_k\mu+1-4\nu^2 e^{2\nu}}<1$, the second term of \eqref{eqn: popov usm big nu} is negative, allowing us to drop it from the inequality and obtain
    \begin{equation}
        R_{k+1}\leq \p{1-\frac{2\gamma_k\mu(1-4\nu^2 e^{2\nu})}{2\gamma_k\mu+1-4\nu^2 e^{2\nu}}}R_k \label{popov descent gamma k2}
    \end{equation}
    Noting that $\frac{2\gamma_k\mu(1-4\nu^2 e^{2\nu})}{2\gamma_k\mu+1-4\nu^2 e^{2\nu}}<1$ since $1-4\nu^2 e^{2\nu}<1$, we trivially have that $R_{k+1}\leq R_0$ for all $k$. Noting also that $\frac{2\gamma_k\mu(1-4\nu^2 e^{2\nu})}{2\gamma_k\mu+1-4\nu^2 e^{2\nu}}$ is an increasing function with $\gamma_k$, we can apply Lemma~\ref{lem: adapeg step lower app} to (\ref{popov descent gamma k2}), giving us
    \begin{align*}
         R_{k+1}\leqfrom{\ref{popov step lower bound}}\p{1-\frac{2\ugamma\mu(1-4\nu^2 e^{2\nu})}{2\ugamma\mu+1-4\nu^2 e^{2\nu}}}R_k 
    \end{align*}
        Unrolling this inequality gives us our final result.  
     \end{proof}

\newpage
\section{Additional Details on Numerical Experiments}\label{app: experiment section}
In this section, we provide additional details on each of our numerical experiments. Code for the numerical experiments can be found in the supplemental material.

\subsection{Calculation of the Gap Function for Monotone Problems}\label{sec: app gap}
In this subsection, we provide additional details on our calculation of the restricted gap function for our experiments on monotone operators and demonstrate our results experimental hold with a second merit function. 

For both the monotone min-max optimization experiment and the traffic equilibrium experiment, we numerically calculated the value of $\texttt{Gap}_{F,U}{\xbark}$ using the \texttt{scipy.optimize SLSQP} solver.  While providing reasonable results and being consistent with our theoretical findings, we acknowledge that since the $\texttt{Gap}_{F,U}$ function is a nonconvex optimization problem for nonlinear operators, these numerical optimizations may land on local rather than global maxima. Consequently, we provide here experimental results for each of the monotone problems using a second merit function: the \textit{primal gap} function $\gapp$, defined as follows \citep{facchinei_finite-dimensional_2003}:
\begin{equation}
    \gapp(x)=\sup_{y\in U} \inprd{F(x)}{x-y}
\end{equation}
This linear optimization problem over a convex set is simpler to calculate than the restricted gap function, and for monotone operators upper bounds $\texttt{Gap}_{F,U}$ since
\begin{align*}
    \texttt{Gap}_P(x)&=\sup_{y\in U} \inprd{F(x)}{x-y} \\
    &=\sup_{y\in U}\p{\inprd{F(x)-F(y)}{x-y}+ \inprd{F(y)}{x-y}} \\
     &\geqfrom{\ref{muSC}} \sup_{y\in U}\inprd{F(y)}{x-y} \\
     &=\texttt{Gap}_{F,U}(x)
\end{align*}
As can be seen in Figures~\ref{fig:primal traffic} and \ref{fig:primal minimax}, both the traffic equilibrium problem and the monotone min-max problems show the same qualitative behavior in the primal gap function as that shown for the restricted gap function in the main text.

\begin{figure}[htbp]
    \centering
    \begin{minipage}[b]{0.48\textwidth}
        \centering
        \includegraphics[width=\textwidth]{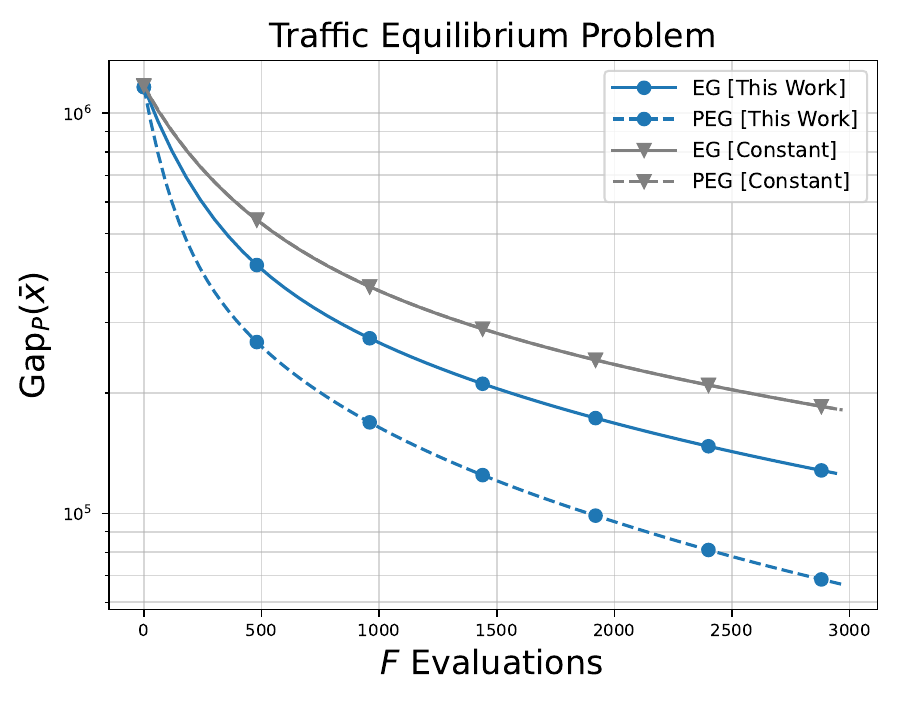}
        \caption{Convergence of the primal gap function for the traffic equilibrium problem experiment.}
        \label{fig:primal traffic}
    \end{minipage}
    \hfill
    \begin{minipage}[b]{0.48\textwidth}
        \centering
        \includegraphics[width=\textwidth]{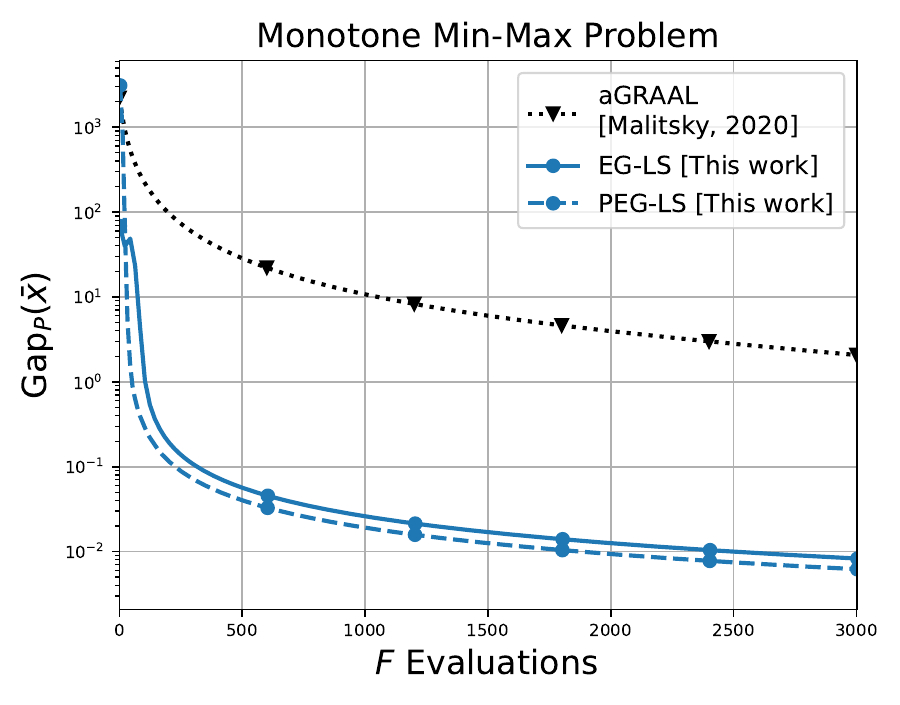}
        \caption{Convergence of the primal gap function for the monotone min-max problem experiment.}
        \label{fig:primal minimax}
    \end{minipage}
\end{figure}

\subsection{Traffic Equilibrium Problem}\label{app:traffic}

In this subsection, we provide a more complete description of the traffic equilibrium problem, the specific traffic network we considered for our experiment, and derive the $(L_0,L_1)$-Lipschitz constants in the case of the BPR cost function.

The traffic equilibrium network problem was first framed as a VI by \cite{dafermos_traffic_1980}. Suppose we have some graph $G(V,E)$, where the vertices represent locations and edges represent routes between each of those locations. We then have $n$ origin-destination (OD) pairs with set``travel demands'' between them. We represent these demands as a single vector $d\in \R^n$. Now say we have $p$ paths connecting these OD pairs, and we define the path-demand matrix $P\in \R^{p\times n}$ as the matrix with entries
$$P_{ij}=\begin{cases}
    1\text{ if path $i$ satisfies OD demand $j$} \\
    0\text{ otherwise}
\end{cases}$$
Letting $f\in\R^p$ be the flows on each of these paths, we require $P^Tf=d$. To calculate the path flow, we define the traffic load on each individual edge as $x\in \R^{|E|}$. Defining the arc-chain incidence matrix as $A\in\R^{|E|\times p}$ such that
$$A_{ij}=\begin{cases}
    1\text{ if edge $i$ is in path $j$}\\ 0\text{ otherwise}
\end{cases}$$
we have that $x=Af$. Lastly, each of these edges has a different cost function $F_i(x)$ that represents how the cost of using the connection scales with the traffic flow on it and possibly several of the other edges. We concatenate these costs into what will become the operator of the VI representing the traffic flow problem, $F(x)$. Then solving the VI
\begin{equation}
    \text{Find $x^*$ such that $\inprd{F(x^*)}{x-x^*}\geq 0$ for all $x$ that satisfy $x=Af, P^Tf=d,f\geq0$} 
\end{equation}
gives us the equilibrium solution for the traffic flow model.
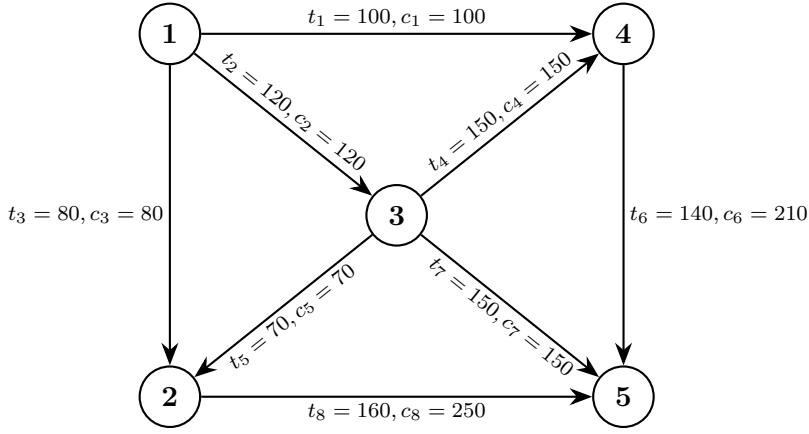
\begin{figure}[htbp]
    \centering
    \begin{tikzpicture}[
        scale=1.2,
        mynode/.style={
            circle,
            draw=black,
            thick,
            minimum size=0.8cm,
            font=\small\bfseries,
            fill=white
        },
        myarrow/.style={
            ->,
            >={Stealth[length=3mm]},
            thick,
            draw=black
        },
        label style/.style={
            font=\scriptsize,
            text=black,
            inner sep=2pt,
            align=center
        }
    ]

    \node[mynode] (1) at (0, 4) {1};
    \node[mynode] (4) at (5, 4) {4};
    \node[mynode] (2) at (0, 0) {2};
    \node[mynode] (5) at (5, 0) {5};
    \node[mynode] (3) at (2.5, 2) {3};

    \draw[myarrow] (1) -- node[label style, sloped, above] {$t_1=100, c_1=100$} (4);
    \draw[myarrow] (1) -- node[label style, sloped, above] {$t_2=120, c_2=120$} (3);
    \draw[myarrow] (1) -- node[label style, left] {$t_3=80, c_3=80$} (2);
    \draw[myarrow] (3) -- node[label style, sloped, above] {$t_4=150, c_4=150$} (4);
    \draw[myarrow] (3) -- node[label style, sloped, below] {$t_5=70, c_5=70$} (2);
    \draw[myarrow] (4) -- node[label style, right] {$t_6=140, c_6=210$} (5);
    \draw[myarrow] (3) -- node[label style, sloped, below] {$t_7=150, c_7=150$} (5);
    \draw[myarrow] (2) -- node[label style, below] {$t_8=160, c_8=250$} (5);

    \end{tikzpicture}
    \caption{Traffic network for traffic equilibrium experiment with BPR cost function parameters.}
    \label{fig:traffic_network_1}
\end{figure}
We followed the standard model where $F_i(x)=F_i(x_i)$. The most commonly used cost function is the BPR cost function \cite{bureau1964traffic}, defined most generally as
\begin{equation}
    F_i(x)=t_i\p{1+a\p{\frac{x_i}{c_i}}^\beta} \label{eqn: bpr app}
\end{equation}
where $t_i$ represents the time to traverse the road in the absence of traffic, $c_i$ is the capacity of the road, and $a$ represents how quickly congestion effects come into effect. While the values of $a=0.15$ and $\beta=4$ reported in the main text are most common, values for $\beta$ from 2 to 12 have also been used in practice \citep{conical_traffic}.

In the standard model, the Jacobian of the operator is a diagonal matrix with entries
\begin{equation}
    \frac{\partial F_i}{\partial x_i}=\frac{t_ia\beta}{c_i}\left(\frac{x_i}{c_i}\right)^{(\beta-1)}
\end{equation}
Since $x_i\geq 0$, we have that the lowest eigenvalue of the Jacobian $\J\geq 0$, making this a monotone problem. Additionally, we have from Equation~\ref{eqn: bpr app} that $\p{\frac{x_i}{c_i}}^\beta=\p{\frac{F_i(x_i)}{t_i}-1}\frac{1}{a}$
\begin{align*}
    \|\J\|&\leq \max_i \frac{t_ia\beta}{c_i}\p{\frac{x_i}{c_i}}^{(\beta-1)} \\
    &= \max_i\frac{t_ia\beta}{c_i} \p{\p{\frac{F_i}{t_i}-1}\frac{1}{a}}^{\nicefrac{\beta-1}\beta}\\
    &=\max_i \frac{\beta a^{1/\beta} t_i}{c_i}\p{\frac{F_i(x)}{t_i}-1}^{\nicefrac{(\beta-1)}{\beta}}\\
&\leq \max_i \frac{\beta\, a^{1/\beta}\, t_i}{c_i}\p{\frac{F_i(x)}{t_i}}^{\nicefrac{(\beta-1)}{\beta}}\\
&=\max_i \frac{\beta a^{1/\beta} t_i^{1/\beta}}{c_i} F_i(x)^{\nicefrac{(\beta-1)}{\beta}}\\
&\leq \beta a^{1/\beta}\max_i\p{\frac{t_i^{1/\beta}}{c_i}} \|F(x)\|^{\nicefrac{(\beta-1)}{\beta}}
\end{align*}
making this a $\frac{\beta-1}{\beta}$-symmetric $(0, \beta a^{1/\beta}\max_i\p{\nicefrac{t_i^{1/\beta}}{c_i}})$-Lipschitz operator. For this paper, we specifically used the originally proposed values of $a=0.15,\beta=4$ \citep{bureau1964traffic}.

We used a specific traffic network inspired by \cite{traffic_problem_source} with parameters for the BPR cost function shown in Figure~\ref{fig:traffic_network_1}. Overall, this network with our choice of $a$ and $\beta$ makes $F$ $\frac{3}{4}$-symmetric $(0,0.1)$-Lipschitz. The only OD-pair is $(1,5)$ with a demand of $1000$. We use this constraint to determine the $L$-Lipschitz constant as $\|J\|\leq \max_i \frac{t_ia\beta}{c_i}\p{\frac{x_i}{c_i}}^{(\beta-1)}\leq \max_i \frac{t_ia\beta}{c_i}\p{\frac{1000}{c_i}}^{(\beta-1)}\approx1750=L$. We used these theoretical values for $L$, $\alpha$, and $L_1$, and we increased $L_0$ to $10^{-8}$ to ensure numerical stability. Lastly, we set $\nu$ to the maximum admissible values given by Theorems~\ref{thm: EG monotone} and \ref{thm: popov monotone} ($1$ for \eg and $\half$ for \peg).

\subsection{Strongly Monotone Problem}
In this subsection, we prove the strong monotonicity of $F(x)=\p{x_1^2,x_2^2}$ when the constraint set $U$ is the disk centered at $\p{\half,\half}$ with radius  $\frac{1}{4}$ and find the solution to the VI. We consider the more general case of  $x$ constrained to a disk centered at $(c,c)$, $c>0$ with radius $r$, $0<r<c$.  The Jacobian of this operator is
\begin{equation*}
    \J=\begin{pmatrix}
        2x_1 & 0\\ 0 &2x_2
    \end{pmatrix}
\end{equation*}
which has as its lowest eigenvalue $2\min\set{x_1,x_2}$. Due to our constraint set, $x_1$ and $x_2$ are never below $c-r>0$, making this operator $2(c-r)$-strongly monotone on $U$. Similarly, the largest eigenvalue of $\J$ is $2\max\set{x_1,x_2}$, which is upper bounded by $2(c+r)$. Consequently, on $U$ the operator is $L$-Lipschitz with Lipschitz constant $2(c+r)$, though in $\R^d$, there is no finite Lipschitz constant \citep{choudhury2025extragradient}. Then since $F$ is strongly monotone on $U$, the VI has a unique solution, and using a geometric argument, we can readily see that $x^*=\begin{pmatrix}
    c-\frac{r\sqrt{2}}{2}, c-\frac{r\sqrt{2}}{2}
\end{pmatrix}^\top$. Specifically, note that $x^*$ is on the boundary of $U$, and $F(x^*)$ is directed towards the center of the disk, which is the inward normal direction. Thus, we have $\inprd{F(x^*)}{x-x^*}\geq0$ for all $x\in U$ (see Figure~\ref{fig: solution_geo}).

\begin{figure}[h]
    \centering
    \includegraphics[width=0.3\linewidth]{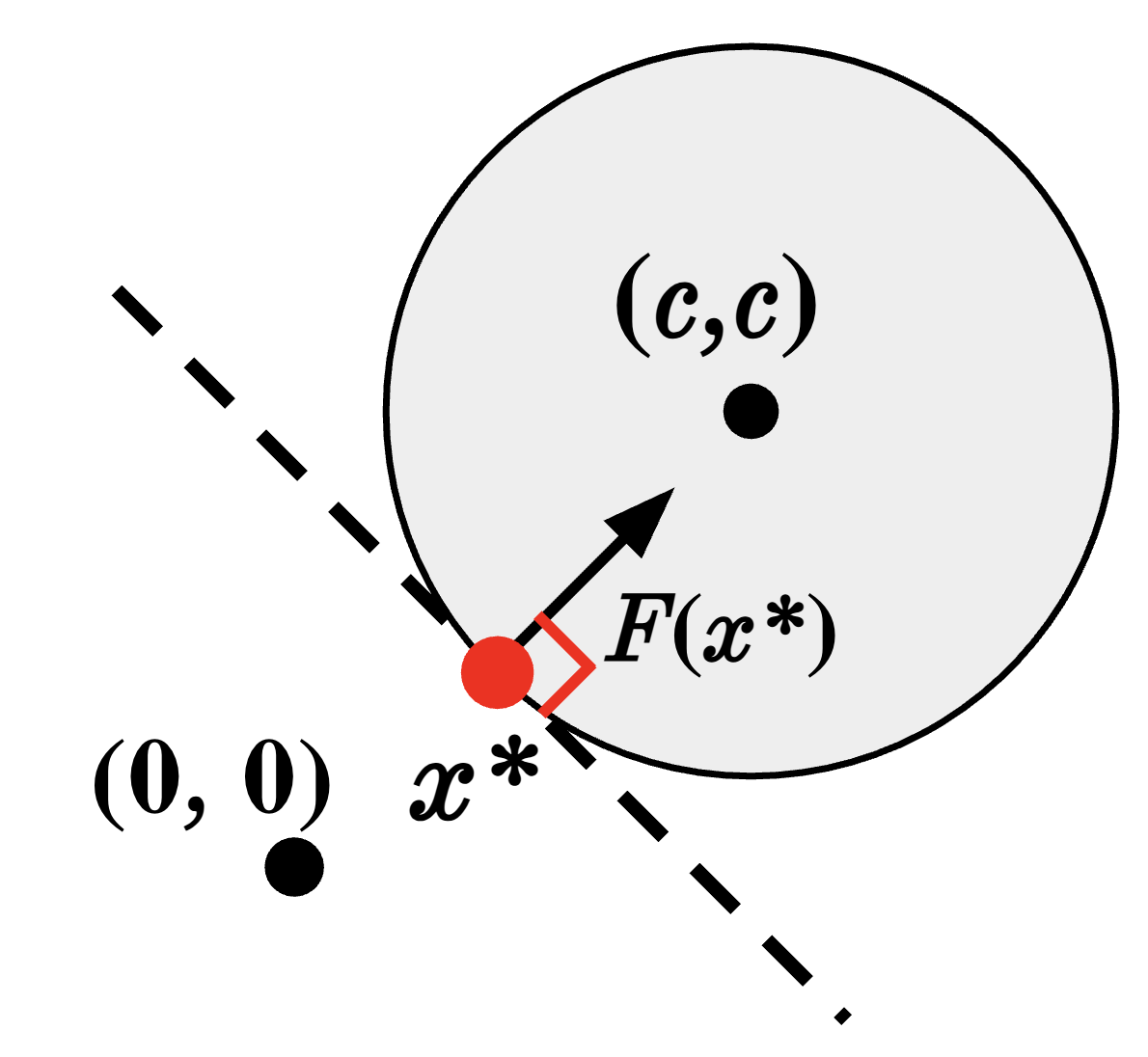}
    \caption{Solution to Strongly Monotone VI with $\alpha\in (0,1)$}
    \label{fig: solution_geo}
\end{figure}

For our experiments, we set $\nu$ near the maximally admissible values according to Theorems~\ref{thm: EG sm} and \ref{thm: popov sm} ($0.61$ for \eg and $0.30$ for \peg).

\subsection{Monotone min-max Problem}\label{app:minimax}
In this subsection, we provide additional details on the monotone min-max problem for our numerical experiment. 

\subsubsection{VI Properties}
From Example 3 of \cite{choudhury2025extragradient}, we have that the min-max problem
\begin{equation}
    \min_{w_1 \in U_1} \max_{w_2 \in U_2}  \frac{1}{4} \|w_1\|^4+ w_1^{\top} \B w_2 - \frac{1}{4}\|w_2\|^4
\end{equation}
is $1$-symmetric $\left(2\tau_0 + \|\textbf{M}\|, 2^{\frac{17}{9}} \tau_1 \right)$-Lipschitz,
for any choice of $\tau_1 > 0$ and $\tau_0 =  \frac{4}{\tau_1^2}, \textbf{M} = \begin{bmatrix}
            0 & \B \\
            -\B^\top & 0
        \end{bmatrix}$. The specific operator for the problem is
        \begin{equation*}
            F(x)=\begin{pmatrix}
                \|w_1\|^2 w_1 + \B w_2 \\
                \|w_2\|^2w_2 - \B^\top w_1
            \end{pmatrix}
        \end{equation*}
which has corresponding Jacobian
\begin{equation*}
  \J=  \begin{pmatrix}
        {\|w_1\|^2I + 2{w_1 w_1^\top}} & \B \\
        -\B^\top &   {\|w_2\|^2I + 2{w_2 w_2^\top}}
    \end{pmatrix}
\end{equation*}

Noting that the off-diagonal is skew-symmetric, we have $\J\succeq \min_{(w_1,w_2)\in U_1\times U_2}\set{\|w_1\|^2, \|w_2\|^2}$, providing us with the strong monotonicity parameter $\mu$. 

Taking $U_1$ and $U_2$ to be balls of radius $r$ in $\R^{50}$ centered at the ones vectors, when $r=\sqrt{50}$, $w_1=w_2=0$ is in $U_1\times U_2$, leading to $\mu=0$ (a monotone problem). However, when $r=3$, we have that $\|w_1\|^2$ and $\|w_2\|^2$ are lower bounded by $(\sqrt{50}-3)^2\approx16.57$, making the problem $16.57$-strongly monotone.

\vspace{-3mm}
\subsubsection{Hyperparameter Tuning}
For \adaeg and \adapeg, we set the initial Lipschitz estimates to the arbitrary but small value of $\lambda_{-1}^0=\lambda_{-1}^1=0.1$. We then set the initial stepsize parameter for the \algname{EG Backtracking} method from \cite{pmlr-v235-vankov24a} so that the first attempted $\gamma_k$ is the same as that from \adaeg. 

For the line-search scaling parameter, \algname{EG Backtracking} \citep{pmlr-v235-vankov24a} uses a parameter $q\in(0,1)$. We ran their algorithm for each $q\in\set{0.1, 0.2,..., 0.9}$, and Figure~\ref{fig:minimax} shows the $q$ which required resulted in the minimum error after $100$ operator evaluations.

For \adaeg and \adapeg, we tested $9$ values for the line-search scaling parameter $\rho$ log uniformly spread between $2.01$ (close to our minimum permitted value of $2$) and $100$. As we did with \cite{pmlr-v235-vankov24a}'s method, Figure~\ref{fig:minimax} in the main text shows the $\rho$ which resulted in the minimum error after $100$ operator evaluations for the strongly monotone problem and after $3000$ evaluations for the monotone problem. We set $\nu$ near the maximal admissible values according to Corollaries~\ref{cor: adaeg convergence app} and \ref{cor: adapeg convergence app} ($0.567$ for \adaeg and $0.3517$ for \adapeg).

For \algname{aGRAAL}, the primary hyperparameter that determines the algorithmic performance is the initial stepsize $\lambda_0$. We follow the recommendation of \cite{malitsky_golden_2020} by making a small perturbation $x'_0$ of the initial point $x_0$ and taking $\lambda_0=\frac{\|x_0'-x_0\|}{\|F(x_0')-F(x)\|}$. For our specific instantiation, we take $x_0'=P_U(0.95x_0)$. 

In Figures~\ref{fig:monotone tuning} and \ref{fig:sm tuning} below, we plot the full results of our hyperparameter tuning. As is seen in Figure~\ref{fig:monotone tuning}, in the monotone setting, \pegls consistently outperforms \algname{aGRAAL} with the recommended initial stepsize choice from~\cite{malitsky_golden_2020}. In Figure~\ref{fig:sm tuning}, we see that \pegls with $\rho<10$ has comparable performance to \algname{EG Backtracking} \citep{pmlr-v235-vankov24a}  across its full range of $q$ values. At larger $\rho$, the performance of \pegls deteriorates somewhat. For practitioners, we would generally recommend $\rho<10$ so that the line-search can more finely explore different estimates of $L_0$ and $L_1$, and this recommendation is supported by our experiments.

\begin{figure}
    \centering
    \includegraphics[width=0.4\linewidth]{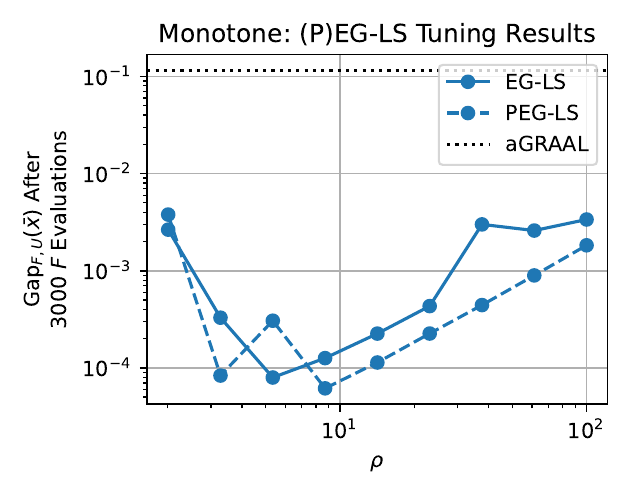}
    \caption{In Figure \ref{fig:monotone tuning}, we show the final values of $\texttt{Gap}_{F,U}(\xbark)$ from our tuning of the line-search constant $\rho$ for the monotone experiment with \adaeg and \adapeg.}
    \label{fig:monotone tuning}
\end{figure}

\begin{figure}
    \centering
    \includegraphics[width=0.7\linewidth]{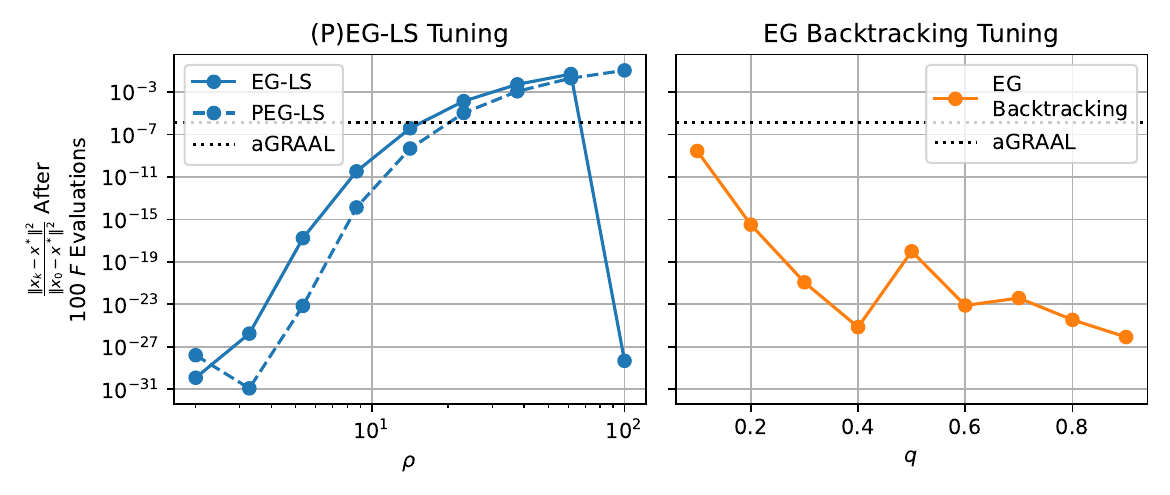}
    \caption{In Figure \ref{fig:sm tuning}, we show the final values of $\frac{\|x_K-x^*\|^2}{\|x_0-x^*\|^2}$ from our tuning of the line-search constant $\rho$ for \adaeg and \adapeg and $q$ for \algname{EG Backtracking} \citep{pmlr-v235-vankov24a} for our strongly monotone experiment.}
    \label{fig:sm tuning}
\end{figure}

\textbf{Determination of $x^*$}
For the determination of $x^*$ for the strongly monotone convergence measure, we ran each of \adaeg, \adapeg, and \cite{pmlr-v235-vankov24a}'s \eg method with backtracking for $1000$ iterations. At this point, the vector norm of the pairwise differences between $x^*_{\adaeg}$, $x^*_\adapeg$, and $x^*_{\algname{EG\,Backtracking}}$ agree to within $1.2*10^{-15}$, indicating the algorithms have fully converged to the same unique solution.

\end{document}